\documentclass[11pt]{article}
\usepackage{amsmath,amssymb}
\usepackage[left=1in,right=1in,top=1.1in,bottom=1in]{geometry}
\usepackage{bm}
\usepackage{expl3}

\usepackage{tikz}
\usepackage{appendix}
\usepackage{subfigure}
\usepackage[english]{babel}
\usepackage{bbm}
\usepackage{graphicx}
\usepackage[center]{caption2}
\usepackage{amsfonts,amssymb,amsmath,latexsym,amsthm}
\usepackage{multirow}
\usepackage{enumerate}
\usepackage{geometry}
\usepackage[usenames,dvipsnames]{pstricks}
\usepackage{epsfig}
\usepackage{pst-grad} % For gradients
\usepackage{pst-plot} % For axes
\usepackage[space]{grffile} % For spaces in paths
\usepackage{enumitem}% delete small point before items
\usepackage{etoolbox} % For spaces in paths
\makeatletter % For spaces in paths
\patchcmd\Gread@eps{\@inputcheck#1 }{\@inputcheck"#1"\relax}{}{}

\newtheorem{thm}{Theorem}[section]
\newtheorem{cor}[thm]{Corollary}
\newtheorem{lem}[thm]{Lemma}

\newtheorem{prop}[thm]{Proposition}

\newtheorem{claim}[thm]{Claim}

\newtheorem{property}[thm]{Property}

\newtheoremstyle{plainupright}
  {\topsep}   % Space above
  {\topsep}   % Space below
  {\upshape}  % Body font
  {}          % Indent amount
  {\bfseries} % Theorem head font
  {.}         % Punctuation after theorem head
  { }         % Space after theorem head
  {}          % Theorem head spec (can be left empty, meaning 'norma\lan')
\theoremstyle{plainupright}
\newtheorem{example}{Example}

\newtheorem{dfn}{Definition}[section]

\let\svthefootnote\thefootnote
\newcommand\blankfootnote[1]{%
	\let\thefootnote\relax\footnotetext{#1}%
	\let\thefootnote\svthefootnote%
}

\def\endpf{\hfill $\Box$}

\newcommand{\xhdr}[1]{\paragraph{\bf #1.}}

\newcommand{\omt}[1]{}

\newcommand\supproof[1]{}

\def\coll{{\mathcal{X}}}
\def\seen{S}
\def\fint{t^*}
\def\zt{t^+}

\def\trueL{{K}}

\def\out{o}

\newcommand{\genlang}{\Lambda}
\newcommand{\termlang}{\Lambda'}

\newcommand{\md}{{\mathcal{D}}}
\newcommand{\mY}{{\mathcal{Y}}}
\newcommand{\mX}{{\mathcal{X}}}
\newcommand{\ma}{{\mathcal{A}}}
\newcommand{\mC}{{\mathfrak{C}}}

\newcommand{\mx}{{\mathfrak{x}}}
\newcommand{\my}{{\mathfrak{y}}}
\newcommand{\mz}{{\mathfrak{z}}}

\newcommand{\mB}{{\mathcal{B}}}
\newcommand{\Acc}{\mathcal{A}_{\text{acc}}}
\newcommand{\dor}{\mu_{\text{order}}}

\newcommand{\dup}{\mu_{\text{up}}}
\newcommand{\dlow}{\mu_{\text{low}}}
\newcommand{\mS}{{\mathcal{S}}}
\newcommand{\tim}{{t_{i-1}}}
\newcommand{\ti}{{t_{i}}}
\newcommand{\tti}{{\theta(t_{i})}}

\newcommand{\tb}{{t_b}}
\newcommand{\ts}{{t^*}}
\newcommand{\ths}{{\theta^*}}
\newcommand{\Fs}{{F^*}}
\newcommand{\Succ}{\text{succ}}

\newcommand{\laj}{\mathsf{J}}

\newcommand{\lan}{\mathsf{L}}
\newcommand{\lanj}[1]{\mathsf{J}_{#1}}

\newcommand{\lang}[1]{\mathsf{L}_{#1}}

\newcommand{\ga}[1]{\Gamma_{#1}}
\newcommand{\fbga}[1]{\widetilde{\Gamma}_{#1}}

\newcommand{\res}[2]{{#1}_{#2}}

\ExplSyntaxOn

\ExplSyntaxOff

\begin{document}

\pagenumbering{gobble}
	
\title{Density Measures for Language Generation}
	
\date{\today}

% \author{Authors anonymized for submission version}
 \author{Jon Kleinberg\thanks{Department of Computer Science and Information Science, Cornell University, Ithaca NY 14853 USA.  Supported in part by a Vannevar Bush Faculty Fellowship, AFOSR grant FA9550-23-1-0410, a Simons Collaboration grant, and a grant from the MacArthur Foundation.} \and Fan Wei\thanks{Department of Mathemaics, Duke University, 120 Science Drive, Durham, NC 27710, USA. Research supported by NSF grants DMS-2404167 and DMS-2401414. } }

\maketitle

\begin{abstract}
The recent successes of large language models (LLMs) have led to a surge of theoretical research into the properties of language generation. A recent line of work has proposed an abstract view of the question --- called \emph{language generation in the limit} --- in which we view language generation as a game played between an adversary and an algorithm: the adversary generates strings from an unknown language $K$, known only to come from a countable collection of candidate languages, and after observing a finite set of these strings, the algorithm must generate new strings from the language $K$ that it hasn't seen before. This formalism highlights an important tension: the trade-off between \emph{validity} (that the algorithm should only produce strings that come from the language) and \emph{breadth} (that the algorithm should be able to produce ``many'' strings from the language). This validity-breadth trade-off is a central issue in applied work on language generation as well, where it arises in the balance between \emph{hallucination}, when models generate invalid utterances, and \emph{mode collapse}, when models only generate from a very restricted set of feasible outputs. Despite its importance, this trade-off has been challenging to study quantitatively. 

In this work we develop ways of quantifying this trade-off, by formalizing the notion of breadth through measures of \emph{density}. Roughly speaking, the density of one language $L$ in another language $\lan'$ is the limiting fraction of strings from $\lan$ among the strings of $\lan'$, where we take the limit over longer and longer finite prefixes of $\lan'$. Existing algorithms for language generation in the limit produce output sets that can have zero density in the true language $K$, in this asymptotic sense, and this represents an important failure of breadth that might seem necessary in any solution to the problem. We show here that such a failure is not in fact necessary: we provide an algorithm for language generation in the limit whose outputs have strictly positive density in the true language $K$. We also study the internal representations built by algorithms for this problem --- the sequence of hypothesized candidate languages they iterate through as they perform generation --- showing a precise sense in which the strongest form of breadth achievable is one that may need to ``oscillate'' indefinitely between hypothesized representations of high density and low density. Our analysis introduces a novel topology on language families, with notions of convergence and limit points in this topology playing a key role in the analysis. 
\end{abstract}

\newpage

\tableofcontents

\newpage

\pagenumbering{arabic}
\setcounter{page}{1}

\section{Introduction}

%%%%%%%%%%%%%%%%%%%%%%%%%%%%%%%
%%%%%% Start of introduction
%%%%%%%%%%%%%%%%%%%%%%%%%%%%%%%

The rapidly expanding power of generative AI 
has made {\em language generation} a ubiquitous
primitive in current computing applications:
given a large training corpus of text, we can build a model
capable of consistently producing valid, new pieces of text that
it has never seen in its training data.
In this sense, the underlying language that the model is generating 
from is never really ``known'' to the model,
since it is provided only indirectly through the 
examples in the training data.
The fact that large language models (LLMs) can nevertheless
generate such fluent utterances is one of the key
conceptual mysteries at the heart of modern methods for language generation.

Recent theoretical work has sought to understand this mystery better
by exploring it at several different levels of abstraction. At a
fine-grained level, recent work has analyzed the computational power
of the transformer architecture that current language models are built on \cite{peng-transformer,sanford-transformer,wang-transformer,weiss-transformer}.
At a more abstract level, researchers have also sought to uncover
the core mathematical structure that underlies language generation,
treating it as a well-defined computational problem in itself,
independent of specific models or algorithms \cite{charikar-pabbarju,kalai2023calibrated,kalavasis-stoc25,KM24,li-generation-pac}.
This second approach aims to
reveal the fundamental mathematical principles governing effective
language generation, abstracting away from implementation details.

Our work falls within this second framework, adopting a basic,
assumption-free approach to study language generation. 
This assumption-free framework has
its roots in work by Gold and Angluin that began roughly 60 years ago
\cite{angluin1979finding,angluin1980inductive,gold1967language};
they were guided by a different language learning problem, but
the basic set-up shares common ingredients.
In the Gold-Angluin model,
we start with an underlying countably infinite ground set $U =
\{u_1, u_2, u_3, \ldots\}$ consisting of all possible strings.
For our purposes, a language $\lan$ is simply a subset of $U$. 
(All languages considered in this paper are infinite.) 
An adversary chooses a target language $K$ that is known only to come
from a countable collection of candidate languages
$\coll = \{\lang{1}, \lang{2}, \dots\}$, and 
the adversary then enumerates the strings of $K$ one by one,
producing string $w_t \in K$ at time step $t$.

The goal in the Gold-Angluin model is for an algorithm to observe
the enumeration of $K$ and guess which language in $\coll$ the adversary
is enumerating.
In particular, in step $t$, the algorithm guesses an index $i_t$
with the goal that $\lang{i_t} = K$.
The algorithm achieves {\em language identification in the limit}
if its guess is always correct after some finite time;
that is, if there is some time $t^*$ such that 
$\lang{i_t} = K$ for all $t \geq t^*$.
In this way, we can think of language identification as a game
played between the adversary --- who enumerates the string of $K$ ---
and an algorithm --- who guesses the identity of the language ---
in alternating turns.
The basic results for the Gold-Angluin model are negative:
identification in the limit is provably impossible for all but very
restricted collections of languages $\coll$
\cite{angluin1979finding,angluin1980inductive,gold1967language}.

The problem we are considering, however, is language {\it generation} 
rather than language identification: the algorithm does not need to
discover the identity of the language, but only to generate new strings from it.
We therefore keep the structure of the game between an adversary and an 
algorithm, but change the rules of the game,
following a model of Kleinberg and Mullainathan \cite{KM24} 
(henceforth, the {\em KM model}).
As before, the adversary chooses a target language $K$ known only
to come from a countable set of candidates 
$\coll = \{\lang{1}, \lang{2}, \dots\}$, and
it enumerates the strings of $K$.
But now the algorithm has a different goal:
at time $t$, having seen the strings $\seen_t$ generated so far, 
the algorithm must generate a string $\out_t$ that comes 
from $\trueL - \seen_t$: that is, from 
true language $\trueL$ but not among the set $\seen_t$ already
enumerated. 
The algorithm wins the game---it {\em generates in the
limit} from $\trueL$---if after some finite time $\fint$, all of its
output strings $\out_t$ come from $K - \seen_t$.

The main result for the KM model is a positive one that contrasts
sharply with the negative results for language identification:
there is an algorithm (henceforth, the {\em KM algorithm})
that can achieve generation in the limit for {\em every}
countable collection $\coll$ of candidate languages. This is a
surprising result, given that we make essentially no assumptions about
the structure of the languages, other than that they come from 
a countable list of candidates $\coll$.
It establishes an abstract sense in which the problem of language generation
is very different from --- and more tractable than --- the problem
of language identification.

\subsection{The trade-off between breadth and validity in language generation}
\label{subsec:intro:tradeoff}

The success of language generation in practice has obscured a deep 
question that is challenging to address, 
both within the theoretical models of the process
and also within the empirical work on deployed systems.
The question is this:
\begin{quote}
{\em When a given algorithm succeeds at language generation --- 
consistently producing valid new utterances based on training data ---
how ``much'' of the language is it actually capable of generating?}
\end{quote}
Intuitively, this is a question about the {\em breadth} of what
is being generated: is a given generation algorithm
capable of producing a broad subset of the language,
or is it succeeding by staying carefully within a narrow
subset of the language?

Our goal in this paper is to study such breadth questions theoretically within
the KM model, where recent worked has begun to demonstrate some of their rich structure \cite{charikar-pabbarju,kalavasis-stoc25,KM24}.  
(Below in Section \ref{subsec:intro:related} we draw contrasts between our approach and the approaches in these recent papers.)
We also note that 
in an empirical sense, it is clear that these same types of breadth question
also lurk in the background of discussions about how ChatGPT,
Claude, Gemini, and other LLMs risk creating a narrow
homogeneity of writing styles ---
a version of what generative AI researchers refer to as 
{\em mode collapse}, a failure of breadth
in which the set of outputs being generated
collapses into a small fraction of the space of all possible outputs
\cite{arjovsky2017towards,arjovsky2017wasserstein}.
% It has been observed that such phenomenon could appear when training 
If we can find more concrete formal definitions of breadth, 
we may be able to eventually start quantifying 
these issues more sharply in deployed LLMs as well.

As we describe in more detail below, we will formulate our questions in
terms of the {\em density} of one language in another:
roughly speaking, we determine the density of a language $\lan$ in another
language $\lan'$ by looking at the fraction of the first $N$ strings in
$\lan'$ that belong to $\lan$, and then taking a limit as $N$ goes to infinity.
In particular, if the set of outputs of a generation algorithm has
high density in the true language $K$, then the algorithm is generating
with high breadth.
In this way, density gives our question of breadth a quantitative dimension:
breadth in generation is not all-or-nothing, but instead 
we can talk about generation algorithms whose outputs have higher or 
lower density in the true language $K$.

Since our goal is to find generation algorithms whose output sets have
high density, 
it's useful to start by observing the sense in which 
the KM algorithm \cite{KM24} naturally tends toward low-density output sets.
We will have more to say about this in the next few sections, but
at a high level, we can describe the KM algorithm  as follows.
The algorithm maintains at each time $t$ a potentially infinite
descending chain of languages nested by inclusion ---
$\lang{c_1} \supseteq \lang{c_2} \supseteq \lang{c_3} \supseteq \cdots$ ---
with the property that (i) each of these languages is {\em consistent}
at time $t$ (the sample $S_t \subseteq K$ is a subset of each of them),
and (ii) after some finite time, the true language $\trueL = \lang{z}$
must be one of these languages.
At time $t$,
the KM algorithm outputs a string from the language $\lang{c_i}$ in
this chain where $c_i$ is maximal subject to $c_i \leq t$.
Since each $\lang{c_i}$ is consistent, and $t$ grows unboundedly,
it follows that we will eventually have $c_i \geq z$,
at which point $\lang{c_i} \subseteq \lang{z} = K$.

Notice that in this way, the KM algorithm isn't just achieving 
element-based generation --- producing output strings $\out_t$ in 
each step $t$ --- but also what we might call 
{\em index-based generation in the limit}:
in each step $t$ it identifies a language $\lang{i_t}$ with the
property that after some finite time $\fint$, we have
$\lang{i_t} \subseteq \trueL$.
(Once $\lang{i_t}$ is a subset of $\trueL$, any string the algorithm
produces from $\lang{i_t}$ must come from $\trueL$, thus satisfying
our original element-based guarantee.)
However, the algorithm is forever uncertain about when the language it
identifies really becomes a subset of $K$ (and in a sense, this
is unavoidable given the negative results of Gold and Angluin for 
identification in the limit).
Therefore it perpetually moves
further and further down its descending chain of languages
$\lang{c_1} \supseteq \lang{c_2} \supseteq \lang{c_3} \supseteq \cdots$,
each step shrinking the set of strings it could generate. This process
can continue indefinitely. In the end, even though we know the algorithm
will eventually reach languages for which $\lang{i_t} \subseteq \trueL$,
the set of strings it produces 
may dwindle to an infinitesimally small fraction of the true
language $K$. It's very easy to construct examples where the density of
the algorithm's output in $K$ is zero in an asymptotic sense.

Even though this description is specialized to the KM algorithm,
there is something intuitively natural about it:
For a language generation algorithm in general, there is a
tension between two goals:
(i) {\em validity}, to guarantee that the algorithm
generates from $K$ in the limit, and
(ii) {\em breadth}, to guarantee that the algorithm
generates a ``large'' amount of $K$, as formalized in our case
by the notion of density.
Both of these ideas are fundamental to language generation both in its theoretical formulations and in deployed applications; they appear under a variety of different names.
Earlier we discussed {\em mode collapse} as a failure of breadth \cite{arjovsky2017towards,arjovsky2017wasserstein};
we note that failures of validity are also a central issue, and they arise in discussions of the tendency of language models to {\em halluncinate} and produce utterances that don't accord with their specifications
\cite{kalai2023calibrated,xu2024hallucination}.

The direct tension between these two ideas --- validity and breadth --- feels fundamental as well:
if we must guarantee validity, it seems natural that we need, in some
sense, to successively shrink the set of strings being considered, so
that we can sure that they are contained inside $K$ eventually.
This naturalness also makes it feel intuitive to wonder whether any
solution to generation in the limit must necessarily tend toward
output sets of zero density for some instances. A main question left
open by \cite{KM24} is whether this is indeed necessary in any
solution to generation in the limit. 

We now pose this question precisely, and then describe our resolution of it.

\subsection{Formalizing breadth in language generation using density}
\label{sec:intro:density}

We now give the precise definition of density that we use
for formalizing breadth in language generation; 
from this definition, we can then state our main questions more precisely.

If $\lan$ and $\lan'$ are two languages in $\coll$, and the elements of $\lan'$ listed in order as
$\lan' = \{\ell_1', \ell_2', \ell_3', \ldots\}$, 
we say that the {\em upper density} of $\lan$ in $\lan'$ is
   \[ \dup(\lan, \lan') = \limsup_{N \to \infty}  \frac{\mid \lan \cap \{\ell_1', \dots, \ell_N'\} \mid}{N}.\]
Analogously, we say that the {\em lower density} of $\lan$ in $\lan'$ is
   \[ \dlow(\lan, \lan') = \liminf_{N \to \infty}  \frac{\mid \lan \cap \{\ell_1', \dots, \ell_N'\} \mid}{N}.\]

It is necessary to have both definitions:
the upper and lower densities can differ dramatically if, for example, 
$\lan$ consists
of increasingly long intervals that alternate between containing
many elements of $\lan'$ and then very few elements of $\lan'$.
It is not hard to create examples like this where 
the upper density can be arbitrarily close to one while 
the lower density is arbitrarily close to zero.

\xhdr{The density of outputs}
With these definitions, studying the 
density of the algorithm's outputs 
in the true language $\trueL$ lets us 
formalize our quantitative notion of breadth.
Specifically let $E = \{w_1, w_2, w_3, \dots\}$ be 
the sequence of strings enumerated in order by the adversary
in an instance of the KM model, and let $\ma$ be any
algorithm that generates from $\trueL$ in the limit.
If we run the interaction of the adversary and the algorithm forever (in each time step the adversary
outputs a string $w_t$, and $\ma $ outputs an string $o_t$ in reply,
and then we iterate over time steps), we get an infinite set $O(E,\ma)
=\{o_1, o_2, \dots\}$ consisting of all strings in $\trueL$
that the algorithm $\ma$ ever 
outputs.\footnote{Note that there is a 
slightly larger set $O'(E,\ma)$ consisting
of {\em all} the strings that $\ma$ ever generates;
we have $O(E,\ma) = O'(E,\ma) \cap \trueL$, and 
since $\ma$ achieves generation in the limit, we
know that $O'(E,\ma)$ only contains at most finitely many
elements that are not also in $O(E,\ma)$.}

Two key quantities are now the lower density
and the upper density respectively of $O(E,\ma)$ in $\trueL$:
that is, $\dlow(O(E,\ma),K)$ and $\dup(O(E,\ma),K)$.
These are both ways of measuring the breadth of the 
algorithm's output, with the lower density serving as
the more demanding measure since it will always be less
than or equal to the upper density.

We now ask:

\noindent\textbf{Question ($\ast$)}
\vspace{-0.1in}
\begin{quote}
    {\em ($\ast$) Is there an absolute constant $c > 0$
and an algorithm $\ma$ that achieves generation in the
limit for every instance,
and for which the lower density of $O(E,\ma)$ in $\trueL$ is always
at least $c$?}
 \end{quote}
(There is also a weaker version of this question in which 
``lower density'' is replaced by the easier-to-achieve ``upper density''.)

We observe that this question makes precise two key points from
our earlier discussion.
First, Question ($\ast$) formalizes a particular trade-off
between validity (that the algorithm achieves generation in the limit)
and breadth (that the algorithm generates a positive fraction of
the strings in $\trueL$, in the sense of lower density).
And second, our earlier observations about the KM algorithm can be 
summarized by saying that there are instances where 
its output set $O(E, \ma)$ (with $\ma$ equal to the KM algorithm) can have upper density 0, and therefore
also lower density 0, in $\trueL$.

We can therefore think of Question ($\ast$) as also asking 
whether the validity-breadth trade-off suggested by the properties of the
KM algorithm is fundamentally necessary:

\noindent\textbf{Question ($\ast \ast$)}
\vspace{-0.1in}
\begin{quote}
{\em 
% ($\ast \ast$)
Is zero lower density (or even zero upper density) 
unavoidable on at least some instances for any algorithm that
achieves generation in the limit?}
\end{quote}

We will survey our results in detail later in the introduction, but
to preview one of the main results now:
we show that zero density is not inevitable.
In particular, we answer Question ($\ast$) in the affirmative,
designing an algorithm $\ma$ whose output set $O(E,\ma)$ always
has lower density at least $c$, for a fixed absolute constant $c > 0$.

\xhdr{Density for index-based generation}
Recall that the KM algorithm achieves a stronger type of guarantee
that we refer to as {\em index-based generation}:
in each step $t$, it guesses a language $\lang{i_t} \in \coll$
in such a way that $\lang{i_t} \subseteq \trueL$ for all $t$
beyond some finite step $\fint$.

We can use the notions of upper and lower density to
reason about breadth for index-based generation as well.
In particular, it is easy to construct instances for which
the upper densities (in $\trueL$) of the languages $\lang{i_t}$ 
chosen by the KM algorithm converge to 0 as $t$ increases;
this is a formal sense in which 
the algorithm achieves generation by explicitly seeking
out increasingly ``thin'' subsets of $\trueL$, and thereby generating
new elements of $\trueL$ forever, but from more and more restricted
parts of it.

We can therefore ask whether this type of trade-off for
index-based generation is necessary as well.

\noindent\textbf{Question ($\ast\!\ast\!\ast$)}
\vspace{-0.1in}
\begin{quote}
{\em 
Is there an algorithm for index-based generation in the limit for which 
the upper (or lower) densities of the languages $\lang{i_t}$
do not converge to 0 as $t$ increases?
More strongly, 
is there an algorithm for index-based generation in the limit for which 
the upper (or lower) densities of the languages $\lang{i_t}$
remain uniformly bounded away from 0 as $t$ increases?
}
\end{quote}

As we will see, the answer to this question is subtle, in that it
is neither fully positive nor fully negative.
We can show that there is an algorithm for index-based generation in the limit 
where $\lang{i_t}$ achieves lower density 1 in $\trueL$ infinitely often,
and therefore the sequence of lower densities (and hence upper densities)
does not converge to 0 in $t$.
On the other hand, we can also show that there are instances where
for any algorithm achieving index-based generation in the limit, the
languages $\lang{i_t}$ it generates must contain an infinite 
{\em subsequence} whose upper densities converge to 0.
Therefore, the best densities achievable for index-based generation require --- on at least some instances --- a kind of oscillation between high densities
and low densities over the steps of the adversary-algorithm interaction.

\subsection{Further Related Work}
\label{subsec:intro:related}

Exploring the question of validity-breadth trade-offs
via the KM model has been the subject of other recent papers;
they pursue different formulations, and our results in fact form
an interesting contrast with theirs.
In particular, Charikar and Pabbaraju \cite{charikar-pabbarju}
ask whether there could be
an algorithm that achieves generation in the limit, with the property
that if we allow it to generate a countable list of strings
after seeing a finite set $S_t \subseteq \trueL$, it is then able to
generate all the remaining strings in $K - \seen_t$.
Kalavasis et al \cite{kalavasis-stoc25}
ask similar questions, and others,
in a different model where
the adversary produces its enumerated strings by first specifying
a probability distribution over the elements of $\trueL$.
In both cases, the papers obtain remarkable negative results. The conclusion from their work is that this very strong notion
of breadth is too much to require in general.

In contrast to these strong negative results, our results establish that
if we ask for slightly less --- an approximate version
as quantified by the upper and lower
density measures, in which it suffices to generate a positive fraction of the language --- then in fact the trade-off between validity
and breadth isn't as severe as the structure of the KM algorithm might initially make it seem:
an approximate middle ground with both generation in the limit
and positive lower density is possible.

\section{Overview of Results}
\label{sec:overview}

In this section, we give an overview of the main results of the paper,
before moving on to their proofs in the subsequent sections.

\subsection{Formal Definition of Generation: Element-Based and Index-Based}

In order to state the results, we begin by briefly recapping the 
KM model and the two types of guarantees for generation in the limit ---
element-based and index-based --- that we use.

\begin{dfn}[KM Model]\label{def:KM}
As before, the KM model consists of a game played between
and adversary and an algorithm.
At the outset, adversary chooses a true language $\trueL$ from 
a countable list of options
$\coll = \{\lang{1}, \lang{2}, \dots\}$,
where each language $\lang{i} \in \coll$ 
is an infinite subset of the ground set $U$.\footnote{As a 
concrete example, we could imagine that the collection
$\coll$ consists of all languages from some well-known language
family, such as all regular languages or all context-free languages,
but the model is in fact general enough 
to handle arbitrary collections $\coll$.}
We imagine the collection being specified by a black box that
answers {\em membership queries} for the algorithm:
given a string $w$ and an index $i$, it will report whether
or not $w \in \lang{i}$.
The adversary chooses a language $\trueL \in \coll$, but
does not reveal the identity of $\trueL$ to the algorithm, and 
it then enumerates the strings of $\trueL$ one by one.
We say that this constitutes an {\em instance} of the generation problem:
a countable collection $\coll$ of candidate languages, together
with an enumeration of one of these languages $\trueL \in \coll$.
Now, let $w_t$ be the string that the adversary enumerates in step $t$,
and let $\seen_t = \{w_1, \dots, w_t\}$ be the finite set of all
strings that the adversary has enumerated up
through time step $t$.
In each step, the algorithm must output a string $\out_t$, with
the goal that hopefully it should come from the true language $\trueL$.
\end{dfn}
\begin{dfn}[Validity]
We say that the algorithm achieves
{\em generation in the limit} 
if there is some time $t^*$ after which the algorithm always 
produces a string from $\trueL$ ---
that is, if $\out_t \in \trueL$ for all $t \geq t^*$.
\end{dfn}
Here $t^*$ is allowed to depend on the adversary enumeration order. 

As mentioned earlier, the KM model is closely related to the 
60-year-old Gold-Angluin model. 
% we also have a game between an  adversary and an algorithm in which the adversary enmerates the strings of a language $\trueL$ chosen from $\coll$;
The only difference is that the algorithm in the Gold-Angluin model
must guess the identity of the language, by outputting an index $i_t$
in each step with the goal that $\lang{i_t} = \trueL$.
% We say that the algorithm achieves {\em identification in the limit}  if there is some time $t^*$ after which the algorithm is always correct --- that is, if $\lang{i_t} = \trueL$ for all $t \geq t^*$.

\xhdr{Element-Based and Index-Based Generation}
Earlier, we also noted that the KM algorithm 
has the following additional structure that is of interest in its own right:
at all times $t$, it tries to maintain a {\em subset} of
the true language $\trueL$, since if it generates a string
$\out_t$ from a subset of $K$, then $\out_t$ must belong to $K$ as well.
More precisely, at every step $t$, 
the KM algorithm produces an index $i_t$,
and it achieves the guarantee that after some time $\fint$,
we have $\lang{i_t} \subseteq \trueL$ for all $t \geq \fint$.
Then it simply generates a string $\out_t \in \lang{i_t} - \seen_t$,
and this string satisfies the requirement that $\out_t \in K - \seen_t$.

We therefore have two related but distinct notions of
generation in the limit:
\begin{itemize}
\item {\em Element-based generation:} 
This is our original notion, in 
which the goal is to produce a string 
$\out_t$ such that $\out_t \in K - \seen_t$ for all $t$ after some finite time
$\fint$.
\item {\em Index-based generation:}
This is the further guarantee that the KM algorithm also achieves,
in which the goal is to produce an index $i_t$ for a
language such that $\lang{i_t} \subseteq K$ for all $t$ after 
some finite time $\fint$.
Index-based generation directly implies element-based generation.
\end{itemize}

Both of these notions of generation
are naturally motivated from the operation of
LLMs in practice: from the user-facing side these models engage in 
element-based generation; but since they operate by maintaining
an internal representation that continually evolves in the face
of training samples, they can be viewed on the back end
as performing a kind of
index-based generation in which the index $i_t$ corresponds to
the algorithm's choice of representation at time $t$.

As a final point in setting up the model, our notion of an algorithm in this paper is any function that takes the collection $\coll$ and the strings seen so far $\seen_t$ and produces a string $\out_t$ (in the case of element-based generation) or an index $i_t$ of a language in $\coll$ (in the case of index-based generation).
This follows the tradition of past work on both the Gold-Angluin model as well as the KM model of generation and its recent follow-ups, where the main distinctions between positive and negative results rely on the algorithm's lack of knowledge of the true language $\trueL$, rather than on issues of computability.
(In the terminology of the KM model, this means that we are studying {\em generation in the limit via a function} \cite{KM24}, where we need to design a function $f_\coll(\seen_t)$ that produces either a string $\out_t$ or an index $i_t$.)

\subsection{Overview of Results: Index-Based Generation}

We now survey the main results of the paper, first for
index-based generation, and then for element-based generation.

\xhdr{Accurate generation}
In the context of index-based generation, there is a natural
first question to ask that doesn't directly involve quantitative bounds 
on density, but helps in reasoning about them.
Specifically, as the KM algorithm iterates through language indices $i_t$,
eventually producing only languages $\lang{i_t}$
that are subsets of $\trueL$ (once $t \geq \fint$), we can ask about the time steps in which
it has the language exactly right: that is, when has it chosen an
index $i_t$ for which $\lang{i_t} = \trueL$?\footnote{For this 
it would be sufficient to have $i_t = z$, but given that
the same language can occur multiple times in the list 
$\coll$ of candidate languages, we might have some index $z' \neq z$
for which $\lang{z'} = \lang{z} = \trueL$.
Going forward, we will make the assumption that $\lang{z}$ represents
the first occurrence of $\trueL$ in $\coll$; that is,
$\lang{i} \neq \trueL$ for all $i < z$.}
We will say that the algorithm is {\em accurate} in a given
step $t$ if $\lang{i_t} = \trueL$.

Now, it is not hard to verify that on many basic instances
of language generation (specified by a collection $\coll$
of candidate languages and an adversary enumerating one of these
language $\trueL = \lang{z}$), the KM algorithm is accurate
in only finitely many steps.
Would it be possible to design an algorithm that achieves
generation in the limit but is accurate in an infinite sequence of steps?
If we start from the KM algorithm --- which
as described in Section \ref{subsec:intro:tradeoff}
involves moving down
an infinite descending chain of languages 
$\lang{c_1} \supseteq \lang{c_2} \supseteq \lang{c_3} \supseteq \cdots$ that eventually contains the true language $\trueL = \lang{z}$ ---
this would seem to require some controlled way of moving back
``up'' the chain toward $\trueL = \lang{z}$ infinitely often.
But if we move back up the chain, how to do we make sure we don't
move too far back infinitely often, repeatedly choosing a proper
superset of $\trueL$ and thus failing to achieve generation in the limit?

In fact, however, we show that it is possible to design an algorithm
that moves back to exactly $\trueL$ infinitely often, 
while also achieving generation in the limit:
\begin{thm}\label{thm:acc-intro}
There is an algorithm that achieves 
index-based generation in the limit and is also accurate
in an infinite sequence of time steps.
\end{thm}

\xhdr{Density guarantees for index-based generation}
Now, let's consider how our density measures can be applied
to analyze algorithms for index-based generation in the limit.
A first natural question is the following.
Suppose an algorithm generates indices $i_1, i_2, i_3, \ldots$
with the property that for some $t \geq \fint$, we have 
$\lang{i_t} \subseteq \trueL$.
What can we say about the sequence of upper densities
of these languages in $\trueL$ --- that is, the sequence 
$d_1, d_2, d_3, \ldots$ where $d_t = \dup(\lang{i_t},\trueL)$?

We first observe that Theorem \ref{thm:acc-intro} has the following
immediate corollary: since infinitely often we have a language
$\lang{i_t}$ with upper and lower density equal to 1 in $\trueL$
(because $\lang{i_t}$ is $\trueL$), we have
\begin{cor}\label{cor:index-limsup-intro}
There is an algorithm that achieves index-based generation in the limit
and has 
$\displaystyle{\limsup_{t \rightarrow \infty} d_t = 1}$
for every instance.\footnote{Just to make sure the definitions
are kept straight, it is important to note
that when we talk about $\displaystyle{\limsup_{t \rightarrow \infty} d_t}$,
we are talking about a $\limsup$ of a sequence of numbers
(the $d_t$ values)
each of which are themselves the $\limsup$ of a sequence of numbers
(the fractions of the first $N$ elements of $\trueL$ that also
belong to $\lang{i_t}$, over all $N$).}
\end{cor}

If the goal is to ensure non-trivial breadth in generation, however,
then we would also like something more: we
would like the sequence $d_t$ to remain uniformly bounded away from
0 beyond some finite time.
This is a requirement on the $\liminf$ rather than the $\limsup$;
we want 
$\displaystyle{\liminf_{t \rightarrow \infty} d_t > 0}$.
It turns out that is not achievable in every instance, since we can
show the following:

\begin{thm}\label{thm:vanishing-instance-intro}
There exist instances of the generation problem where {\it any}
algorithm achieving index-based generation in the limit must 
satisfy 
$\displaystyle{\liminf_{t \rightarrow \infty} d_t = 0}$.
\end{thm}

Taking Corollary \ref{cor:index-limsup-intro} and
Theorem \ref{thm:vanishing-instance-intro} together,
it says that there are instances where
the best an algorithm for index-based generation in the limit
can do is to produce languges $\lang{i_t}$ whose upper densities
oscillate infinitely often between 1 and arbitrary small values.
This is a surprising consequence of this combination of results,
since it says that solutions which swing between these extremes
are in fact necessary for index-based generation that seeks
to maximize breadth.

\xhdr{A characterization of vanishing breadth}
Let us say that the instances described by 
Theorem \ref{thm:vanishing-instance-intro} have
{\em vanishing breadth} (in an index-based sense):
any algorithm achieving index-based generation in the limit must
satisfy
$\displaystyle{\liminf_{t \rightarrow \infty} d_t = 0}$.

Not every instance has vanishing breadth; some are much more tractable,
and allow for positive values of 
$\displaystyle{\liminf_{t \rightarrow \infty} d_t}$.
It is therefore interesting to ask if we can characterize 
the instances that have vanishing breadth, and it turns out that we can,
using a concept --- an {\em infinite perfect tower of languages} ---
that will play an important role in much of our analysis.

For a given instance of language generation, with a collection $\coll$
of candidate languages,
an infinite perfect tower is a sequence of languages
that ``converges,'' in a sense we will make precise next,
to a terminal language that is a proper
superset of all of them.  
More concretely, let
$\genlang_1, \genlang_2, \genlang_3, \ldots $
be a countable sequence of languages, each from $\coll$,
and let $\termlang$ be a language in $\coll$ such that
$\genlang_j \subsetneq \termlang$ for all $j$.
We say that $\genlang_k$ {\em fixes} a string $w \in \termlang$
if $w \in \genlang_i$ for all $i \geq k$, and $k$ is the minimum
index that has this property for $w$.
(That is, either $w \not\in \genlang_{k-1}$, or $k = 1$.)
Let $B_k$ denote the set of all strings that are fixed by $\genlang_k$.

Now we can write,
\begin{dfn}[infinite perfect tower]\label{def:perfecttower-intro}
We say that a sequence of languages 
$\genlang_1, \genlang_2, \genlang_3, \ldots $, each from $\coll$,
forms an {\em infinite perfect tower} with respect to a language
$\termlang \in \coll$ if 
\begin{itemize}
\item[(i)] $\genlang_j \subsetneq \termlang$ for all $j \geq 1$.
\item[(ii)] $\genlang_j$ fixes at least one string of $\termlang$, 
for all $j \geq 1$.
\item[(iii)] Every string of $\termlang$ is fixed by some 
language $\genlang_j$ in the sequence.
\end{itemize}
\end{dfn}
We call $\termlang$ the {\it terminal language} of this infinite perfect tower $(\genlang_1, \genlang_2, \dots)$. 
Intuitively, as we iterate through the languages 
$\genlang_1, \genlang_2, \genlang_3, \ldots $, every
string of $\termlang$ eventually appears and after some point stays forever;
and the sequence is non-redundant in that each language $\genlang_j$
represents the moment of permanent appearance for some string of $\termlang$.

An infinite perfect tower in which all the languages
$\genlang_j$ have small upper density in $\termlang$ is a useful object,
since we can show it allows an adversary to take the true language
to be $\trueL = \termlang$ and force any algorithm achieving 
index-based generation in the limit to guess languages of small
upper density infinitely often.
And conversely, we can show that if the algorithm from 
Theorem \ref{thm:acc-intro} guesses languages of small upper density
in $\trueL$ infinitely often, then we can extract from its guesses an
infinite perfect tower $\genlang_1, \genlang_2, \genlang_3, \ldots $
with respect to $\trueL$ in which all the languages
$\genlang_j$ have small upper density in $\trueL$.
When filled in, these arguments can be used to establish the following:

\begin{thm} \label{thm:vanishing-char-intro}
An instance of the generation problem, specified by
a collection $\coll$ of candidate languages and a true language $\trueL$,
has vanishing breadth (in an index-based sense) if and only if
for any $\epsilon > 0$, there is an infinite perfect tower 
$\genlang_1, \genlang_2, \genlang_3, \ldots $ in $\coll$ with
respect to $\trueL$ in which each $\genlang_j$ satisfies
$\dup(\genlang_j, \trueL) \leq \epsilon$.
\end{thm}

Furthermore, in Theorem \ref{thm:2} later in the paper, we generalize
this theorem to characterize, for a given instance of the generation problem,
the highest value of 
$\displaystyle{\liminf_{t \rightarrow \infty} d_t}$ that any
algorithm for index-based generation in the limit can achieve.

\subsection{Overview of Results: Element-Based Generation}
We now give an overview of our results on density as a measure
of breadth for element-based generation.
These results include the affirmative resolution of
Question ($\ast$) that we previewed early in this section.

Element-based generation differs from index-based generation in 
several important ways.
First, for an algorithm to achieve element-based generation in the limit,
there is nothing requiring it to name the language $\lang{i_t}$ that
it is generating from at time $t$;
all it needs to do is to output a string $\out_t$.
Second, the element-based structure of the output allows us to
ask different kinds of questions about density, even
for algorithms that are achieving element-based generation as
consequence of index-based generation.
In particular, consider the following contrast.
Our earlier results on index-based generation tracked the algorithm's
``internal state'' via the sequence of indices $i_t$ for the languages
it was generating from; as we saw there, the density of
$\lang{i_t}$ in $\trueL$ might oscillate wildly across
different time steps $t$.
For element-based generation, on the other hand, 
the most natural questions don't track the fluctuations in the 
algorithm's internal state, but instead study the properties
of a single countable set: $O(E,\ma)$, which 
we recall is the full set of 
strings in $\trueL$
generated by the algorithm $\ma$ over all the time steps
in its interaction with an adversary who is enumerating the sequence
of strings $E$.

The set $O(E,\ma)$ is a natural object to study because, by definition,
it constitutes the set of all possible 
outputs that a user of the algorithm would be able to see over
the course of the algorithm's execution.
Asking about the density of $O(E,\ma)$ in $\trueL$
is therefore a way to ask directly about the algorithms's breadth of outputs
as perceived by an entity interacting with it,
separately from any internal representations (via the indices $i_t$)
that it might be using to generate these outputs.
This different way of asking the question thus gives an algorithm
a chance to succed at element-based generation in ways that
can be obscured by guarantees for index-based generation:
the fact that an algorithm spends a lot of time generating
from languages $\lang{i_t}$ that have small upper density in $\trueL$
doesn't actually rule out the possibility that 
$O(E,\ma)$ could still in principle have large lower density in $\trueL$.

\xhdr{The density of $O(E,\ma)$}
This contrast in guarantees between index-based and element-based generation is in fact what we find.
We first provide an algorithm with a positive lower bound on
the upper density of $O(E,\ma)$,
and then an algorithm with a positive lower bound on
the lower density of $O(E,\ma)$ (the latter being a harder
guarantee to achieve, since the lower density can in general
be arbitrarily smaller than the upper density).

As a first observation, we note that even if the collection
$\coll$ consisted only of the single language $\trueL$ --- so that
there was no uncertainty about which language the adversary was
enumerating from --- we should still not expect $O(E,\ma)$ to have
lower or upper density above $1/2$.
To see why this is, 
consider an adversary that always enumerates the 
element of $\trueL$ with the lowest index among all 
strings in $\trueL$ that neither it nor the algorithm has produced yet.
Then among the first $N$ elements of $\trueL$ in order, for any $N$, 
the adversary will produce at
least half of the available strings of $\trueL$ before the algorithm does,
and so the density of the algorithm's outputs can be at most $1/2$.
And the algorithm can essentially achieve $1/2$ in this simple case
just by always generating the lowest-indexed string in $\trueL - \seen_t$.
So $1/2$ is a natural absolute upper bound on what we can expect
for the lower or upper density of $O(E,\ma)$ in $\trueL$.

Our first result matches this for upper density on arbitrary instances:

\begin{thm} \label{thm:element-upper-intro}
For every $c < 1/2$, 
there is an algorithm $\ma$ that achieves element-based
generation in the limit on every instance,
and has $\dup(O(E,\ma),\trueL) \geq c$.
\end{thm}

It turns out to be much more challenging to establish a lower bound
on the lower density of $O(E,\ma)$ in $\trueL$, and the following
theorem is the result in the paper that requires the most involved 
analysis:

\begin{thm} \label{thm:element-lower-intro}
There exists an absolute constant $c > 0$
and an algorithm $\ma$ that achieves element-based
generation in the limit on every instance,
and has $\dlow(O(E,\ma),\trueL) \geq c$.
\end{thm}

We are able to establish $c = 1/8$ as a concrete value
for the constant in Theorem \ref{thm:element-lower-intro},
but there is no indication that this is tight and our proof is not intended to optimize the value of $c$.

\paragraph{Proof idea and Topological point view.}

To give a brief sense for how the proof of Theorem \ref{thm:element-lower-intro} goes, we start by recalling the intuition for Theorem \ref{thm:vanishing-char-intro}, in which the adversary can cause difficulties for an algorithm by following the sequence of languages in an infinite perfect tower with respect to some terminal language $\termlang$ (where the infinite perfect tower consists entirely of languages that have low upper density in $\termlang$). But the adversary can do more, and this is where the arguments become more involved. In particular, for some instances it is possible to construct an infinite sequence of infinite perfect towers whose terminal languages themselves form an infinite perfect tower (for an illustration of a variant of this construction, see Figure \ref{fig:tree2} in Section \ref{sec:element:lower}). We can then iterate a construction like this in a tree structure: each node of the tree corresponds to a language $\lan$, and the (countable) set of children of $\lan$ are the languages in an infinite perfect tower whose terminal language is $\lan$. For some instances, this tree structure can be built up to an arbitrary height, and this type of structure poses particular challenges for our analysis.

Given that we view an infinite perfect tower as a kind of convergent
sequence of languages, this iterated adversary strategy is analogous
to constructions in topology in which we find the limit points of a space,
then the limits of those limit points, then the limits of those
limit points, and so forth.
We show that this analogy can be made fully precise by
defining a topology on the languages in $\coll$. This viewpoint is useful as it allows us to apply certain topological theorems and properties on the space of languages. 

We therefore define a topological space $(\coll, \mathcal{T})$ whose underlying point set consists of the languages in $\coll$ and whose open sets are constructed as follows.
Given that the crucial question for generation in the limit
is to find languages $\lan' \in \coll$ contained in larger languages $\lan \in \coll$
while fully contains a finite set $F$ of strings seen so far,
we define basic open sets in a form that reflects this. For each language $\lan \in \coll$ and each finite subset $F$ of the ground set of all the possible strings (which is again a countable set), 
define a basic open set as
\[
U_{\lan,F} = \{\lan' \in \mathcal{X} \mid F \subseteq \lan' \subseteq \lan\}.
\]
Let the collection of sets \( U_{\lan,F} \), ranging over all languages \( \lan \in\mathcal{X} \) and all finite sets \( F  \subset \mathbb{N} \), serve as a basis for the open sets. These basis define a topology on \( \mathcal{X} \).  
This topological space is first-countable and Hausdorff. Furthermore, this space is also equipped with a partial order defined by the subset relation on the languages in $\coll$, so this space is an object known as a {\em partially ordered space}, or {\em pospace}, which combines a topological structure with an underlying partial order \cite{gierz-pospace}.

We find that this topology both clarifies some of our basic constructions ---
for example, there exists an infinite perfect tower with respect to
a language $L \in \coll$ if and only if $\lan$ is a limit point in
this topology (as we will prove in Section \ref{sec:element:lower}) --- and also explains the
iterated adversarial strategy described above in some situations, in which languages are stratified
into different levels;
we find that these levels are in fact, in important cases, closely related to 
 the different {\em Cantor-Bendixson derivatives} of $\coll$
under the topology defined by our open sets \cite{settheory}.

Now, if we think of these levels as leading downward to increasingly restrictive languages, which are therefore more likely to be contained in $\trueL$, then the original idea behind the KM algorithm can be viewed as a process of proactively descending down the levels, in order to ensure that the strings being generated will eventually lie in the true language $\trueL$. However, as discussed earlier, if we need to guarantee that the set of all generated strings has a reasonably large density in the true language, we also need to try to be aggressive in terms of guessing the correct language, and we will see that this in effect involves the opposite process, to regularly go back up through the levels.
This comes with the danger of ``overshooting'' our target, as in the discussion around Theorem \ref{thm:acc-intro} above: we need to make sure that we don't overshoot infinitely often, going too high in the levels and no longer remaining inside the true language $\trueL$. 

The Cantor-Bendixson derivatives will help us ensuring that we do not overshoot often. Even once we have the Cantor-Bendixson derivative level of each language, however,
there are multiple further issues in the analysis that need to be handled in order to reason about the possible
adversary strategies. More serious complications will  happen in the case when the Cantor-Bendixson rank of $(\coll, \mathcal{T})$ is infinite or when the \emph{perfect kernel} of the topological space is non-empty. In the latter case, the limit points of the whole set could just be the whole set itself. (Since $\coll$ is infinite, it is completely possible that every language is a limit point in the space $(\coll, \mathcal{T})$). In such a situation, there is no clear way for how we should even build the hierarchical levels as the derivatives are not providing any information, and thus it is unclear how to define these ``levels" and also which two languages should be on adjacent ``levels". We defer these details to the full proof
of the theorem at the end of the paper.

With this as our overview, we now proceed with the full details of
the results and their proofs, beginning in the next section.

%%%%%%%%%%%%%%%%%%%%%%%%%%%%%%%%%%%%%%%%%%%%%%%%%%
% End of introduction
%%%%%%%%%%%%%%%%%%%%%%%%%%%%%%%%%%%%%%%%%%%%%%%%%%

\section{A Validity-Guaranteed Algorithm that is Accurate Infinitely Often}

We now begin presenting the results and their proofs in full detail, starting with index-based generation. 
Recall that  we have a countable list of
candidate languages $\coll = \{\lang{1}, \lang{2}, \dots\}$, where each $\lang{i}$ in $\coll$
is a subset of some countable ground set $U$, and the languages are listed by some pre-determined order;
an adversary chooses one of these languages $\trueL \in \coll$,
and it enumerates the strings of $\trueL$ one by one in some order determined by the adversary.
Let $\seen_t$ be the finite set of all
strings that the adversary has enumerated up 
through time step $t$.

In {\em index-based generation}, we consider an algorithm that tries to guess which language it should be generating from;
in step $t$, the algorithm outputs $i_t$ (intending to generate
from $\lang{{i_t}} \in \coll$), 
and it achieves index-based generation in the limit if if there
is some $\fint$ such that for all steps $t \geq \fint$, we have
$\lang{{i_t}} \subseteq \trueL$.
The KM algorithm \cite{KM24} achieves index-based generation in the limit, and as a result it also achieves {\em element-based generation} in the limit, since it can simply choose any string $\out_t \in \lang{i_t} - \seen_t$ and then we have $\out_t \in \trueL - \seen_t$ once $t \geq \fint$.

In index-based generation, we say that an algorithm is \emph{accurate} at time step $t$ if $\lang{{i_t}} = \trueL$. Unfortunately, in many examples, the KM algorithm is accurate for only a finite number of time stamps. This is because the KM algorithm seeks out languages $\lang{i_t}$ to guess that tend in general to continuously shrink, ensuring that after some finite time, at each step, the suggested language is a subset of $\trueL$. Our goal is to show the following unexpected result: we can design an algorithm that achieves index-based generation in the limit while also guaranteeing that the algorithm will be accurate infinitely often. 
This is the content of the Theorem \ref{thm:acc-intro} from Section \ref{sec:overview}, which we restate here.

\begin{thm}\label{thm:acc}
There is an algorithm that can guarantee 
index-based generation in the limit and also be accurate
in an infinite sequence of time steps.
\label{stmt:inf-accuracy}
\end{thm}

The key point is the combination of these two guarantees.
Index-based generation in the limit on its own is achieved by
the KL algorithm in \cite{KM24}.
And if we simply wanted an algorithm that was accurate in an
infinite sequence of time steps, it could be oblivious to
the adversary's
enumeration and just list the indices in the ``triangular'' order
1, 1, 2, 1, 2, 3, 1, 2, 3, 4, ...\footnote{Note
however that no such oblivious listing of indices can 
work for all collections $\coll$
if index-based generation in the limit is required, since 
the algorithm's listing must contain languages $\lang{i}, \lang{j} \in \coll$
with $\lang{i} \not\subseteq \lang{j}$ that are both listed infinitely often,
and then the adversary can simply declare that the true language
$K = \lang{j}$.}
We also note that the KL algorithm  does not achieve
the combination of guarantees in Theorem \ref{stmt:inf-accuracy}.

We describe an algorithm $\Acc$ that achieves Theorem
\ref{stmt:inf-accuracy} by starting with some definitions.
First, the true language $\trueL$ might appear multiple times
in the sequence of languages in $\coll$;
let $\lang{z}$ be its earliest appearance in terms of the language ordering in $\coll$.
(That is, $\lang{i} \neq K$ for all $i < z$.)

A language $\lang{n}$ is {\em consistent} at step $t$ if
$\seen_t \subseteq \lang{n}$.
Define $\res{\coll}{n} = \{\lang{1}, \lang{2}, \dots, \lang{n}\}$, i.e., the first $n$ languages in $\coll$.

\begin{dfn}\label{dfn:sc}
We say a language $\lang{n}$ is {\em strictly critical} 
at step $t$ if
$\lang{n}$ is consistent with $\seen_t$, and for every language
$\lang{i} \in \res{\coll}{n}$ that is consistent with $\seen_t$,
we have $\lang{n} \subsetneq \lang{i}$.
\end{dfn}
This is almost the same as the definition of a critical language from 
\cite{KM24}, except that we require $\lang{n}$ to be a proper subset of 
each consistent $\lang{i}$ with $i < n$, rather than simply a subset.

In general, whether a language is strictly critical can change 
over different time steps $t$, but
there is always at least one strictly critical for every time step $t$,
since there is always at least one consistent language at time $t$,
and the first consistent language in the ordering of $\coll$ must be
strictly critical.

We can also show

\begin{claim}\label{claim:consist}
    If a language $\lang{n}$ is strictly critical in step $t$, then $\lang{n}$
will continue be strictly critical in every step $r > t$
for which it is consistent at step $r$.
\label{stmt:stay-critical}
\end{claim}

\proof{
If $\lang{n}$ is consistent at step $r$, then for every $i < r$ 
for which $\lang{i}$ is consistent at step $r$, we have
$\lang{n} \subsetneq \lang{i}$ since $\lang{i}$ must also have been consistent at step $t$.
\endpf
}

We next show that the true language eventually becomes strictly
critical; this is a small adaptation of a result from \cite{KM24},
which showed the same thing for critical languages,
rather than strictly critical ones.

\begin{lem}\label{lem:Ksc}
There exists a time step $\zt$ such that for all $t \geq \zt$,
the language $\lang{z}$ is strictly critical at step $t$.
\label{stmt:true-strictly-critical}
\end{lem}

\proof{
For any language $\lang{i}$ with $i < z$, if $\lang{i}$ is a strict subset of $\lang{z}$, then there must be a time when the adversary output some string in $\lang{z} \setminus \lang{i}$. Since there are in total $z-1$ languages come earlier than $\lang{z}$ in the original language ordering, there must be a finite time $\zt$ after which all the languages $\lang{i}$ for $i<z$ that remains consistent after time $\zt$ are supersets of $\lang{z}$.

We claim that $\lang{z}$ is also strictly critical for all $t \geq \zt$.
This is because $\lang{z} \subseteq \lang{i}$ for all consistent 
$\lang{i}$ with $i < z$; since such an $\lang{i}$ is not equal to $\lang{z}$,
it must be a proper superset of $\lang{z}$, 
and hence $\lang{z}$ is strictly critical.
\endpf
}

Now, at time $t$, consider the sequence of all strictly critical languages,
$\lang{{c_t(1)}}, \lang{{c_t(1)}}, \lang{{c_t(3)}}, \ldots$,
where $c_t(j) < c_t(j+1)$.
(Note that this sequence could be either finite or infinite, and
our algorithm will need to implicitly handle both cases.)
By the definition of strict criticality, this sequence is nested by
proper inclusion, in that $\lang{{c_t(j+1)}} \subsetneq \lang{{c_t(j)}}$ for all $j$. So this sequence forms a proper descending chain of languages under set inclusion. 
Let $h_t$ be maximum $j$ such that $c_t(j) \leq t$.

We define $i_1$ arbitrarily, and then for $t > 1$, we define $i_t$ as 
follows.
\begin{itemize}
\item[(a)] If $w_t$ belongs to all the strictly critical languages from 
time $t-1$ --- that is, if $w_t \in \lang{{c_{t-1}(j)}}$ for all $j$ ---
then we define $i_t = c_{t-1}(h_{t-1})$.
\item[(b)] Otherwise, $w_t$ does not belong to all
the strictly critical languages from time $t-1$.
Because the strictly critical languages are nested by proper inclusion,
this means there is some index $k_{t-1}$ such that
$w_t \in \lang{{c_{t-1}(j)}}$ for all $j \leq k_{t-1}$ and 
$w_t \not\in \lang{{c_{t-1}(j)}}$ for all $j > k_{t-1}$.
We define $i_t = c_{t-1}(k_{t-1})$. If no such $k_{t-1}$ exists, then simply let $i_t = c_{t-1}(h_{t-1})$. This edge casethat $k_{t-1}$ does not exist will never appear after some finite amount of time because of Claim \ref{stmt:stay-critical}.
\end{itemize}
In each step, the algorithm guesses $i_t$ as its index for the language
$\lang{{i_t}}$. The crucial part of this algorithm is that at time $t$, the judgement is based on the sequence of strictly critical languages at time $t-1$ instead of time $t$.

First we show

\begin{prop}
The algorithm achieves index-based generation in limit.
\label{stmt:index-based-gen-limit}
\end{prop}

\proof{
In the enumeration of $\lang{z}$, we know by Lemma
\ref{stmt:true-strictly-critical} that there is a step $\zt$
such that for all $t \geq \zt$, the language $\lang{z}$ is 
strictly critical at step $t$.
Let $\fint = 1 + \max(z,\zt)$.
For each $t \geq \fint$, we have $z < t$,
and the language $\lang{z}$ is equal to 
one of the strictly critical languages 
$\lang{{c_{t-1}(j)}}$ in the previous step $t-1$.

If the algorithm uses case (a) in step $t$, then since 
$z \leq t-1$ and $c_{t-1}(h_{t-1})$ 
is the maximum index in $\coll$ of a critical language
less than or equal to $t-1$, we must have $z \leq c_{t-1}(h_{t-1})$.
Therefore $\lang{{i_t}} = \lang{{c_{t-1}(h_{t-1})}} \subseteq \lang{z}$ as required.

If the algorithm uses case (b) in step $t$, then since $w_t \in \lang{z}$,
we must have $z \leq c_{t-1}(k_{t-1})$.
Since the strictly critical languages are nested by proper inclusion,
we therefore have
$\lang{{i_t}} = \lang{{c_{t-1}(k_{t-1})}} \subseteq \lang{z}$.
\endpf
}

The following result now establishes Theorem  \ref{stmt:inf-accuracy},
by showing that from any point in time, there will always be
a future step in which it is accurate.

\begin{prop}
For every step $t$, there is some step $r \geq t$ in which 
the algorithm is accurate.
\label{stmt:alg-later-accurate}
\end{prop}

\proof{
It is sufficient to show the result for $t \geq \fint$,
for the value of $\fint$ defined 
in the proof of Proposition \ref{stmt:index-based-gen-limit}.
For such values of $t$, let $p_t$ be the index of 
$\lang{z}$ in the list of
critical languages at time $t$;
that is, $z = c_t(p_t)$.

Let us suppose by way of contradiction that the algorithm
is not accurate in any step $r \geq t$.
Since we know from the proof of Theorem  \ref{stmt:index-based-gen-limit} that
the algorithm chooses a language $\lang{i_r} \subseteq \lang{z}$,
in each of these steps, it follows that
$\lang{i_r}$ is a proper subset of $\lang{z}$.
By the definition of the algorithm, this language comes from the
list of strictly critical languages at time $r-1$;
in this list, it has some index $q_{r-1}$, and we have $q_{r-1} > p_{r-1}$.

This means in particular that the list of strictly critical languages
at each step $r \geq t$ contains at least one language after $\lang{z}$.
In other words, the number $c_r(p_r + 1)$ is defined for all $r \geq t$:
this is the index in $\coll$
of the strictly critical language immediately 
following $\lang{z}$ in step $r$.

Let's consider the sequence of numbers
$c_r(p_r + 1)$ as $r$ increases from $t$.
There is no index $r$ such that $c_s(p_s + 1)$ remains constant for
all $s \geq r$; this is because 
$\lang{c_r(p_r + 1)}$ is a proper subset of $\lang{z}$, and so in
particular there is some string $u_r \in \lang{z} - \lang{c_r(p_r + 1)}$,
and some step $s > r$ in which $u_r$ is enumerated by the adversary.
Once this happens, $\lang{c_r(p_r + 1)}$ is no longer consistent,
and so it can no longer be on the list of strictly critical languages;
Therefore $c_s(p_s + 1) \neq c_r(p_r + 1)$.

We say that a step $r > t$ is a {\em changepoint} if 
$c_r(p_r + 1) \neq c_{r-1}(p_{r-1} + 1)$;
the argument in the previous paragraph establishes that there is
an infinite sequence of changepoints.
Since $c_r(p_r + 1)$ starts at $c_t(p_t + 1)$ when $r = t$, and
never decreases below $z$, the difference 
$c_r(p_r + 1) - c_{r-1}(p_{r-1} + 1) \neq 0$ cannot be negative at
every changepoint; therefore, there is some step $r$ for which
$c_r(p_r + 1) - c_{r-1}(p_{r-1} + 1)$ is positive.

We consider this value of $r$:
the index in $\coll$ of the next strictly critical language after $\lang{z}$
has increased.
This means that $\lang{c_{r-1}(p_{r-1} + 1)}$ is no longer critical
in step $r$, and by Claim
\ref{stmt:stay-critical}, this can only happen if 
$\lang{c_{r-1}(p_{r-1} + 1)}$ is no longer consistent.
Therefore, the string $w_r$ enumerated by the adversary in step $r$
does not belong to $\lang{c_{r-1}(p_{r-1} + 1)}$.

It follows that in step $r$, the algorithm uses case (b),
and the index $k_{r-1} = p_{r-1}$.
Therefore, in step $r$ the algorithm sets $i_t = c_{r-1}(p_{r-1})$ ---
that is, the algorithm is accurate in step $r$.
\endpf
}

\section{Breath vs Validity - Index Based}

We now begin making use of our notions of density, focusing first on the case of index-based generation, and the way that density can serve as a quantitative notion of breadth.  Beyond just the formal definition of density, it is useful to consider some of the conceptual issues that motivate the definitions, and some of the associated concepts like an infinite perfect tower.  We therefore start with a discussion of these points, and then go on to our characterization of the density that is achievable for index-based generation.

\subsection{Definition of index based breadth}

We start by recalling the definitions of upper and lower density from Section \ref{sec:intro:density}, repeated here for ease of use.
   \begin{dfn}\label{def:up}
Given two countable sets $\lan$ and $\lan'$ where the cardinality of $\lan'$ is infinite, and the elements of $\lan'$ listed in order as
$L' = \{\ell_1', \ell_2', \ell_3', \ldots\}$, 
we say that the {\em upper density} of $\lan$ in $\lan'$ is
   \[ \dup(\lan, \lan') = \limsup_{N \to \infty}  \frac{\mid L \cap \{\ell_1', \dots, \ell_N'\} \mid}{N}.\]
Analogously, we say that the {\em lower density} of $\lan$ in $\lan'$ is
   \[ \dlow(\lan, \lan') = \liminf_{N \to \infty}  \frac{\mid L \cap \{\ell_1', \dots, \ell_N'\} \mid}{N}.\]
   \end{dfn}

The upper density is very sensitive to the ordering of the strings in the languages. 
Below is an example.
\begin{example}[Reordering strings in $\lan$ and $\lan'$ can result in different upper density]\label{example:reorder}
    Let $\lan$ be the set of even positive integers and $\lan'$ be the set of positive integers. Order the integers in $\lan$ in increasing order.
If the ordering in $\lan'$ is also in increasing order,  the upper density of $\lan$ in $\lan'$ is $\frac{1}{2}$.
        Consider an alternative ordering for $\lan'$:
        Start with the first $2^0 = 1$ smallest even number (2).
            Add the next smallest available odd number (1).
             Include the next $2^1 = 2$ smallest even numbers (4, 6).
            Add the next smallest available odd number (3).
             Include the next $2^2 = 4$ smallest even numbers (8, 10, 12, 14).
           Add the next smallest available odd number (5).
             Continue this pattern, adding $2^k$ smallest available even numbers followed by one odd number.
       This process yields a valid ordering of all integers in $\lan'$. Under this ordering, the upper density of $\lan$ in $\lan'$ is 1.
\end{example}

    Despite this sensitivity, the upper density remains a natural definition, and in all our applications, the ordering is important. The ordering, in some sense, indicates which strings are more significant, with the more important strings come earlier in the ordering. In the second ordering above, the scarcity of odd numbers suggests they are less important, implying that $\lan$ captures most of the important properties of $\lan'$. Thus, it is reasonable to conclude that the upper density of $\lan$ in $\lan'$ is 1.

    In particular, our Theorems \ref{thm:vanishing-instance-intro}, \ref{thm:vanishing-char-intro}, and later Theorem \ref{thm:2} are resistant to this ordering sensitivity. In other words, even if the ordering of elements in different languages might be inconsistent, these results still hold. 

The next two definitions introduce the {\it index-based breadth} for an algorithm.

\begin{dfn}[index-based breadth]
Let $\{\lang{1}, \lang{2}, \ldots\}$ be a collection of countably many distinct languages, and $K$ be the true language. Let $\ma$ be an algorithm trying to perform index-based generation, so at each time step $t$ it is trying to guess one language $\lang{i_t}$ it should generate from. We say that the \textit{index-based breadth of $\ma$} with respect to $\{\lang{1}, \lang{2}, \ldots\}$ and $K$ is at most $\beta$, if for any $\epsilon > 0$, the adversary can generate input such that there exist infinitely many time stamps $t_1 < t_2 < \cdots$ where, for each $t_j$, at least one of the following conditions holds:
\begin{enumerate}
    \item The language algorithm $\ma$ guesses at time $t_j$ includes some string not in $K$ (a failure of index-based generation in the limit), or
    \item The upper density of the language $\lang{i_{t_j}}$ guessed by $\ma$ at time $t_j$ is at most $\beta + \epsilon$ in $K$.
\end{enumerate}
Define the index-based breadth of $\ma$ with respect to $\{\lang{1}, \lang{2}, \dots\}$ and $K$ to be
 \[ \beta^* = \inf \{\beta: {\text{the index-based breadth of $\ma$ is at most $\beta$}}\}.\]
\end{dfn}

For index-based generation, we 
say that an algorithm  $\ma$ {\it guarantees validity} if, regardless of how an adversary enumerates the strings in the true language $K$ as input, there exists a finite time $\fint$ (which may depend on the adversary's input sequence) such that at any time $t \geq \fint$, the language $\lang{i_t}$ guessed by the algorithm at time $t$ will be a subset of $K$.

We now define an important notion that will play a central role in determining the index-based breadth. These concepts will also be crucial in analyzing the element-based breadth in later sections.

\begin{dfn}[infinite perfect tower]\label{def:perfecttower}
    Fix a countable infinite ordered list of distinct languages $(\Lambda_1, \Lambda_2, \ldots)$ and let  $K$ be another languages. Define $B_1 = \bigcap_{i \geq 1} \Lambda_i$, and for each $k \geq 2$, 
      \[ B_k = \left(\bigcap_{i \geq k} \Lambda_i\right) \setminus (B_1 \cup B_2 \cup\dots \cup B_{k-1}).\]
      (Observe that $B_k$ is the set of strings of $\trueL$ that are {\em fixed} by $\Lambda_k$ in the sense defined in Section \ref{sec:overview}.)
      We say $(\Lambda_1, \Lambda_2, \dots)$ is an {\it infinite perfect tower} with respect to $K$ if the following hold simultaneously.
  \begin{enumerate}
      \item (Stronger Completeness) $\bigcup_{j=1}^\infty B_j = K$, and for each $j \geq 1$, $\Lambda_j \subsetneq K$.
      \item (Tower property) for each $j \geq 1$, $B_j \neq \emptyset$.
  \end{enumerate} 
\end{dfn}
This definition is equivalent to Definition \ref{def:perfecttower-intro}. We call $K$ the {\it terminal language} of this infinite perfect tower $(\genlang_1, \genlang_2, \dots)$.
 
 We will see in Theorems \ref{thm:1} and \ref{thm:2} that the existence of the ``thinnest" infinite perfect tower in $\coll$ will essentially determine the index-based breadth guarantee of any algorithm in the KM model that guarantees validity.

 Below are some examples regarding infinite perfect towers. 
\begin{example}\label{ex:1}
Let $K$ be an infinite-dimensional vector space with basis $\{e_i\}_{i=1}^{\infty}$, where each string is a finite sum $\sum c_i e_i$ with $c_i \in \mathbb{N}$ and only finitely many $c_i$ are non-zero. The strings of $K$ are enumerated diagonally: first list strings supported on $e_1$ with $\ell_2$ norm $\leq 1$, then strings supported on $\text{span}\{e_1,e_2\}$ with $\ell_2$ norm $\leq 2$, and so on. Each $\lang{i_j}$ is defined as the vector space over $\mathbb{N}$ with basis $\{e_1,\dots,e_j\}$, and has upper density zero in $K$. The sequence $(\lang{i_1}, \lang{i_2}, \dots)$ is an infinite perfect tower. 
Here, $B_j = \text{span}\{e_1,\dots,e_j\} \setminus \text{span}\{e_1,\dots,e_{j-1}\}$ is non-empty, and we have a strictly increasing chain $\lang{i_1} \subsetneq \lang{i_2} \subsetneq \lang{i_3} \subsetneq \cdots$ whose union equals $K$.
\end{example}

\begin{example}
Let the collection of languages be $\lang{i} = \{n \in \mathbb{Z} \setminus \{0\} : i \mid n\}$ and $K = \lang{1}$. No infinite perfect tower exists since $\bigcap_{i \in I} \lang{i} = \emptyset$ for any infinite index set $I$.
\end{example}

Below is an example with an infinitely perfect tower however the upper density is large.
\begin{example}
Let \(K\) be the set of rational points inside the unit disc \(\mathbb{D}^2\) with its boundary $\mathbb{S}^1$. Consider a mapping $f: \mathbb{N} \to  \mathbb{S}^1$. For \(i \in \mathbb{N}\), construct \(\lang{i}\) by taking all the rational points on the unit disc $\mathbb{D}^2$ enclosed in the  convex hull of the points $f(1), \dots, f(i)$. With proper choice of the function $f$, we can obtain an infinite perfect tower
where the upper density of $\lang{i}$ in $K$ tends to $1$ as $i \to \infty$. Here, rational points are ordered by the sum of absolute values of their coordinates.
\end{example}

\subsection{Characterization of Zero Breadth}
Our first result in this framework states an ``if and only if" condition for the breadth to be zero. 
The following is equivalent to the statement of 
Theorem \ref{thm:vanishing-char-intro} presented in Section \ref{sec:overview}.

\begin{thm}\label{thm:1}
Let $\coll= \{\lang{1}, \lang{2}, \dots\}$ be a collection of countably many distinct languages, each as a countable ordered  sequence of   strings.

\begin{enumerate}
    \item(The ``if" direction) Suppose for any $\epsilon > 0$, there exists an infinite perfect tower $(\lang{i_1}, \lang{i_2}, \dots)$ with respect to $K$,  such that for every $ j \geq 1$, the upper density of $\lang{i_j}$ in $K$ is at most $\epsilon$.
Then any algorithm in the KM model that guarantees validity will have zero index-based breadth with respect to $K$. 

\item(The ``only if" direction) Suppose there exists some \(\epsilon > 0\) such that there is no infinite perfect tower \((\lang{i_1}, \lang{i_2}, \ldots)\) with respect to \(K\), where for every $j \geq 1$ the upper density of \(\lang{i_j}\) in \(K\) is at most \(\epsilon\). Then there exists an algorithm \(\ma\) in the KM model which guarantees validity, and the breadth of \(\ma\) is at least \(\epsilon\).
\end{enumerate}
\end{thm}

\subsubsection{Proof of Theorem \ref{thm:1}, the ``if" direction}
\begin{proof}[Proof of Theorem \ref{thm:1}: the ``if" direction]

Suppose for any $\epsilon > 0$, there exists an infinite perfect tower of languages $\lang{i_1}, \lang{i_2}, \dots$ such that
for each $ j \geq 1$, the upper density of $\lang{i_j}$ in $K$ is at most $\epsilon$.

We want to couple several scenarios, when the true language is $L_{i_j}$ versus $K$. Recall the definition of $B_j$'s in Definition \ref{def:perfecttower}. There are two cases.

Case 1: If all of the \(B_j\)'s are finite. If the adversary enumerates the strings in \(B_1\) first, then \(B_2\), followed by \(B_3\), and so on, they can exploit the assumption \(\bigcup B_j = K\) and the fact that the algorithm guarantees validity in the following way. From some time \(T_1\) onward, the adversary knows that the algorithm will generate only strings from \(K\). Suppose the \(T_1\)-th string enumerated by the adversary this way is in \(\lang{j_1}\). The index $j_1$ exists by the Completeness assumption in an infinite perfect tower. The adversary will know that if they enumerate strings in \(B_1, B_2,\) etc., until reaching the \(T_1\)-th step in \(B_{j_1}\), the algorithm will start outputting  strings which are in \(K\).

The first $T_1$ strings are all from $\lang{j_1}$. The adversary then begins to enumerate strings in \(\lang{j_1}\), pretending that the true language is \(\lang{j_1}\). Since the algorithm always guarantees validity, from some time \(T_2 > T_1\), the algorithm will ensure all outputs are from \(\lang{j_1}\). Let \(j_2\) be the smallest index larger than \(j_1\) such that all the first \(T_2\) strings enumerated (which are all in \(\lang{j_1}\)) are in \(B_1 \cup \dots \cup B_{j_2}\). Note that \(B_1 \cup \dots \cup B_{j_2} \subset \lang{j_2}\). The adversary will then pretend the true language is \(\lang{j_2}\) by enumerating the remaining strings in \(B_1 \cup \dots \cup B_{j_2}\) first, followed by the rest of the strings in \(\lang{j_2}\).

Again, since the algorithm guarantees validity, by a similar argument, at some finite time \(T_3 > T_2\), the algorithm will only output strings in \(\lang{j_2}\). This process is repeated. Thus, we find an infinite sequence of time stamps \(T_1 < T_2 < T_3 < \dots\) such that at each time stamp \(T_s\), the possible set of outputs of the algorithm will only be strings in \(\lang{j_{s-1}}\). This implies that the breadth is at most \(\epsilon\) since the upper density of each language in this infinite perfect tower is at most \(\epsilon\), as desired.

The subtlety here is that we need to ensure the way the adversary generates input will indeed enumerate all the strings in \(K\). This is guaranteed by our analysis and procedure above precisely because each \(B_i\) is finite and also due to the Stronger Completeness property (\(\bigcup_{j=1}^\infty B_j = K\)).

\

Case 2: Suppose there are $B_j$'s which are infinite. Without loss of generality, we may assume that $B_1$ is infinite (otherwise we can truncate the first few languages in the infinite perfect tower and it will still be an infinite perfect tower). 
The proof is similar to the previous one, with careful consideration of the requirement that the adversary must eventually enumerate all strings in $K$.

Given that $B_1$ is infinite, the adversary initially pretends that the true language is $\lang{i_1}$. By saying the adversay pretends that the true language is $\lang{i_1}$, the adversary should have a plan to {\it be able} to enumerate all the strings in $\lang{i_1}$. Note that the adversary cannot  enumerate all the strings in $B_1$ first, since this cannot be extended to a full enumeration of all the strings in $\lang{i_1}$ if $\lang{i_1} \setminus B_1$ is again infinite. Therefore, the adversary could instead alternate strings in $B_1$ and $\lang{i_1} \setminus B_1$, until the later is exhausted if ever happens. This is a {\it plan} to enumerate all the strings in $\lang{i_1}$. Due to the algorithm's validity, for this enumeration plan of the adversary, there exists a finite $T_1$ such that at $T_1$, all the possible strings generated by the algorithm should belong to $\lang{i_1}$. 

The adversary's strategy is as follows:

1. Enumerate the first $T_1$ strings in $\lang{i_1}$ in the way described above, compelling the algorithm to output only strings from $\lang{i_1}$ at time $T_1$. 

2. Let $j_2 > i_1$ be the smallest index such that the first $T_1$ strings are in $B_1 \cup \dots \cup B_{j_2}$. Such $j_2$ exists by the stronger completeness condition. Since $B_1 \cup \dots \cup B_{j_2} \subset \lang{j_2}$, the adversary will now pretend that the true language is $\lang{j_2}$. The way the adversary enumerates the strings in $\lang{j_2}$ is similar to above, again alternate between the remaining un-enumerated strings in $B_1 \cup \dots \cup B_{j_2} \subset \lang{j_2}$ and $\lang{j_2} \setminus (B_1 \cup \dots \cup B_{j_2})$, until the latter is exhausted if ever happens.

Similar to the previous case, at some $T_2 > T_1$, the algorithm will exclusively output strings only from $\lang{j_2}$. Let the first $T_2$ strings all belong to $B_1 \cup \dots \cup B_{j_3}$ for some $j_3 > j_2$. Then now the adversary pretends that the true language is $\lang{j_3}$. The process is repeated.  

By the Stronger Completeness condition, the adversary will eventually enumerate all the strings in $K$. This is why the  stronger completeness condition in the definition of an infinite perfect tower is important. Furthermore, at each time stamps $T_t$ for $t \geq 2$, the possible set of output from the algorithm are subsets of $\lang{j_t}$. Therefore the breath of the algorithm is at most $\epsilon$ by definition.  
\end{proof}

\subsubsection{Proof of Theorem \ref{thm:1}, the ``only if" direction}
We will show that the algorithm in Theorem \ref{thm:acc} works, which we denote this algorithm as $\Acc$. 

We first prove some general claims about the possible sequences of languages their algorithm touches on as time progresses. 

\begin{dfn}
  Given a language $K$ and an infinite ordered sequence of languages 
$ (\lanj{1}, \lanj{2}, \dots)$,
  where a language $\lanj{j}$ could be repeated in the sequence. We say this sequence is {\it feasible} if the following conditions hold simultaneously.
\begin{enumerate}
    \item (proper subset) For each $t \geq 1$,   $\lanj{t} \subsetneq K.$
    \item(string enumeration) There is an enumeration of strings in $K = (w_1, w_2, \dots)$ such that for each $t \geq 1$, $w_1, \dots, w_t \in \lanj{t}$. 
\end{enumerate}  
\end{dfn}

\begin{claim}[Union tail]\label{claim:uniontail}
    Let $\mathcal{L} = (\lanj{1}, \lanj{2}, \dots)$ be an infinite feasible sequence.  For each $t \geq 1$, 
   \[
    \bigcup_{t': t' \geq t} \lanj{t'}= K. \]
\end{claim}
\begin{proof}
    Notice that for each $t' \geq t$, $w_{t'} \in \lanj{t'}$. Furthermore, $w_1, w_2, \dots, w_t \in \lanj{t}$. Thus the claim is proved as $(w_1, w_2, \dots)$ is an enumeration of all the strings in $K$. 
\end{proof}

\begin{claim}\label{claim:heredity}
    An infinite subsequence of an infinite feasible sequence is again an infinite feasible sequence. 
\end{claim}
\begin{proof}
    This is by the string enumeration property, and 
    by noticing that if $w_1, \dots, w_t \in \lanj{t}$, then $w_1, \dots, w_t \in \lanj{t'}$ for all $t' > t$. 
\end{proof}

\begin{claim}\label{claim:finite}
    In an infinite feasible sequence $\mathcal{L}$, each language $\lan$ can only appear finitely number of times.
\end{claim}
\begin{proof}
    Let $w_T$ be an string in $K \setminus \lan$.  Thus after time $T$, $\lan$ cannot appear anymore in $\mathcal{L}$ by the string enumeration property. This also means $\lan$ can only appear in the first $T$ positions in $\mathcal{L}$, and thus the occurrence of $\lan$ is finite. 
\end{proof}

\begin{claim}\label{claim:inf}
    Any  feasible sequence $\mathcal{L}$ contains infinitely number of distinct languages. 
\end{claim}
\begin{proof}
    This is a corollary of Claim \ref{claim:finite}. If $\mathcal{L}$ contains a finite number of distinct languages, and each language only appears finitely number of times in $\mathcal{L}$, then the feasible sequence has to have finite length, a contradiction. 
\end{proof}
The claims above could also be easily seen by the topological definition of infinite perfect tower. We will elaborate more in Claim \ref{claim:iptequiv}. 

\begin{claim}\label{claim:K'unionisK}
 Let $\mathcal{L}$ be an infinite feasible sequence with respect to $K$.   For any finite collection of strings $K'$ in $K$, the union of the strings over the languages in $\{\lan:  \lan \in \mathcal{L} {\text{ and }}  K' \subset \lan \}$  is  $K$. 
\end{claim}
\begin{proof}
    Suppose the last string in $K'$ being enumerated is $w_{T_0}$.  By the string enumeration property, for any $t > T_0$, it holds that $K' \subset \lanj{t}$. Then the claim follows by Claim \ref{claim:uniontail}.
\end{proof}
\begin{claim}\label{claim:K'infinite}
  There is no finite collection of strings $K'$ in $K$, such that the set of languages $\{ \lan : \lan \in \mathcal{L} {\text{ and }}  K' \subset \lan \}$  has finite cardinality as a set (i.e., the number of distinct such languages is finite).
\end{claim}
\begin{proof}
    Suppose the last string in $K'$ being enumerated is $w_{T_0}$. Thus for any $t > T_0$, $\lanj{t} \in \{ \lan : \lan \in \mathcal{L} {\text{ and }}  K' \subset \lan \}$. However, by Claim \ref{claim:finite}, each language $\lan$ can only appear finite number of times in $\mathcal{L}$. Thus $\{ \lan : \lan \in \mathcal{L} {\text{ and }}  K' \subset \lan\}$ should contain infinitely many distinct languages. 
\end{proof}

By Theorem \ref{thm:acc},  no matter how  the adversary enumerates the strings in $K$, there is a finite $T$ such that from $T$ on, the algorithm $\Acc$ always outputs strings from $K$. In particular, this algorithm implies that for any $t \geq T$, the algorithm will only output from some language $\lang{i_t} \subset K$. This results in a sequence 
\[
\mathcal{L}_0=(\lang{i_T}, \lang{i_{T+1}}, \dots).
\] 
Note that some of the $\lang{i_t}$ might repeat itself, and a language could reappear in the sequence after disappearing for sometime. The only time when $\lang{i_t} \neq \lang{i_{t+1}}$ is when a string not in $\lang{i_t}$ is being generated by the adversary or because $\lang{i_t}$ is not among the first $t$ languages (with respect to the language ordering in the chain) in the descending (under set inclusion) chain of strictly critical languages at time $t$.  If a list $\lang{i_t}$ is switched out because $s_{{t+1}} \notin \lang{i_t}$, then $\lang{i_t} \neq \lang{i_{t+1}}$ and the language $\lang{i_t}$ will never appear again in the sequence. 

Let $\mathcal{L}_1$ be the subsequence of $\mathcal{L}_0$ by removing all the languages which equal to $K$. 
Note that if $\mathcal{L}_1$ becomes finite, it means that the algorithm eventually stabilizes in $K$, and thus the breadth is one. 
\begin{claim}\label{claim:f1}
   If $\mathcal{L}_1$ is an infinite sequence, then $\mathcal{L}_1$ is also an infinite feasible sequence.  Any infinite subsequence of $\mathcal{L}_1$ is also an infinite feasible sequence.
\end{claim}
\begin{proof}
    The proof is by Claim \ref{claim:heredity}. 
\end{proof}

From now on, we assume that the breadth is not one, as otherwise we are done. 

We will prove by contradiction. We assume that there is no infinite perfect tower where the upper density of each language is at most $\epsilon$ is $K$. Suppose the breadth is at most $\epsilon$.
We will use the properties derived from the claims above to extract an infinite perfect tower where the upper density of each language is at most $\epsilon$ is $K$, which will lead to a contradiction and thus complete the proof.

Let $\epsilon > 0$. The algorithm in Theorem \ref{thm:acc} having breadth at most $\epsilon$ implies that there exists an infinite sequence of time stamps $t$ such that the language at time $t$, which is $\lang{i_t}$ in $\mathcal{L}_0$, has upper density at most $\epsilon$ in $K$. 
We denote by $\bar{\mathcal{L}}$ the sequence of languages induced by these $\lang{i_t}$ in $\mathcal{L}_1$ (or equivantly, in $\mathcal{L}_0$). By Claim \ref{claim:heredity}, $\bar{\mathcal{L}} = (\lanj{1}, \lanj{2}, \dots)$ is again an infinite feasible sequence.

We now show that $\bar{\mathcal{L}}$ contains an infinite perfect tower $\bar{\mathcal{L}}' = (\lanj{1}', \lanj{2}', \dots)$ where the upper density of each language is at most $\epsilon$ is $K$.

Notice the string $w_1 \in \bigcap_{\lan \in \bar{\mathcal{L}}} \lan$ as $w_1 \in \bigcap_{\lan \in {\mathcal{L}_0}} \lan$. We build the infinite perfect tower by first letting $\lanj{1}'$ to be the first language in $\bar{\mathcal{L}}$, which is $\lanj{1}$.  Consider the string $w_{t_2}$ with the smallest index that is not in $ \lanj{1}' = \lanj{1}$. Let $\lanj{i_2}$ be the left-most language in the ordering of $\bar{\mathcal{L}}$ to the right of $\lanj{1}$ but contains all the strings enumerated from $w_1$ all the way to $w_{t_2}$ (inclusive of $w_{t_2}$). Such an index $i_2$ must exists, because after some finite time, all the languages in $\bar{\mathcal{L}}$ will contain all the strings from $w_1, w_2, \dots, w_{t_2}$. Since $w_{t_2} \notin \lanj{1}$, clearly $\lanj{i_2}$ is not a subset of $\lanj{1}$.  Let $\lanj{2}'$ be $\lanj{i_2}$.

Let $K'$ be the set of strings  $\{w_1, \dots, w_{t_2}\}$. Consider the subsequence of $\bar{\mathcal{L}}$ induced by all the languages in $\bar{\mathcal{L}}$ which contains $K'$. Call this subsequence $\bar{\mathcal{L}}_1$. By Claims \ref{claim:K'unionisK} and \ref{claim:K'infinite} applied to the infinite feasible sequence $\bar{\mathcal{L}}$, this subsequence $\bar{\mathcal{L}}_1$ contains infinite many distinct languages and has union $K$, and thus is also an infinite feasible sequence by Claim \ref{claim:heredity}. Furthermore, all the languages in this sequence contain $K'$. Continuing repeating this process, we will reach an infinite sequence of distinct languages
$\bar{\mathcal{L}}'= (\lanj{1}', \lanj{2}', \dots)$ such that their union is $K$, and $w_1, \dots, w_t \in \lanj{t}$ for every $t \geq 1$. 

We hope that by removing some languages from $\bar{\mathcal{L}}'$ if necessary, we will get an infinite perfect tower. 

For the feasible sequence $\bar{\mathcal{L}}'= (\lanj{1}', \lanj{2}', \dots)$, we slightly abuse notation by defining $B_1 = \bigcap_{i \geq 1} L_i$, and for each $k \geq 2$, 
      \[ B_k = (\bigcap_{i \geq k} \lanj{i}' )\setminus (B_1 \cup B_2 \cup\dots \cup B_{k-1}).\]
By our construction, $\bigcup_{i \geq 1} B_i = K$. 

We first need to guarantee that each $B_j \neq \emptyset$ by removing some languages from $\bar{\mathcal{L}}'$ if necessary. 

For each $i$, define $D_i = B_1 \cup \dots \cup B_i$. Thus $D_i \subset \lanj{i}'$. Then $B_{i+1} = \emptyset $ if and only if $D_i = D_{i+1}$.  We claim that for each $i$, there is some $i' > i$ such that $D_{i'} \neq D_i$. This is because some string $w$ not in $D_i$ will eventually be output by the adversary at some finite time $T > i$, and all the languages $\lanj{t}'$ for $t \geq T$  would contain this string $w$. This implies  $D_{T+1} \neq D_i$.

Therefore, we can partition the time stamps $\mathbb{N}_{\geq 1}$ into infinitely many intervals of finite lengths $[a_1, a_2-1], [a_2, a_3-1], [a_3, a_4-1], \dots $ etc such that within each interval the sets $D_i$ are identical. In other words, 
for each $s \geq 1$, \[ D_{a_s} = D_{a_s+1} = \dots = D_{a_{s+1}-1} \subsetneq D_{a_{s+1}}.\]

We will conduct the following procedure to obtain a subsequence of $\bar{\mathcal{L}}'$ which is an infinite perfect tower. For each interval $[a_s, a_{s+1}-1]$, by the definition of $D_i$'s,
there is an string $w \in D_{a_{s+1}} \setminus D_{a_s}$ that is not shared by  all the languages $\lanj{a_s}', \lanj{a_s+1}', \dots, \lanj{a_{s+1}-1}'$. Thus there exists an index $i_s$ where $a_s \leq i_s < a_{s+1}$  such that $\lanj{i_s}'$ does not contain this string $w \in \lanj{a_{s+1}}' \setminus \lanj{a_s}'$. We remove all the languages $\lanj{a_s}', \lanj{a_s+1}', \dots, \lanj{a_{s+1}-1}'$ except this single language $\lanj{i_s}'$. Notice that the string $w$ satisfies $w \in \lanj{t}'$ for all $t \geq a_{s+1}$, and in particular, $w \in \lanj{i_{s+1}}'$ 
but $w$ is not in $\lanj{i_s}'$.

We simultaneously conduct this replacement for all the intervals at once, and obtain a new infinite sequence
\[ \bar{\mathcal{L}}^*= (\lanj{i_1}'
, \lanj{i_2}', \dots).\] Define $B_k' = (\bigcap_{t \geq k} \lanj{i_t}' ) \setminus (B_1' \cup B_2' \cup\dots \cup B_{k-1}')$ recursively. Our argument earlier shows that some string in $\bigcap_{t \geq s+1} \lanj{i_t}'$ is not in $\lanj{i_s}'$, and thus $ B_{s+1}' \neq \emptyset$. Furthermore, by Claim \ref{claim:heredity} and the definition of an infinite feasible sequence, we have that $\bar{\mathcal{L}}^*$ is an infinite perfect tower. 

However, by our assumption  that the breadth is at most $\epsilon$ and by our choice of $\bar{\mathcal{L}}$, each language in the infinite perfect tower $\bar{\mathcal{L}}^*$ has upper density at most $\epsilon$ in $K$. This contradicts with the assumption that there is no infinite perfect tower with upper density less than $\epsilon$. 

\subsection{Determining Index-based Breadth in the general case}
We generalize Theorem \ref{thm:1} to determine the worst-case breadth for algorithms which guarantee validity. 
\begin{dfn}[Truth index]
Fix a collection of countably many distinct languages $\coll$ and the true language $K \in \coll$. Define the {\it truth index} with respect to this collection $\coll$  and $K$ to be the smallest $\tau \geq 0$ such that for any $\epsilon > 0$, there exist an infinite perfect tower of languages $(\lang{i_1}, \lang{i_2}, \dots)$ with respect to $K$ from $\coll$ such that the upper densities $\dup(\lang{i_j}, K) \leq \tau +\epsilon$ for all $j \geq 1$.

In the case when there is no infinite sequence of languages $(\lang{i_1}, \lang{i_2}, \dots)$ from $\coll$ that could form an infinite perfect tower with respect to $K$, 
we define the truth index $\tau$ with respect to $\coll$ and $K$ to be $1$. 
\end{dfn}

Below is an example showing the existence of any given truth index $0 < \tau \leq 1$. 
\begin{example}
Consider any enumeration of rational numbers in $[\tau, 1]$. Let $Q(n)$ be the first $n$ such rational numbers. For each integer $i \geq 1$, let the language $L_i$ be $Q(i) \cup (\mathbb{Q} \cap [0, \tau])$.  Let the true language $K = \mathbb{Q} \cap [0,1]$. 
     Thus the truth index with respect to this set of languages and $K$ is $\tau$. 
\end{example}

The theorem below is a generalization of Theorem \ref{thm:1}, which determines the breadth value given any countable set of languages and the true language $K$, over the worst-possible adversary and any clever algorithm that guarantees validity. It essentially says that the breadth is the truth index. More precisely, it is given by the following theorem. 
\begin{thm}[Determining the index-based breadth value] \label{thm:2}
Fix any list of countable number of languages $\coll$ and a true language $K \in \coll$. Let $\tau \in [0,1]$ be the truth index with respect to $\coll$ and the true language $K$. Then the following hold. 
\begin{enumerate}
    \item For any algorithm $\ma$ in the KM model that guarantees validity, the adversary has a way to generate input such that the index-based breadth of the algorithm $\ma$ with respect to $\coll$ and $K$ is at most $\tau$.
    \item There is an algorithm $\ma$ in the KM model such that for any adversary input, the index-based breadth of the algorithm $\ma$ with respect to $\coll$ and $K$ is at least $\tau$. 
\end{enumerate}
\end{thm}
Note that Theorem \ref{thm:1} is the special case of this theorem when the truth index $\tau = 0$. 

\begin{proof}
    \emph{(Sketch.)} We first prove the theorem when $
    \tau= 1$. Item 1 is obvious as breadth is always at most one. To prove Item 2, the proof is almost identical to the proof  of Theorem \ref{thm:1}. The only difference is that here we also need to consider the case when there is no infinite perfect tower. This case is already handled in the ``only if" part proof of Theorem \ref{thm:1}.
    The proof for the general case when $0 < \tau < 1$ is almost identical to the proof of Theorem \ref{thm:1}. 
\end{proof}

\section{Breadth vs Validity - Element-based}

For the remainder of the paper, we focus on 
{\em element-based generation} in the limit.
Recall that this is the original definition of generation:
in each time step $t$, the algorithm outputs a string $\out_t$, and the goal is to have $\out_t \in \trueL - \seen_t$ for all $t$ after some finite $\fint$.
We will say that an element-based generation algorithm {\em guarantees validity} if it satisfies this property.

Now, we begin by recalling the definition of $O(E,\ma)$ from Section \ref{sec:intro:density}.
Let $E = \{w_1, w_2, \dots\}$ be any adversary sequence enumerating a language $K$, and let $\ma$ be any algorithm in the KM model that generates from $K$ in the limit (thus guarantees validity). In each time step the adversary outputs a string $w_t$, and $\ma $ outputs an string $o_t$ in reply. Let the infinite set $O(E,\ma) =\{o_1, o_2, \dots\}$ consist of all strings that the algorithm $\ma$ ever outputs. 

The central question in this section and the next is the following.
\begin{quote}
    What is the density of $O(E, \ma)$ in $K$? In particular, does the requirement of validity inherently impose constraints that limit the density of $O(E, \ma)$ in $K$?
\end{quote}

We will study this question for upper density in the current section, and the more complex version of the question for lower density in the following section.

Before we get this, we first observe an important point about the ordering of the strings and its role in defining density.
Recall that there is an underlying set $U$ of all strings; each language $\lang{i} \in \coll$ is an infinite subset of $U$.
For evaluating density, we need an ordering on the strings in each language $\lang{i}$, since we need to be able to speak about the first $N$ strings of $\lang{i}$, and then to take the limit as $N$ increases.
We will assume that the ordering of strings is consistent across all languages.  (For example, it might be simply inherited from the ordering of strings in $U$.) That is, if \( v_1 \) precedes \( v_2 \) in some language \( \lan \), then for any other language \( \lan' \) containing \( v_1 \) and \( v_2 \), it must also hold that \( v_1 \) precedes \( v_2 \) in \( \lan' \).
Without a universal ordering of strings, this notion of  density becomes ill-defined. For instance, two languages \( K \) and \( K' \) might contain almost the same set of strings but have very different orderings. In such a case, the algorithm would be unable to determine which language is the true one, potentially leading to the same output sequence yielding vastly different densities in \( K \) and \( K' \). An instance of this phenomenon is illustrated in Example \ref{example:reorder}.

\subsection{Upper density guarantee for element-based breadth in the limit}

\begin{thm}
There exists an algorithm $\ma$ in KM model that guarantees validity such that, for any $c < 1 $ and any enumeration $ E $ of the true language $ K $ by an adversary, the output $ O $ always ensures that the set $O$ has an upper density of at least $ c/2 $ within $ K $.
\end{thm}
Recall that, as we discussed earlier in the introduction, since the algorithm is not allowed to repeat an element generated by the adversary, the highest achievable upper density is at most 
$1/2$.

We first need a definition for the ``finite" version of the upper density.  Given a finite set $O_t = (o_1, o_2, \dots, o_t)$ and an infinite ordered set $K = (v_1, v_2, \dots)$, where $O_t \cap K \neq \emptyset$. The {\it ordered density} of $O_t$ in $K$   is defined as 
 \[
 \dor(O_t, K) = \frac{|O_t \cap K|}{N_T}
\]
 where $|\cdot|$ is the cardinality of a finite set, and $N_T$ is the minimum string index in $K$ such that $(O_t \cap K ) \subset \{v_1, \dots, v_{N_T}\}$.
\begin{proof}

% \end{dfn}

Let us denote the algorithm in Theorem \ref{thm:acc}, which guarantees accuracy infinitely often, as $\Acc$. We aim to adapt this algorithm to ensure that for any $ c < 1$, there exists a modified algorithm $\Acc' $ such that, regardless of the adversary's enumeration $ E $, the output set $O(E, \Acc') $ consistently has an upper density of at least $ \frac{c}{2} $ in the true language $K$.

The new algorithm $\Acc'$ is essentially a ``lazy" version of $\Acc$. We describe and couple $\Acc$ and $\Acc'$ in the following way. 

Given an enumeration $E$, for any time $t$, let $i_t$ be the index of the language identified by $\Acc$. Therefore, 
the indices of the languages guessed by $\Acc$ at times $1,2,\dots$ would be $i_1, i_2, \dots$ (Note that the indices $i_t$ only depend on the first $t$ strings in the sequence $E$). We say an index $i_t$ is {\it rich} if the language $\lang{i_t}$  is a strict super set of the language  $\lang{i_{t-1}}$. Note that the rich indices are {\it completely} determined by the adversary input $E$ and the algorithm $\Acc$.

The output of the algorithm $\Acc'$ is defined sequentially as follows:

The language index identified by $\Acc'$ initially follows the same index identified by $\Acc$. Thus, at time $t$, $\Acc'$ identifies $\lang{i_t}$, just as $\Acc$ does, until a specific time $t_1$, which is the first point when a rich index is encountered. Starting from $t_1$, $\Acc'$ becomes {\it lazy}. At this stage, it begins continuously outputting strings from $\lang{i_{t_1}}$ that have not been used previously by either $\Acc'$ or the adversary. These strings are selected in increasing order according to the underlying string ordering. 
This process continues until time $t_2 - 1$, at which point the ordered density of the outputs $(o_1, o_2, \dots, o_{t_2-1})$  is at least $c/2$ within $\lang{i_{t_1}}$. Achieving this is feasible because, starting from $t_1$, the algorithm exclusively outputs strings from $\lang{i_{t_1}}$. Finally, at time $t_2$, the algorithm $\Acc'$ performs a refresh.

During the time interval from $ t_1+1 $ to $t_2 $, the new algorithm $\Acc'$ might have identified different languages compared to $\Acc$. Specifically, during this time interval, while the new algorithm $\Acc' $ is only identifying $L_{i_{t_1}}$, the original algorithm $\Acc$ would have identified a sequence of languages $ \lang{i_{t_1+1}}, \dots, \lang{i_{t_2}} $. Let the indices of this sequence be referred to as the {\it missed interval} at time $ t_2 $, denoted by $\mathcal{M}(t_2) = (i_{t_1+1}, \dots, i_{t_2}) $. These indices represent the set of languages that $\Acc$ would have identified between time $ t_1+1 $ and time $ t_2 $.

If no index in $\mathcal{M}(t_2) = (i_{t_1+1}, \dots, i_{t_2})$ is rich, then starting at time $t_2$, $\Acc'$ will refresh, output an string from $\lang{i_{t_2}}$, and repeat the process, i.e., output repeatedly strings from $\lang{i_{t_2+1}}, \lang{i_{t_2+2}}, \dots$ until finding the next index that is rich.

If there are indices in $\mathcal{M}(t_2) = (i_{t_1+1}, \dots, i_{t_2})$ that are rich, then there are some indices corresponding to languages  which are maximal with respect to string inclusion. Pick the latest  timestamp between $t_1+1$ and $t_2$ which corresponds to a maximal rich index, say it is $i_{t_1'}$ with $t_1+1 \leq t_1' \leq t_2$. Algorithm $\Acc'$ will then output from $\ga{t_1'}$ until sometime $t_3 -1 > t_2$ when the ordered density is at least $c/2$ in $\ga{t_1'}$.

Now at timestamp $t_3$, and we can again define the set of missed indices $\mathcal{M}(t_3) = (i_{t_1'+1}, \dots, i_{t_3})$. If no index there  is rich, then $\Acc'$ will refresh, and at time $t_3$ output an string from $\lang{t_3}$ and repeat. 

If some index in $\mathcal{M}({t_1'+1, t_3}) = (i_{t_1'+1}, \dots, i_{t_3})$ is  rich, then we find the timestamp in this interval whose corresponding language is rich and maximum (break the tie by the latest time stamp), say index $t_2'$, for some $t_1'+1 \leq t_2' \leq t_3 $, then $\Acc'$ will repeatedly generate strings from $\ga{t_2'}$ until time $t_4-1$ so that the ordered density is at least $c/2$ in $\ga{t_2'}$, and repeat by setting $\mathcal{M}({t_4}) = (i_{t_2'+1}, \dots, i_{t_4-1})$, and repeat.

We now claim that $O(E, \Acc') $ has an upper density of at least $c/2 $ in the true language $K = \lang{z} $. By the validity of $\Acc$, we can assume that, after some point in time, $\Acc$ only identifies languages that are subsets of $ K $. Consequently, after some finite time, if a missing interval contains the index $z $, then $z $ will always remain rich within this interval. As a result, the true language $K = \lang{z}$ will consistently be the maximum language and $z$ is always rich.

We prove that there exist infinitely many timestamps $T_1 < T_2 < T_3 < \dots $ at which the ordered density of the algorithm $\Acc'$ in $ K $ is at least $c/2$. This result implies the theorem, as the algorithm always outputs the earliest unused string in $K $. Consequently, suppose at a particular timestamp $T_j$, the output is the $\ell$-th string in $K$, then up till $T_j$, the first $\ell $ strings in $ K $ are fully occupied either by the adversary or by $\Acc'$, which aligns with the definition of the upper density when picking the first $ \ell $ strings in $ K $.

We now prove by contradiction. Suppose after some time $T$, the ordered density of the algorithm $\Acc'$ in $K$ is always strictly less than $c/2$.  We may also assume after $T$, the language identified is always a subset of $K$. 
Theorem \ref{thm:acc} implies that $K$ will be identified by $\Acc'$ infinitely often.

Let $i_t' $ denote the index of the language identified by $\Acc'$ at time $t $. Recall the language identified by $\Acc$ at time $t$ is $\lang{i_t}$. If $i_t' \neq i_t $, then $i_t $ must belong to some missing interval. 
Note that by Theorem \ref{thm:acc}, the true language $K = \lang{z}$ will appear infinitely often. If $z = i_t'$, then by our algorithm $\Acc'$, the ordered density in $K$ will be larger than $c/2$ for some time after $T$, which is a contradiction. Therefore $z \neq i_t'$. 
Thus the index $ i_t $ lies in a missing interval. By the argument above, $K$ will be the maximum language whose index is rich in this missing interval. By the description of the algorithm, the algorithm will address this missing interval, by repeatedly outputting strings from $ K $ (be lazy) until the ordered density in $ K $ reaches at least $ c/2 $, which leads to a contradiction again. 
\end{proof}

The lazy version of the algorithm in this proof might produce many more strings that do not belong to $\trueL$ compared to the original algorithm it is adapted from; this could be viewed as a kind of increased {\it hallucination}.  Despite this, however, the lazy algorithm still achieves element-based generation in the limit, and in the process its output set achieves high upper density.

The result here is very much oriented toward guaranteeing a large upper density in $\trueL$.
By itself, it does not rule out the possibility that the lower density could be very small. 
We take up this question in the next section, where we design an algorithm whose output set $O(E,\ma)$ has a non-trivial lower density, in that it will be lower-bounded by an absolute constant.

\section{Lower density guarantee in element-based breadth in the limit}
\label{sec:element:lower}

In general, guaranteeing a large lower density is much more challenging than guaranteeing a large upper density.  
To guarantee a large upper density, it suffices to find a sequence of (possibly very sparse) positions \( N_1, N_2, \dots \to \infty \) such that for each \( i \), the intersection of the algorithm's overall output \( O(E,\ma) \) with the first \( N_i \) strings in \( K \) has a large size relative to \( N_i \).  
However, to guarantee a high lower density, we must show that for all but finitely many positions \( N \), the intersection of \( O(E,\ma) \) with the first \( N \) strings in \( K \) has a large size relative to \( N \).  

It is not hard to construct examples where the lower density and upper densities differ significantly. 
Here is a basic example.  For each $k$, define the interval $I_k = [3^k, 2 \times 3^k]$, so $|I_k| = 3^k + 1$. Define the gap interval $J_k = [2 \times 3^k + 1, 3^{k+1} - 1]$, so $|J_k| = 3^k - 1$. Together, the intervals $I_k$ and $J_k$ form a partition of the set of positive integers. 
Consider a language $\lan = \bigcup_{k \geq 0} I_k$.
Then the upper density of $\lan$ in $\mathbb{N}$ is as large as $3/4$. However the lower density is at most $1/2$. 
Moreover, by modifying this example, the upper and lower densities can be made arbitrarily far apart.

\paragraph{Positive lower density}
Initially, it might seem
that regardless of the algorithm $\ma$ used to ensure validity, the adversary can always enumerate the strings of $K$ in a way that keeps the lower density of $O(E,\ma)$ in $K$ small. In fact, though, we will show there exists an algorithm $\ma$ that can achieve a non-trivial positive value $c > 0$ for the lower density on all instances.
This is the content of 
Theorem \ref{thm:element-lower-intro},
and the current section builds up to the proof of it.

The proof of Theorem \ref{thm:element-lower-intro} contains 
the most involved arguments in the paper, and so we build up to it in a few steps.
First, it is helpful to begin with some basic motivating examples in Section \ref{subsec:element:lower:examples}
that suggest at a high level why we adopted the definitions and approaches that we did. 
These examples become increasingly rich, and analyzing them requires some work in itself.
The examples reveal some structured types of enumeration strategies for the adversary in general, and 
suggest the need for a more formal description of 
how these strategies operate.
We do this by introducing a topology on the
collection of languages $\coll$, 
in Section \ref{subsec:element:lower:topology}.
The topological definitions make clear that a key aspect of what is going on in the examples can be described by the \emph{Cantor-Bendixson rank} of the topological space we've defined, and in 
Section \ref{subsec:element:lower:finite} we show how to handle the case in which this rank is finite.
Finally, we handle the general case, when the rank may be infinite, in Section \ref{subsec:element:lower:infinite}.

\subsection{Motivating examples}
\label{subsec:element:lower:examples}

We begin with a sequence of examples of increasing complexity, which help in understanding how different types of adversarial enumerations of the true language $\trueL$ might operate.

\begin{example}\label{exam:fallback}
    All the languages are subsets of positive integers. 
    Let the true language $K$ be the set of all positive integers. 
For each $k$, the language $\lang{k} = [k] \cup 100\mathbb{N_+} =  \{1,2,\dots, k\} \cup \{100,200,300,\dots\}$.  For notational convenience, we will say that the true language $K$ is denoted $\lang{0}$.

For example, $\lang{1} = \{1, 100,200,300, \dots\}$ and $$\lang{120} = \{1,2,3,4,5,\ldots,116,117,118,119,120,200,300,400, 500,\dots\},$$ where $\lang{120}$ contains the first 120 natural numbers before moving on to the subsequent multiples of 100.
(Note that the number ``100'' here is just a stand-in for any arbitrarily large constant.)
\end{example}

This example will turn out to be useful in understanding why guarantees for element-based generation can be stronger than guarantees for index-based generation, and in the process will show some of the strategies we'll employ for designing element-based generation algorithms.
In particular, the sequence of languages $\lang{1}, \lang{2}, \lang{3}, \ldots$ form an infinite perfect tower with respect to the terminal language $\trueL$, with each language $\lang{i}$ having upper density .01 in $\trueL$. 
Now suppose we have an algorithm $\ma$ that achieves index-based generation in the limit.
If we follow the adversary strategy used to prove Theorem \ref{thm:vanishing-char-intro}, we can think of an adversary that initially pretends the true language is $\lang{1}$ for a sufficiently long time, forcing the algorithm $\ma$ to eventually guess $\lang{1}$.
Then, the adversary can switch to pretending that the true language is $\lang{2}$ for a sufficiently long period, causing $\ma$ to eventually guess either $\lang{1}$ or $\lang{2}$, and this process continues indefinitely.  In this way, there is an infinite subsequence of time steps at which $\ma$ is guessing languages of upper (and hence also lower) density at most .01 in $\trueL$.

On first inspection, it might seem that if we now turn this into an algorithm for element-based generation in the limit, any such algorithm will spend almost all of its time generating multiples of 100, resulting in an output set $O(E,\ma)$ that is close to density 0.01 on this instance.
But this is where the element-based guarantee allows us to say something stronger.
Specifically, it doesn't matter if the algorithm $\ma$ infinitely often guesses languages $\lang{i}$ of low density provided that cumulatively, element by element over time, it ends up generating enough of $\trueL$.
We now show how this works, and how to achieve a set $O(E,\ma)$ with lower density close to 0.5 on this example.

The stronger algorithm for Example \ref{exam:fallback} works as follows. Each time the adversary enumerates a multiple of 100, the algorithm outputs the next unused string that is a multiple of 100. However, if the adversary enumerates a number $n$ that is not a multiple of 100, the algorithm outputs $n+1$.  This is the crucial point: the algorithm does this even though there is no ``evidence'' based on what it's seen so far that $n+1$ belongs to the true language $\trueL$.  To handle the special case where $n+1$ is already enumerated or happens to be a multiple of 100, we refine the rule: in such cases, the algorithm instead outputs the next unused natural number after $n$.

The first key observation is that this algorithm guarantees validity. If the adversary enumerates from $K$, then validity is trivially satisfied. On the other hand, if the adversary enumerates from some $\lang{k}$, there exists a time $t$ by which the adversary has enumerated all strings of ${1,2,\dots,k}$. After this time, the algorithm may output $k+1$ (or the next available number if $k+1$ is a multiple of 100), but beyond that, it will output only multiples of 100 indefinitely. Consequently, after a finite time, the algorithm enumerates only from the true language.

The second key observation is that this algorithm yields a lower density of approximately 0.5 for $O(E, \ma)$ in $\trueL$. 
If the adversary's true language is $K$, the algorithm and adversary effectively alternate strings of $K$ in increasing order. If the adversary's true language is $\lang{k}$, then they similarly alternate strings of the ``core'' subset in increasing order, aside from an additional finite set of strings. This alternating behavior can be used to directly show that the density approaches 0.5 in the long run.

Another way to view this algorithm --- which will be important for thinking about more complex examples --- is that each time the adversary pretends that the true language is some $\lang{i}$, the algorithm actually ``overshoots" by producing strings from $\lang{i+1}$ for a finite amount of time (in this particular example, overshooting by one time step). Since the algorithm only ``overshoots'' for a finite duration, validity is always guaranteed.  

Furthermore, due to this ``overshooting,'' when the true language is $\mathbb{N}_+$, the algorithm will not produce only multiples of 100 in the limit.  

This effect of {\it falling back} (i.e., overshooting) to a larger language proves particularly useful when designing algorithms with high lower density guarantees. However, determining the appropriate fallback language is highly non-trivial, as we must ensure both validity and a high lower density.  

The next example illustrates a collection of languages $\coll$ with some specific appropriate structure, and it allows us to talk about a richer case in which one can effectively assign the fallback language.

\begin{example}\label{example:1}
This example serves as the foundational test case for motivating the more general algorithm. It also further elucidates the basic advantage of ``falling back".

The construction of the languages works as follows. All the strings are subsets of $\mathbb{N}$. We start with some markers, at integers $a_0 = 0, a_1 = 3, a_2 = 3^2, a_3 = 3^3, \dots$. Let $\lang{n}$ be the language $\bigcup_{i = 0}^\infty [a_i, a_i + n]$.
For example, the first picture below illustrates $\lang{1}$. It means attaching an interval of length 1 at each of the marker $a_i$. The second picture below illustrates $\lang{4}$, where we attach an interval of length 4 at each of the marker $a_i$.

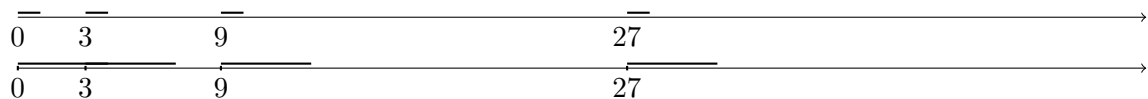
\vspace{0.2in}
\begin{tikzpicture}[scale=0.3]
    % Draw the x-axis
    \draw[->] (0,0) -- (50,0) node[right] 
    {};
    
    % Define the left endpoints of the intervals
    \foreach \x in {0,3,9,27,81} {
        % Draw the interval of length 1
        \draw[thick] (\x,0.2) -- ({\x+1},0.2);
        % Label the left endpoint
        \node[below] at (\x,0) {\x};
    }
\end{tikzpicture}

\begin{tikzpicture}[scale=0.3]
    % Draw the x-axis
    \draw[->] (0,0) -- (50,0) node[right] {};
    
    % Define the left endpoints of the intervals
    \foreach \x in {0,3,9,27,81} {
        % Draw the interval of length 4
        \draw[thick] (\x,0.2) -- ({\x+4},0.2);
        % Add vertical ticks at the start of each interval
        \draw[thick] (\x,0.1) -- (\x,-0.1);
        % Label the left endpoint
        \node[below] at (\x,0) {\x};
    }
\end{tikzpicture}
So we have an infinite sequence of nested languages $\lang{1} \subsetneq \lang{2} \subsetneq \lang{3} \subsetneq \cdots$. Let the true language be $K = \mathbb{N}$.

The algorithm which guarantees non-zero lower density is as follows.  Define a sequence of positive integers $N_1, N_2, \dots$ which goes to infinity quickly. 
At time $t$ suppose the adversary generates $w_t$, and let $\lang{i_t}$ be smallest language that is consistent (i.e., contains all the previous strings  $\seen_t = \{w_1, \dots, w_t\}$ generated by the adversary). As long as  $w_t \leq \max(N_{i_t}, o_1, \dots, o_{t-1}, w_1, \dots, w_{t-1})$, the algorithm will generate $w_t + 1$, or the smallest unused string. Otherwise,  the algorithm generates the next smallest unused string in $\lang{i_t}$. From this we can see that the lower density of $O(E, \ma)$ in $K$ is almost $1/2$ as long as $N_1, N_2, \dots$ go to infinity sufficiently fast.

This algorithm works because it has the ``fall back" feature: by manually forcing the algorithm to output $w_t + 1$ whenever $w_t$ is small. 
\end{example}

\begin{example}\label{example:2}

This example builds upon Example \ref{example:1} and provides further motivation for the general algorithm.

For each positive integer \( i \), let \(\lang{i}^{(0)}\) be the language $\lang{i}$ defined in Example \ref{example:1}. Define the strictly increasing function \( f_i^{(1)}(k): \mathbb{N}_+ \to \mathbb{N}_+ \) such that its image consists of the strings of \(\lang{i}^{(0)}\). Such an \( f_i^{(1)} \) is thus unique. For each positive integer \( j \), we construct the language \(\lang{i:j}^{(1)}\) so that its \( \ell \)-th string is given by \( f_i^{(1)}(\lang{j}^{(0)}(\ell)) \), where \(\lang{j}^{(0)}(\ell)\) denotes the \( \ell \)-th string of \(\lang{j}^{(0)}\), ordered from smallest to largest.

This construction ensures that
\[
\bigcup_{j=1}^\infty \lang{i:j}^{(1)} = \lang{i}^{(0)}.
\]
In essence, we have formed a hierarchical structure resembling a tree, where each level refines the previous one. The shape of this tree is as follows in Figure \ref{fig:ex2}.

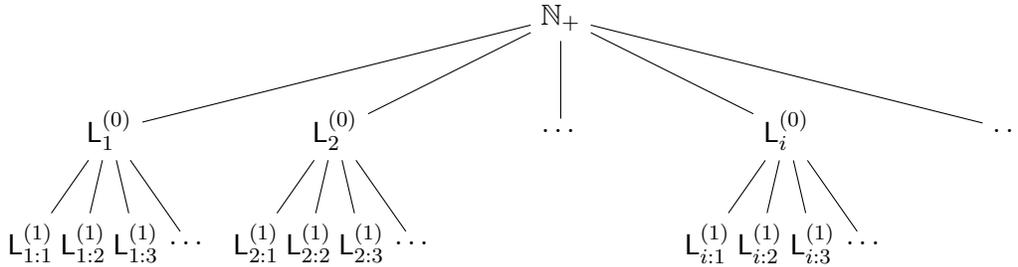
\begin{figure}[h]
\begin{tikzpicture}[
 level distance=1.5cm, % Vertical distance between levels
  level 1/.style={sibling distance=3cm}, % Distance between K_1, K_2, etc.
  level 2/.style={sibling distance=0.7cm} % Closer distance between leaves
]
  % Define styles for nodes
\node (N) {$\mathbb{N}_+$}
    child {node (K1) {$\lang{1}^{(0)}$}
        child {node (L1) {$\lang{1:1}^{(1)}$}}
        child {node (L2) {$\lang{1:2}^{(1)}$}}
        child {node (Ld) {$\lang{1:3}^{(1)}$}}
        child {node (La) {$\dots$}}
    }
    child {node (K2) {$\lang{2}^{(0)}$}
        child {node (Lb1) {$\lang{2:1}^{(1)}$}}
        child {node (Lb2) {$\lang{2:2}^{(1)}$}}
        child {node (Lbd) {$\lang{2:3}^{(1)}$}}
        child {node (Lb) {$\dots$}}
    }
    child {node (Kd) {$\dots$}}
    child {node (Ki) {$\lang{i}^{(0)}$}
        child {node (Li1) {$\lang{i:1}^{(1)}$}}
        child {node (Li2) {$\lang{i:2}^{(1)}$}}
        child {node (Lid) {$\lang{i:3}^{(1)}$}}
        child {node (Li) {$\dots$}}
    }
    child {node (Kd) {$\dots$}}

;  
\end{tikzpicture}
\caption{Illustration of Example \ref{example:2}}\label{fig:ex2}
\end{figure}

Figure \ref{fig:ex2} illustrates the hierarchical structure of the tree. The leaves under $\lang{1}^{(0)}$, represented by the sequence $\lang{1:1}^{(1)}, \lang{1:2}^{(1)}, \dots$, collectively form the union $\lang{1}^{(0)}$ and is an infinite perfect tower. As a result, $\lang{1}^{(0)}$ is labeled as the parent of these leaves, which are ordered accordingly. The same relationship holds for each $\lang{i}^{(0)}$ and its corresponding leaves.

This tree also exhibits a recursive structure: since $\lang{1}^{(0)}, \lang{2}^{(0)}, \dots$ form an infinite perfect tower and their terminal language (union) is $\mathbb{N}_+$, the root of the tree is naturally $\mathbb{N}_+$, with its children ordered. The root has infinitely many children, and each vertex $\lang{i}^{(0)}$ likewise has infinitely many children.

Now, consider the set of all languages appearing in the tree. The algorithm proceeds as follows: at each time step $t$, identify the smallest index $i_t$ such that $\lang{i_t}^{(0)}$ remains consistent with the adversary’s input so far. Next, find the smallest $j_t$ such that $\lang{i_t:j_t}^{(1)}$ is also consistent. These indices $i_t$ and $j_t$ are well-defined. Moreover, $i_t$ is monotonic in $t$, and $j_t$ is also monotonic when restricted to instances where $i_t = i$.

The key challenge is that the algorithm must correctly determine which language to fall back to, ensuring that it maintains consistency while guaranteeing good lower density.

\begin{enumerate}  
\item If $i_t = i_{t-1}$ but $j_t \neq j_{t-1}$, the algorithm aggressively guesses that the true language is $\lang{i_t}^{(0)}$. It identifies the largest string $w$ (with respect to $\lang{i_t}^{(0)}$) output by both the adversary and the algorithm up to time $t$. The algorithm then outputs the next smallest unused string in $\lang{i_t}^{(0)}$, provided it is at most $w+2$. If no such string exists, the algorithm instead outputs the next smallest unused string in $\lang{i_t:j_t}^{(1)}$. This behavior persists until the pair $(i_t, j_t)$ changes.

Falling back to $\lang{i_t}^{(0)}$ is justified. If the true language is indeed $\lang{i_t:j_t}^{(1)}$, the algorithm eventually restricts its outputs to strings from this set. However, if the true language is $\lang{i_t}^{(0)}$, relying solely on $\lang{i_t:j_t}^{(1)}$ could result in tiny low lower density. To prevent this, the algorithm temporarily falls back to $\lang{i_t}$ for a finite time.

At this stage, the algorithm cannot fall back to $\mathbb{N}$, as this would lead to incorrect behavior. For example, if the true language were $\lang{1}^{(0)}$, an adversary could successively pretending that the true languages are  $\lang{1:1}^{(1)}, \lang{1:2}^{(1)}, \lang{1:3}^{(1)}, \dots$, forcing the algorithm to fall back to $\mathbb{N}$ infinitely often. This would break the validity guarantee.

\item If $i_t \neq i_{t-1}$, the algorithm aggressively guesses that the true language is $\mathbb{N}_+$. It identifies the largest string $w$ (with respect to $\mathbb{N}_+$) output by both the adversary and the algorithm up to time $t$. From that point onward, as long as $(i_t, j_t, k_t)$ remains unchanged, the algorithm always outputs the next smallest unused string in $\mathbb{N}_+$ that is at most $w+2$. If no such string exists, it outputs the next smallest unused string in $\lang{i_t}^{(0)}$.  
\end{enumerate}  

In summary, the algorithm operates by iteratively identifying the lowest and leftmost consistent descendant of the root and falling back to one of its ancestors. It always falls back to the least common ancestor of two identified languages, if the identified minimum consistent language changes.

We now verify that this algorithm generates the true language in the limit while ensuring a large lower density of output strings.  

Case 1: The true language is some \( \lang{i:j}^{(1)} \).  
  After some finite time \( T \), the algorithm will generate only strings from \( \lang{i:j}^{(1)} \), since the strategy only changes when the indices \( i_t, j_t \) are updated.  

Case 2: The true language is some \( \lang{i}^{(0)} \).  
  In this case, the algorithm primarily generates strings from \( \lang{i:1}^{(1)}, \lang{i:2}^{(1)}, \dots \). However, each time the algorithm transitions to the next leaf in the tree, it temporarily falls back to \( \lang{i}^{(0)} \). This scenario mirrors the argument in Example \ref{example:1}, where applying the same algorithm to \( \lang{1}^{(0)}, \lang{2}^{(0)}, \dots \) with respect to \( \mathbb{N}_+ \) guarantees validity and ensures that the output maintains a large lower density within \( \lang{i_t}^{(0)} \).  

Case 3: The true language is \( \mathbb{N}_+ \). 
  A similar argument applies, except now we track changes in the index \( i_t \). Eventually, the algorithm will output a large fraction of the initial strings of \( \lang{1}^{(0)} \), then fall back to $\mathbb{N}$, and then output a large fraction of the initial strings of \( \lang{2}^{(0)} \), and so on. This ensures that the lower density remains large.  

\end{example}

 \begin{example}\label{example:3}
This is a slightly more complicated example illustrating that we can also build upon the previous example. 

For positive integers \(i, j, k\), define the function \(f_{i:j}^{(2)}(k): \mathbb{N}_+ \to \mathbb{N}_+\) to be the unique strictly increasing function such that its image is the set of strings in \(\lang{i:j}^{(1)}\).  

We will construct the language \(\lang{i:j:k}^{(2)}\) so that the \(\ell\)-th string is given by \(f_{i:j}^{(2)}(\lang{i:j}^{(1)}(\ell))\), where \(\lang{i:j}^{(1)}(\ell)\) is the \(\ell\)-th string of \(\lang{i:j}^{(1)}\), ordered from smallest to largest. Clearly, this construction ensures that  
\[
\bigcup_{k=1}^\infty \lang{i:j:k}^{(2)} = \lang{i:j}^{(1)}.
\]

What we have created is a tree of the following structure in Figure \ref{fig:tree2}.  

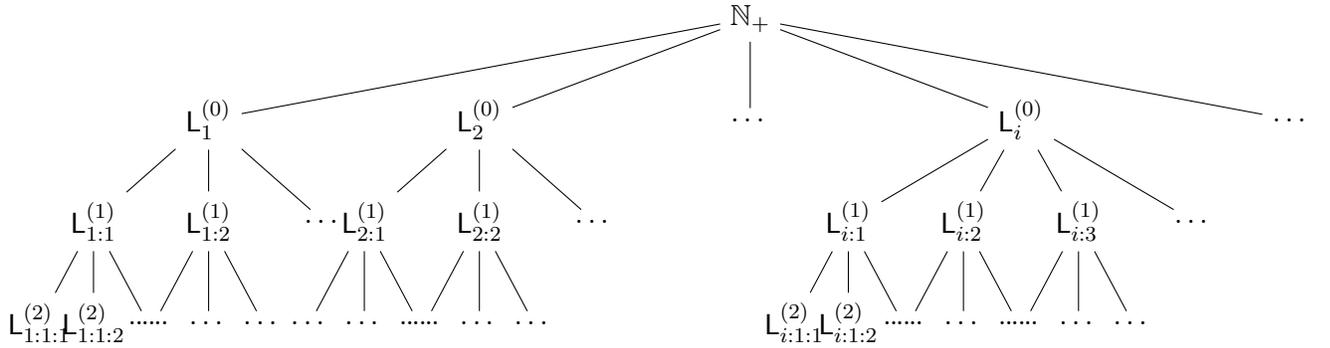
\begin{figure}[h]
    \centering
\begin{tikzpicture}[scale = 0.9,
     level 1/.style={sibling distance=4cm, level distance=1.5cm},
    level 2/.style={sibling distance=1.7cm, level distance=1.5cm},
    level 3/.style={sibling distance=0.8cm, level distance=1.5cm},
    % level 4/.style={sibling distance=0.1cm, level distance=2cm},
 % level distance=1.5cm, % Vertical distance between levels
 %  level 1/.style={sibling distance=4cm}, % Distance between K_1, K_2, etc.
 %  level 2/.style={sibling distance=2cm} % Closer distance between leaves
 %  % level 3/.style={sibling distance = 0.3cm}
 %  level 3/.style={sibling distance = 1cm}
 %  level 4/.style={sibling distance=0.3cm},
]
  % Define styles for nodes
\node (N) {$\mathbb{N}_+$}

        child {node (L1) {$\lang{1}^{(0)}$}
            child {node (L11) {$\lang{1:1}^{(1)}$}
                child {node (L111) {$\lang{1:1:1}^{(2)}$}}
                 child {node (L112) {$\lang{1:1:2}^{(2)}$}}
                  child {node (L113) {$\dots$}}  
              }
            child {node (L12) {$\lang{1:2}^{(1)}$}
               child {node (L121) {$\dots$}}
                 child {node (L122) {$\dots$}}
                  child {node (L123) {$\dots$}}
               }
            % child {node (L13) {$\lang{1:3}^{(1)}$}}
            child {node (La) {$\dots$}}
               }
        child {node (L2) {$\lang{2}^{(0)}$}
            child {node (L21) {$\lang{2:1}^{(1)}$}
                 child {node (L211) {$\dots$}}
                 child {node (L212) {$\dots$}}
                  child {node (L213) {$\dots$}}
                 }
            child {node (L22) {$\lang{2:2}^{(1)}$}
                 child {node (L221) {$\dots$}}
                 child {node (L222) {$\dots$}}
                  child {node (L223) {$\dots$}}      
            }
            % child {node (L23) {$\lang{2:3}^{(1)}$}}
            child {node (Lb) {$\dots$}}
        }
        child {node (Kd) {$\dots$}}
        child {node (Li) {$\lang{i}^{(0)}$}
            child {node (Li1) 
                   {$\lang{i:1}^{(1)}$}
            child {node (Li11) {$\lang{i:1:1}^{(2)}$}}
            child {node (Li12) {$\lang{i:1:2}^{(2)}$}}
            % child {node (Li13) {$\lang{i:1:3}^{(2)}$}}
            child {node (Lbbb) {$\dots$}}
                    }
            child {node (Li2) {$\lang{i:2}^{(1)}$}
                   child {node (Li21) {$\dots$}}
                 child {node (Li22) {$\dots$}}
                  child {node (Li23) {$\dots$}} 
                  }
            child {node (Lid) {$\lang{i:3}^{(1)}$}
                 child {node (Li31) {$\dots$}}
                 child {node (Li32) {$\dots$}}
                  child {node (Li33) {$\dots$}}     
              }
             child {node (Lid) {$\dots$}}
        }
            child {node (Kd) {$\dots$}}
;

\end{tikzpicture}
    \caption{Illustration of a tree structure for Example \ref{example:3}}
    \label{fig:tree2}
\end{figure}

Let the collection $\coll$ of languages consist of all languages that appear in this tree. The algorithm is a generalization of the one in Example \ref{example:1}, and proceeds as follows: at each time \(t\), it searches for the leftmost language \(\lang{i_t}^{(0)}\) with which the current language is consistent. Under this \(\lang{i_t}^{(0)}\), it then looks for the leftmost \(\lang{i_t:j_t}^{(1)}\) that is consistent, and beneath this \(\lang{i_t:j_t}^{(1)}\), it looks for the leftmost \(\lang{i_t:j_t:k_t}^{(2)}\) that is consistent.

As \(t\) progresses, the algorithm behaves as follows:

When \(k_t \neq k_{t-1}\), but \(i_t = i_{t-1}\) and \(j_t = j_{t-1}\), the language falls back to \(\lang{i_t:j_t}^{(1)}\).

When \(j_t \neq j_{t-1}\), but \(i_t = i_{t-1}\), the language falls back to \(\lang{i_t}^{(0)}\).

When \(i_t \neq i_{t-1}\), the language falls back to \(\mathbb{N}_+\).

When we say that the algorithm ``falls back'' starting at some time $t$ to some language $\lan$,  we mean the following, which is similar to our plan in earlier examples: suppose at time \(t\), the correct consistent child string \(w\) is the largest string in $\lan$ that has been produced by the input and adversary up to time \(t\). From that point onward, as long as the algorithm does not change its strategy (i.e., does not change the minimum consistent language identified), it will always output the next unused string in $\lan$ that is at most \(w + 2\). Once no string smaller than \(w+2 \) remains in $\lan$, the algorithm continues by outputting from the appropriate child of $\lan$ in the tree. 

Again in this algorithm, once the identified language changes, the algorithm always falls back to to the smallest common ancestor of the two languages. 
\end{example}

\subsection{Setup: A topological point of view}
\label{subsec:element:lower:topology}

The general theorem is strongly motivated by the examples above, though the structure of the sets in these examples is relatively simple. First of all, in a general set of languages $\coll$, we might not be able to get the simple tree structure as in the examples. Furthermore, we might also not find the ``minimum" language that is consistent. Handling the general case requires involved technical details, as the set of languages can intersect each other in arbitrary ways.
It turns out that we can transfer our insights from the examples to the general case by introducing a topological structure on the set of languages; we will find that a number of the structures we discussed in the examples have natural analogues in the general case via topological definitions.

We now set up the topological structure that we use.
First, without loss of generality, we assume each language in $\coll$ is a subset of $\mathbb{N}$ (any discrete countable set has strings to be enumerated by $\mathbb{N}$). 
Given our countable collection of languages ${\mathcal{X}}$, we define a topology $\mathcal{T}$ on it as follows. For each language $\lan \in \coll$ and each finite subset $F$ of $\mathbb{N}$, 
define one of the basic open sets to be
\[
U_{\lan,F} = \{\lan' \in \mathcal{X} \mid F \subseteq \lan' \subseteq \lan\}.
\]
Let the collection of sets \( U_{\lan,F} \), ranging over all languages \( \lan \in\mathcal{X} \) and all finite sets \( F  \subset \mathbb{N} \), serve as a basis for the open sets. This defines a topology on \( \mathcal{X} \).  

This topological space is first-countable and Hausdorff. Most importantly, it possesses the key property we require, namely, that the definition of limit points aligns with that of the infinite perfect tower we have defined. 

\begin{claim}\label{claim:iptequiv}
  Language $\lan' \in \mathcal{X}$ is a limit point in $(\mathcal{X}, \mathcal{T})$ if and only if $\lan'$ is a terminal language of some infinite perfect tower.  
\end{claim}
\begin{proof}
One direction is straightforward. Suppose there exists an infinite perfect tower \( \lang{1}, \lang{2}, \dots \) with terminal language \( \lan' \). We show that \( \lan' \) is a limit point under \( \mathcal{T} \). It suffices to verify that each basis set \( U_{\lan', F} \) contains at least one language \( \lang{i} \) from the tower.  For each string \( a \in F \), let \( n_a \) be the index \( j \) such that \( B_j \) contains \( a \), where \( B_j \) is as defined in Definition \ref{def:perfecttower}. Such an \( n_a \) exists because the union \( \bigcup_j B_j = \lan' \) and \( a \in \lan'\), with the sets \( B_j \) being disjoint. Since \( F \) is finite, we can define \( N = \max_{a \in F} n_a \). Choosing any \( n > N \), we find that \( \lang{n} \) contains \( F \), so \( \lang{n} \in U_{\lan, F} \). Thus we have that  \( \lan'\) is a limit point.

For the converse, suppose \( \lan' \) is a limit point under \( \mathcal{T} \). By definition, each open set \( U_{\lan', F} \) contains a language \( \lang{i} \) distinct from \( \lan' \). Let \( F_i \) consist of the first \( n_i \) strings of \( \lan' \) for some $n_i$. Then let $\lang{i}$ be a language in \( U_{\lan', F_i} \) that is different from $\lan'$. Let $n_{i+1}$ be the minimal $j$ such that the ${j}$-th string in $\lan'$ is not in $\lang{i}$. Such $j$ exists because $\lang{i} \subsetneq \lan'$. By our construction, $n_{i+1} > n_i$. Repeatedly define the sequence $(n_i)_i$ and the sequence of languages $(\lang{i})_i$. 
This generates a sequence of distinct languages \( \lang{1}, \lang{2}, \dots \) such that each \( \lang{i} \subsetneq \lan' \) and contains the first \( n_i \) strings of \( \lan' \). By definition, this sequence forms an infinite perfect tower with terminal language \( \lan' \). 
\end{proof}

Given a topological space \((\mX, \mathcal{T})\), define its {\it derived set} \(d(\mX)\) as the set of limit points of \(\mX\) with respect to \(\mathcal{T}\). Using this, we work with a topological notion know as the {\it Cantor-Bendixson rank} of \(\mX\), defined as follows \cite{settheory}.

For a given set \(\mX\), we recursively construct a sequence of subsets \(\mX^{(i)}\) as follows. Let \(\mX^{(0)} = \mX\). Define \(\mX^{(1)} = d(\mX)\), the set of limit points of \(\mX\).
Recursively, for each integer \(i \geq 1\), set \(\mX^{(i+1)} = d(\mX^{(i)})\), i.e., the set of limit points of \(\mX^{(i)}\) in \(\mX\).

It is well known that the sequence \(\mX^{(i)}\) is nested and decreasing. This process terminates when the derived set becomes empty or stabilizes, meaning \(d(\mX^{(\alpha+1)}) = d(\mX^{(\alpha)})\) for some ordinal \(\alpha\). The smallest such \(\alpha\) for which \(\mX^{(\alpha)} = \mX^{(\alpha+1)}\) is called the {\it Cantor-Bendixson rank} of \((\mX, \mathcal{T})\).

If the process stabilizes to a non-empty set, this final set is known as the {\it perfect kernel} of \(\mX\). It is known that for any countable subset of a first-countable space, the Cantor-Bendixson rank is countable. Since our topology is first-countable, this applies in our setting. Note that a perfect kernel in our setting will be that every point is a limit point in this set. One could imagine that this should look like in Figure \ref{fig:tree2} where we have tree of infinite depth.

In the rest of the section, without loss of generality, we could assume the ground set of strings is the set of natural numbers $\mathbb{N}$, and $K$ is a subset of $\mathbb{N}$. $K$ consists of the set of ordered integers $\{\psi(1), \psi(2), \dots\}$. In other words, $\psi: \mathbb{N} \to \mathbb{N}$ is an increasing function where the $i$-th string in $K$ is $\psi(i)$. For each string $x \in K$, we denote $\Succ_K(x)$ to be the next string in $K$ following immediately after $x$. i.e., $\Succ_K(x) = \psi( \psi^{-1}(x)+1)$.

\subsection{Finite rank case}
\label{subsec:element:lower:finite}

We first prove a weaker version of our main result, in which we restrict to the special case in which the collection of languages has empty perfect kernel and finite Cantor-Bendixson rank with respect to the topology we have defined.  This turns out to be a useful first step, and  
some of the proof techniques will be important ingredients later in the general setting. 
Specifically, we will prove the following.

\begin{thm}\label{thm:finiteRankHighLD}
    There is an algorithm $\ma$ that achieves element-based generation in the limit and has the following property. Suppose the collection of languages $\coll$ has associated topology $(\mX, \mathcal{T})$ with empty perfect kernel and Cantor-Bendixson rank at most $r$ for some finite positive integer $r$. Then for any adversarial enumeration $E$ of one of the languages $\trueL \in \coll$, the set of output strings $O(E,\ma)$ generated by the algorithm has a lower density in $K$ that is at least $1/(3(r+1))$. 
\end{thm}

\begin{proof}
We first establish a hierarchy on the set of languages \(\mathcal{X}\) as follows.  

Since \(\mathcal{X}\) has an empty perfect kernel, there exist languages in \(\mathcal{X} \setminus d(\mathcal{X}) = \mathcal{X} \setminus \mathcal{X}^{(1)}\). Define this set as \(\mathcal{Y}_0\). More generally, for each \(i \geq 0\), let  
\[
\mathcal{Y}_i = \mathcal{X}^{(i)} \setminus \mathcal{X}^{(i+1)}
\]
where $\mathcal{X}^{(0)} = \mathcal{X}$, 
so that the sets \(\mathcal{Y}_i\) are disjoint and together form a partition of \(\mathcal{X}\).  

By construction, each language \(\my \in \mathcal{Y}_i\) is a limit point of \(\mathcal{X}^{(i)}\) but not of \(\mathcal{X}^{(i+1)}\). This implies that there exists an infinite sequence (an infinite perfect tower) of languages in \(\mathcal{X}^{(i)}\) accumulating at \(\my\), but no such infinite sequence exists in \(\mathcal{X}^{(i+1)}\).  
Specifically, the sets \(\mathcal{Y}_i\) consist of isolated points in \(\mathcal{X}^{(i)}\) with respect to the topology \(\mathcal{T}\). In particular, for each \(\my \in \mathcal{Y}_i\), there exists an open neighborhood \(N_\my\) containing \(\my\) such that no other point in \(\mathcal{Y}_i\) belongs to \(N_\my\).

We say languages in $\mathcal{Y}_i$ are of {\it level $i$}. Purely by the definition of the derived sets $\mathcal{X}^{(i)}$, we have the following two claims. 
    
The next claim is the most important property of the levels we would like to achieve. 
        \begin{claim}\label{claim:nosamelevel}
Suppose a language \(\mx\) is at the \(i\)-th level. Then \(\mx\) cannot be the terminal language of an infinite perfect tower of languages consisting only of languages at level \(i\) or higher.        \end{claim}
        \begin{proof}
If \( \mx \) is at level \(0\), the claim follows trivially from the definition of level \(0\) and \(\mathcal{Y}_0\). Now, assume \(\mx\) is at level \(j+1\) for some \(j \geq 0\), so that \(\mx \in \mathcal{X}^{(j+1)}\). If \(\mx\) were the terminal language of an infinite perfect tower consisting of strings at level \(j+1\) or higher, then we would have \(\mx \in d(\mathcal{X}^{(j+1)})\), which implies \(\mx \in \mathcal{X}^{(j+2)}\). This contradicts the assumption that \(\mx\) is at level \(j+1\), as it would instead belong to level at least  \(j+2\).
        \end{proof}

        \begin{claim}
Let \(\my\) be any language at level \(i\) for some \(i \geq 1\). Then \(\my\) can be expressed as a limit point of a sequence of languages that lie solely at level \(i-1\).
        \end{claim}
        \begin{proof} 
       Suppose \(\my\) is at level \(i\). If there exists an infinite perfect tower at level \(i\) or higher with limit \(\my\), this would contradict Claim \ref{claim:nosamelevel}. If there is no sequence in \(\mY_{i-1}\) with limit \(\my\), it implies that \(\my\) is an isolated point in \(\mX^{(i-1)}\). Then \(\my\) should belong to \(\mY_{i-1}\) rather than \(\mY_i\), which is a contradiction.
\end{proof}

The levels establish a ranking for the set of languages in \(\mX\). We will dynamically construct a directed tree structure on \(\mX\) that respect this ranking, meaning that a parent is always on a strictly higher level than its child, as described in the following property.

\begin{property}\label{property2}
All directed edges point from a language at a lower level to a language at a strictly higher level. The language of the child is always a strict subset of the parent language. 
\end{property}

Before explaining how we construct the trees, we first present some useful claims.    

\begin{claim}\label{claim:Kipt}
    If there is an infinite perfect tower with terminal language \( K \), then for any language \( \laj \) that is a strict subset of \( K \), there exists an infinite perfect tower starting at \( \laj \) and with terminal language \( K \).
\end{claim}

\begin{proof}
    Let the infinite perfect tower leading to \( K \) be \( \lanj{1}, \lanj{2}, \dots \). Choose an arbitrary string \( a \in \laj \). By the definition of an infinite perfect tower, and recalling the definition of \( B_j \)'s, there must be some \( j \) such that \( a \in B_j \), since \( \bigcup_{j=1}^\infty B_j = K \). Therefore, the sequence \( \laj, \lanj{j+1}, \lanj{j+2}, \dots \) forms an infinite perfect tower starting at \( \laj \) and with $K$ as the terminal language.
\end{proof}

At each time \( t \), we will eliminate the languages that are not consistent up to time \( t \). Thus, the set of languages remaining at time \( t \) is always a subset of \( \mX \). Note that by definition, if \( \my \) is inconsistent, meaning the adversary shows an string not in \( \my \), then all descendants of \( \my \) will also become inconsistent since the descendants are strict subsets of $\my$.

If there is no infinite perfect tower leading to the true language \( K \), then, by the same argument in Theorem \ref{thm:1}, the output density could approach \( 1/2 \). Our proof will also address this case. For convenience, however, it may be easier for the reader to assume that infinite perfect towers do lead to \( K \).

Let the Algorithm in Theorem \ref{thm:acc} be denoted as $\Acc$, and recall the definition of the strictly critical languages as in Definition \ref{dfn:sc}. By Lemma \ref{lem:Ksc} in Theorem \ref{thm:acc}, after some finite time \( T \), the true language \( K \) will always be strictly critical. We assume that we are operating after time \( T \).

At any time stamp $t$, we will use $\ga{t}$ denote the language identified by $\Acc$ at time $t$. We will keep this notation throughout this proof. 
Note that we will not use $\ga{t}$ to denote the $t$-th language in the original list of languages anymore.

\vspace{0.1in}
\noindent\textbf{Dynamic trees $\md_t$. }
Now, we explain how we construct the dynamic forest \( \md_t \) and the dynamic tree \( \md_t^* \) at each step, ensuring that they preserve Property \ref{property2}.

At each time \( t \), after the adversary generates a string, we will have a sequence of strictly critical sets, and the language \( \ga{t} \) identified by $\Acc$ will be among them. We will also eliminate the languages that are no longer consistent. We construct a forest \( \mathcal{D}_t \) on the set of remaining strings as follows:

For each language \( \lan \) (which may not necessarily be critical), we draw a directed edge from \( \lan \) to the minimum strictly critical language that strictly contains $\lan$ and has a level higher than that of $\lan$, if such a language exists. More specifically:

\begin{enumerate}
    \item\label{e1} If there is no strictly critical language that strictly contains $\lan$ and whose level is above that of $\lan$, do not add a directed edge from $\lan$ (so $\lan$ is a root).
    \item\label{e2} If there are strictly critical languages that strictly contain $\lan$ and whose level is above $\lan$, and a minimum (under inclusion) exists among them, add a directed edge from $\lan$ to this minimum.
    \item\label{e3} If no minimum (under inclusion) exists, it means that there is an infinite descending chain of strictly critical languages on the levels above $\lan$. In this case, we pick the \( t \)-th language (from largest set to smallest set under set inclusion) on this descending chain. 
    % on each of these chains at these levels and select the minimum one.
\end{enumerate}

Since each vertex has exactly one directed edge pointing upward, we obtain a forest. Let \( \md_t^* \) denote the connected component of the forest containing \( \ga{t} \). In this way, we have constructed a forest where there is a directed edge from each child to its parent. (Note that we do not exclude the possibility that \( \ga{t} \) is a single vertex, i.e., it has neither a child nor a parent.)

The following claim is an important ingredient in proving validity. It heavily relies on our definition of levels, which guarantees the property in Claim \ref{claim:nosamelevel}.
\begin{claim}\label{claim:finiteTime}
There exists a finite time \( T \) which depends on the adversary's enumeration, such that 
for every language $\lan$ which is a strict subset of \( K \) at the level of \( K \) or higher,  for any  time \( t>  T \), the identified language \( \ga{t} \) is not a descendant of (including being identical to) $\lan$ in $\md_t$.
\end{claim}
\begin{proof}
Consider all subsets that are strict subsets of \( K \) and are at the level of \( K \) or higher, and whose descendants have once been identified by the algorithm. Clearly, their union is a subset of \( K \).

If their union is a strict subset of \( K \), then as soon as the adversary outputs an string in \( K \) but not in the union of these subsets, all these subsets (and their subsets) will become inconsistent and thus be purged.

If their union is \( K \), assume, for the sake of contradiction, that for an infinite amount of time $t$,  their children are identified as $\ga{t}$. However, since each set that is a strict subset of \( K \) can only be identified a finite number of times, these children which were identified by the algorithm  form an infinite perfect tower with terminal language \( K \). Consequently, these original sets, being subsets of \( K \), would also form an infinite perfect tower, which could be easily seen from the topological definition. This contradicts Claim \ref{claim:nosamelevel}.
\end{proof}

\begin{claim}\label{claim:finitetimeKancestor}
     After some finite time $T$ which depends on the adversary enumeration, any language that is identified by the algorithm $\Acc$ will have $K$ as an ancestor 
     in $\md_t^*$. 
\end{claim}
\begin{proof}
We first assume that we are after the finite time guaranteed by Claim \ref{claim:finiteTime} and also after the finite time during which \( K \) is always strictly critical. Additionally, we assume that \( T \) is larger than the position number of \( K \) in the original language listing. This ensures that for any descending chain of strictly critical languages, \( K \) is always among the first \( T \) of them. Therefore, whenever we encounter Item \ref{e3} in the tree construction above, the language we select will be a subset of \( K \).

After $T$, we can assume that the language \( \ga{t} \) identified by the algorithm $\Acc$ at time \( t \) is a subset of \( K \). By Claim \ref{claim:finiteTime}, we may further assume that \( \ga{t} \) is below the level of \( K \). 

Now, suppose \( \ga{t} \) is one level below \( K \). It must have a outdegree one, since connecting to \( K \) is one option (by Item \ref{e1}). Therefore, it must choose to connect to some other language for its parent. If its parent is above the level of \( K \), then, by the minimality of the directed edges (Items \ref{e2} and \ref{e3}), it must connect to some strict subset of \( K \), which contradicts Claim \ref{claim:finiteTime}. If its parent is a set on the same level as \( K \), then again, by minimality (Items \ref{e2} and \ref{e3}, and the fact that \( T \) is larger than the position number of \( K \) in the original language ordering), it must connect to a strict subset of \( K \), which again contradicts Claim \ref{claim:finiteTime}.

Now, suppose \( \ga{t} \) is \( i \geq 2 \) levels below \( K \). For the sake of contradiction, assume that \( K \) is not an ancestor of \( \ga{t} \). Note that \( \ga{t} \) cannot be a root, because it could have been connected to some language, with \( K \) being an option (by Item \ref{e1}). Consider the root of the tree containing \( \ga{t} \). The root must be either a superset or a subset of \( K \), since after some finite time \( T \), \( K \) itself is strictly critical, and the set of critical strings is forming a descending chain. Additionally, only strictly critical languages could have incoming edges, as per our design of the directed edges.

If the root is a subset of \( K \), it cannot be on or above the level of \( K \), by Claim \ref{claim:finiteTime}. Therefore, this root must be below the level of \( K \), and since it is a subset of \( K \), it must have an outgoing directed edge, with \( K \) being one option. This contradicts the fact that this is a root.

So, we assume the root is a superset of \( K \). We now consider the topmost descendant of this root, such that the descendant is a subset of \( K \), below the level of \( K \), but not a descendant of \( K \). Such a descendant must exist since \( \ga{t} \) itself is one. Call this descendant \( \lan' \). By our assumption and Claim \ref{claim:finiteTime}, the parent of \( \lan' \) is a superset of \( K \). However, the fact that this parent is a superset of \( K \) and is on or below the level of \( K \) would contradict the way we assign the directed edges (Items \ref{e2} and \ref{e3}).
\end{proof}

In particular, after some finite time $T$, it suffices to only look at the children of $K$ in $\md_t^*$, since each language $\ga{t}$ identified by the algorithm $\Acc$ will be part of $\md_t^*$. By Claim \ref{claim:finiteTime}, they will be descendants of $K$. 

\begin{cor}\label{cor:valid}
After some finite time \( T \) (which depends on the adversary's input), let \( \ga{t} \) and \( \ga{t+1} \) be the two languages identified by $\Acc$ at times \( t \) and \( t+1 \), respectively. Then, the true language \( K \) is  an ancestor of  \( \ga{t} \) in the tree \( \md_t^* \), and also an ancestor of  \( \ga{t+1} \) in the tree  \( \md_{t+1}^* \).
\end{cor}

The property above is crucial to prove the validity of the algorithm described below.

\vspace{0.1in}

\noindent\textbf{Fallback and Fallback string list. }
We maintain a {\it fallback string list} set $\mS_t$ at each time $t$, which contains the strings that have the priority of being output by the algorithm.
Given the fallback string list at time $t-1$, denoted by $\mS_{t-1}$, which consists of strings ordered from smallest to largest according to the original underlying string ordering (for example, if all languages are subsets of $\mathbb{N}$, then $\mS$ is ordered in increasing order within $\mathbb{N}$). By saying that the algorithm {\it falls back} to some language $\fbga{t}$ at time $t$, we mean the following:

Suppose that at time $t$, the string $w$ is the largest string (with respect to the universal ordering of strings) that has been produced by the input and the adversary up to time $t$. Let $w' \in \fbga{t}$ be the smallest string in $\fbga{t}$ that is larger than $w$.  

We add the set of unused strings in $\fbga{t}$ that are at most $\Succ_{\fbga{t}}(w')$ to $\mS_{t-1}$. Then, we order the strings in $\mS_{t-1}$ from smallest to largest and output. We then output the smallest unused string in $\mS_{t-1} \cup \ga{t}$. Finally, we remove the most recent output from $\mS_{t-1}$ and let the update set be $\mS_t$.

\vspace{0.1in}
\noindent\textbf{Algorithm Description.}

The algorithm works as follows.
\begin{enumerate}
\item If $\ga{t+1} = \ga{t}$, the algorithm continues the previous strategy and does not fall back. 
    \item If $\ga{t+1}$ is a subset of $\ga{t}$ and they are not on the same level, then we {\it fall back} to $\ga{t+1}$. 
    \item Else, we look for the minimum language $\mz$ such that $\mz$ is an ancestor of $\ga{t}$ in $\md_t^*$ and $\mz$ is also an ancestor of $\ga{t+1}$ in $\md_{t+1}^*$.  
    \begin{enumerate}
        \item If there is no $\mz$ such that $\mz$ is an ancestor of $\ga{t}$ in $\md_t^*$ and $\mz$ is also an ancestor of $\ga{t+1}$ in $\md_{t+1}^*$, then continue the previous strategy. 
        \item
If there exists a language \( \mz \) that is both an ancestor of \( \ga{t} \) in \( \md_t^* \) and an ancestor of \( \ga{t+1} \) in \( \md_{t+1}^* \), then find the minimum such \( \mz \) in terms of set inclusion. Note that the minimum must exist because the Cantor-Bendixson rank of the set \(( \mathcal{X}, \mathcal{T} )\) is at most \( r \), meaning that \( \ga{t} \) has at most \( r \) ancestors, and similarly, \( \ga{t+1} \) has at most \( r \) ancestors.

        {\it Fall back} to language $\mz$. 
    \end{enumerate}
\end{enumerate}

Again in this algorithm, once the identified language changes, the goal of the algorithm is to always fall back to to the smallest ``common ancestor" of the two languages, although the underlying trees may have changed. 

Note that Corollary \ref{cor:valid} guarantees that 3 (a) will not happen after some finite amount of time.  Corollary \ref{cor:valid2} below guarantees validity of our algorithm. 

\begin{cor}\label{cor:valid2}

After some finite time \( T \) (which depends on the adversary's input), let \( \ga{t} \) and \( \ga{t+1} \) be the two languages identified by $\Acc$ at times \( t \) and \( t+1 \), respectively. If $\ga{t} \neq \ga{t+1}$, then we will always fall back to a language that is a subset of \( K \), including $K$ itself.
\end{cor} 
\begin{proof}
Since the $\Acc$ will return to \( K \) infinitely often, Item 2 of the algorithm ensures that we will never fall back to a strict superset of \( K \) whenever \( \ga{t} \) is \( K \) after some finite amount of time. By Corollary \ref{cor:valid}, after some finite time \( T \), we will always fall back to a language whenever $\ga{t} \neq \ga{t+1}$ and all the languages that we fall back to will either be subsets of \( K \) or \( K \) itself. Therefore, validity is guaranteed by this new algorithm.
\end{proof}

 We next show that the output has high lower density in $K$. The high-level idea is as follows: at each time stamp $t$, we fall back to some language $ \fbga{t}$. 
 The very next time the adversary outputs an string that is not in $\fbga{t}$ which therefore making $\fbga{t}$ inconsistent, the fallback language at that moment, say $\fbga{t'}$, must climb up to a strictly higher level compared to $\fbga{t}$. Since the rank is finite, the fallback process can only climb up a finite number of times until reaching $K$, which limits the number of times the adversary can consistently output strings that are inconsistent relative to the current fallback language. The proof is intricate because we do not have a static underlying tree, rather, $\mathcal{D}_t^*$ is dynamic and changes over $t$. However, the core intuition outlined above remains valid and is formalized in the proof below.

\vspace{0.2in}

Let $O = \{o_1, o_2, \dots\}$ be the set of strings output by the algorithm where $o_t$ is the output at time $t$. Let $w_t$ be the string generated by the adversary at time $t$. Let $W$ be the set of strings $w_t$ such that $w_t \notin O$, and $o_t > \Succ_K(w_t)$ and $o_{t-1} > \Succ_K(w_t)$. This is considered a bad set because $w_t$ is far away from either the output $o_{t-1}$ and $o_t$. 
\begin{claim}\label{claim:climbup}
    In the end, there exists some $m$ such that after the $m$-th string in $K$, there cannot be  $r+1$ consecutive strings in $K$ that are in $W$. 
\end{claim}
\begin{proof}
Suppose we are after the finite time  $T$ guaranteed by Corollary \ref{cor:valid}. Let \( I \) be an interval in $W$, consisting of consecutive strings in $K$ that are produced after $T$. Suppose that all these strings are occupied by the adversary and are never output by the algorithm. This implies that the adversary occupies some strings within the interval, and the next available strings in the language fall outside \( I \).

Suppose the chronological sequence of strings generated by the adversary among the strings in \( I \) is \( b_1, b_2, \dots \), which may not correspond to consecutive time stamps. Let $t_i$ be the time at which $b_i$ is generated by the adversary. Thus, we have $t_1 < t_2 < \dots$.

After the adversary occupies \( b_1 \) at time \( t_1 \), the identified language is $\ga{t_1}$ by $\Acc$. We claim that $\ga{t_1}$ does not contain any of the strings in \( I \setminus \{b_1\} \). Suppose not, and let $x$ be the smallest string in $I \setminus \{b_1\}$ that is contained in $\ga{t_1}$. Since the algorithm always output the smallest unused string in $\mS_{t_1-1} \cup \ga{t_1}$, we must have $o_{t_1} \leq x$. Since $I$ is an interval in $K$, it means $x \leq \Succ_K(b_1)$. Therefore $o_{t_1} \leq \Succ_K(b_1)$, contradicts with the fact that $b_1 \in W$. 

Furthermore, $\ga{t_1} \neq \ga{t_1 -1}$. As otherwise, by a similar argument, $o_{t_1-1} \leq w_{t_1}$,  contradicts with the fact that $b_1=w_{t_1} \in W$.  Therefore at time $t_1$, the algorithm falls back to a superset language $\fbga{t_1}$. 

We now prove that \( \fbga{t_1} \) does not contain any of the strings in \( I \setminus \{b_1\} \). We prove this claim by contradiction. Let $x$ be the smallest (with respect to the string ordering in $K$) string in $I \setminus \{b_1\}$ that is contained in $\fbga{t_1}$. Since $I$ is an interval, $x \leq \Succ_K(b_1)$. Thus the string $x$ will be added to the fallback string list $\mS_{t-1}$. Therefore before the time that $x$ is used either by the adversary or the algorithm, the algorithm should have always output a string that is less than $x$ by our falling back rule and how the fallback string list works. Since $x \in I$, eventually $x$ is generated by the adversary at some time $t_i > t_1$. However this means $o_{t_i-1} < x$. This contradicts with the fact that $x = w_{t_i} \in W$.  Thus we have proved that $\fbga{t_1}$ does not contain any of the strings in \( I \setminus \{b_1\} \).

At time \( t_1 \), the language identified by the algorithm, \( \ga{t_1} \), is a descendant of \( \fbga{t_1} \) in \( \md_{t_1}^* \) (since we fall back to it). Also notice that as $t$ progresses, $\fbga{t_1}$ will remain critical until it becomes inconsistent by Claim \ref{claim:consist}. However, at time \( t_2 \), the language \( \ga{t_2} \) is not a descendant of \( \fbga{t_1} \) in \( \md_{t_2}^* \) because \( \fbga{t_1} \) does not contain \( b_2 \), but \( \ga{t_2} \) does. Thus at time $t_2$, the fallback language $\fbga{t_1}$ is not consistent anymore.

Therefore, there must exist a time \( t_1 \leq t_1' < t_2 \) such that \( \ga{t_1'} \) is a descendant of \( \fbga{t_1} \) in \( \md_{t_1'}^* \) at time \( t_1' \), but \( {\ga{t_1'+1}} \) is not a descendant of \( \fbga{t_1} \) in \( \md_{t_1'+1}^* \) at time \( t_1'+1 \).

At time $t_1'+1$, the fallback language $\fbga{t_1'+1}$ should be an ancestor of $\ga{t_1'}$ in  the tree $\mathcal{D}_{t_1'}^* $. 
% as well as an ancestor of $\ga{t_1'+1}$ in  the tree $\mathcal{D}_{t_1'+1}^*$. 
Since $\ga{t_1'}$ is a descendant of $\fbga{t_1}$ in $\mathcal{D}_{t_1'}^* $, it means that $\fbga{t_1'+1}$ is either a subset and descendant of $\fbga{t_1}$ or a superset and ancestor of $\fbga{t_1}$.
Note that $\fbga{t_1'+1}$ cannot be a subset and descendant of $\fbga{t_1}$, since $\ga{t_1'+1}$ is not a descendant of $\fbga{t_1}$ in $\mathcal{D}_{t_1'+1}^*$.  Therefore $\fbga{t_1'+1}$ is a strict superset and ancestor of $\fbga{t_1}$. 
% By how we design the fallback mechanism, $\fbga{t_1'+1}$ is also an ancestor of $\ga{t_1'}$ in  the tree $\mathcal{D}_{t_1'}^* $. Since $\fbga{t_1}$ is also an ancestor of $\ga{t_1'}$ in  the tree $\mathcal{D}_{t_1'}^* $, it means that in the tree $\mathcal{D}_{t_1'}^* $, the language  $\fbga{t_1'+1}$ is an ancestor of the language $\fbga{t_1}$. This implies that  \( \fbga{t_1'+1} \) is at a strictly higher level than \( \fbga{t_1} \). 

We can now continue the argument by replacing \( t_1 \) with \( t_1' \) and \( t_2 \) with \( t_3 \). We claim that \( \fbga{t_1'} \) does not contain \( I 
\setminus \{b_1, b_2\}\) by a similar argument.
If not, suppose $\fbga{t_1'}$ contains some strings in \( I 
\setminus \{b_1, b_2\}\). Let $x$ be the smallest such string. Since $b_1$ has already been taken by the adversary by now, the fallback string list should now include $x$.  Then by the definition of how the fallback works, since $x$ is occupied by the adversary first, it means up to the time when $x$ is taken by the adversary at time $t_i$, the algorithm is always outputting elements that are smaller than $x$. In particular, it means that the output by the algorithm right before the adversary generates $x$, which is $o_{t_i-1}$, is smaller than $x = w_{t_i}$. This contradicts with the fact that $x \in W$.  

So we have that \( \fbga{t_1'} \) does not contain \( b_3 \).
But at time \( t_3 \), the language identified by the algorithm $\Acc$, \( \ga{t_3} \), contains \( b_3 \). The same argument applies, leading to the conclusion that there exists \( t_1' \leq t_2' < t_3 \) such that  \( \ga{t_2'} \) is a descendant of \( \fbga{t_1'} \) in \( \md_{t_2'}^* \) at time \( t_2' \) where \( \fbga{t_2'} \) is the language fallback to at time $t_2'$, but \( {\ga{t_2'+1}} \) is not a descendant of \( \fbga{t_1'} \) in \( \md_{t_2'+1}^* \) at time \( t_2'+1 \).
By a similar argument, it implies that 
 \( \fbga{t_2'} \), has a strictly higher level than \( \fbga{t_1'} \).

This argument can be repeated, finding a \( \fbga{t_i'} \) such that $t_{i-1}' \leq t_i' < t_{i+1}$ at time $t_i'$ whose level is strictly higher than that of \( \fbga{t_{i-1}'}\). This process must stop in at most \( r - 1 \) steps, since the rank of \(\mX \) is bounded by \( r \). Thus, the adversary can output at most \( r \) consecutive strings that are not occupied by the algorithm.
\end{proof}
We now bound the lower density of $O$ in $K$, by considering the upper density of $K \setminus O$ in $K$. 

Notice that $K \setminus O$ is the disjoint union of $M$ and $(K \setminus O) \setminus M$, we will upper bound their upper densities in $K$ separately. 
For $M$, since there cannot be $r$ consecutive strings in $K$ that all belong to $M$ in the limit, we have that the upper density of $M$ in $K$ is at most $r/(r+1)$. This implies $\dlow(O \cup (K \setminus O \setminus M), K) \geq 1/(r+1)$. 

Let $Y_1 = \{ w_t \notin O: w_t \in (K \setminus O) \setminus M,  o_t \leq \Succ_K(w_t)\}$, and let $Y_2 = \{ w_t \notin O: w_t \in (K \setminus O) \setminus M,  o_{t-1} \leq \Succ_K(w_t)\}$. Thus $Y_1 \cup Y_2 = K \setminus O \setminus M$. For any $N$ sufficiently large, let $[N]_K$ to denote the first $N$ strings in $K$. Then $\dor( Y_1 \cap [N]_K, K) \leq \dor(O \cap [N]_K, K ) + o_N(1)$, and $\dor( Y_2 \cap [N]_K, K) \leq \dor(O \cap [N]_K, K ) + o_N(1)$. 
Together with the fact that $\dlow(O \cup Y_1 \cup Y_2, K) \geq 1/(r+1)$, we have $\dor(O\cap [N]_K, K) \geq 1/(3(r+1)) + o_N(1)$. This implies $\dlow(O, K) \geq 1/(3(r+1))$. 
\end{proof}
Even though this theorem addresses the special case where the Cantor-Bendixson rank of $(\mX, \mathcal{T})$ is finite and the perfect kernel is empty, it nonetheless illustrates how to handle situations where the underlying tree $\md_t^*$ evolves dynamically over time. This dynamic behavior presents one of the key technical challenges when analyzing the algorithm with fallback, in contrast to Examples \ref{example:2} and \ref{example:3}, where the trees remain static.

\subsection{Infinite rank case}
\label{subsec:element:lower:infinite}
In this section we prove the general version of the theorem, incorporating additional new ideas to ensure that the algorithm's output set $O(E,\ma)$ has large lower density even when $\mX$ has infinite rank or when the perfect kernel is non-empty under the topology we have defined.  In this way, the result applies to all instances of the language generation problem.
We state the result as follows; it is equivalent to Theorem \ref{thm:element-lower-intro} that we previewed in Section \ref{sec:overview}.

\begin{thm}\label{thm:infRankHighLD}
    There is an algorithm $\ma$ that achieves element-based generation in the limit and has the following property. Given any countable collection of languages $\coll$ with  an underlying ordering of the strings in all the languages, and for any adversarial enumeration $E$ of one of the languages $\trueL \in \coll$, the set of output strings $O(E,\ma)$ generated by the algorithm has a lower density in $K$ that is at least $1/8$. 
\end{thm}

To get started with the proof of Theorem \ref{thm:infRankHighLD}, we begin by observing that the topological space $(\coll,\mathcal{T})$ is also equipped with a partial order determined simply by the subset relation on the languages in $\coll$.
This makes it a {\em partially ordered space}, or {\em pospace} \cite{gierz-pospace}. The Szpilrajn extension theorem  \cite{totalorder} says that every partial order can be extended to a total (linear) order. In the case for the countable pospace, this linear order is isomorphic to a subset of $\mathbb{Q}$. In other words, there is a mapping $\ell: \mX \to \mathbb{Q}$ such that $\ell(\mx) < \ell(\my)$ if $\mx$ is a strict subset of $\my$. This mapping $\ell(\mx)$ will be the ``level" of language $\mx$. Clearly the property in Claim \ref{claim:nosamelevel} is preserved under $\ell$, which is the property that is needed for our algorithm to work.

Notice that any $\coll$ with finite rank and empty perfect kernel still has a level $\ell$ which is mapped to $\mathcal{Q}$. This way of defining $\ell$ will make this set of languages have infinite number of levels. At a hindsight, a linear ordering of this nature has a significant drawback for our purposes and thus may not sound ideal. For instance, in Figure \ref{fig:tree2},  the language sequence  such as $\lang{1}, \lang{2}, \dots$ will not share the same level under $\ell$; instead, under $\ell$ they form an increasing chain in the poset. Similarly, $\lang{2:1}^{(1)}, \lang{2:2}^{(1)}, \lang{2:3}^{(1)}, \dots$ also do not have the same level but rather form an increasing chain in the poset. The most undesirable scenario occurs when we climb up a very deep chain. For example, when transitioning from $\lang{2:i}^{(1)}$ to $\lang{2:i+1}^{(1)}$, we only fallback to $\lang{2:i+1}^{(1)}$, and so on. This type of increasing chain of languages, where each step only falls back to the larger set of the two, poses problems for us. Since the chain could be arbitrarily long in this infinite poset,
 we cannot get a good estimate as in Theorem \ref{thm:finiteRankHighLD},  where the lower density guarantee is  inversely proportional to the height of the chain. In fact, if we only use the algorithm from Theorem \ref{thm:finiteRankHighLD}, an adversary can enumerate $K = \mathbb{N}$ in such a way that the the output $O$ is missing arbitrarily long intervals in $K$, precisely due to the existence of arbitrarily long chains. This could result in small lower density by appropriately adjusting the example.

One idea to break these long chains is by collapsing languages which are similar to one single level.
For example, one could force languages with density above $1-\epsilon$
 relative to each other to be on the same level. This will significantly shorten the number of levels to $\log_{1/\epsilon} \delta$ if we stop collapsing when bottom level has density just goes below $\delta$ in \(K\). However, a challenge with this approach is that ensuring validity is difficult. Since it is likely that the true language $K$ will be on the same level as many of its relatively dense subsets, the following issue arises. Let $\ga{t}$ and $\ga{t+1}$ denote the languages identified by the algorithm $\Acc$ at times $t$ and $t+1$, respectively. If we fall back to a higher level whenever $\ga{t}$ and $\ga{t+1}$ differ but are on the same level, we cannot guarantee validity. This is because once we reach the level of $K$, we would fall back to a strict superset of $K$. However, if we do not fall back to a higher level when $\ga{t}\neq \ga{t+1}$ are on the same level, we will again face the same issue of an arbitrarily long increasing chain.

The new idea in this section leverages the flexibility in the number of strings in the language we fallback to that we can store in the fallback string list $\mS_t$. Intuitively, by the previous argument, the issue of a low-density interval of arbitrarily long length, not occupied by the algorithm's output $O$, arises when we ascend a long increasing chain. The length of this interval is directly related to the height of the chain.  

However, before ascending in the chain, there must be a point where we drop from a much higher level, as $\Acc$ guarantees that we return to $K$ infinitely often (where $K$ is at a high level). Once we descend, we know the number of levels we drop, and consequently, how much we can ascend in the future. Thus, we will charge the fallback string list $\mS_t$ 
 as we descend by increasing the size of the strings in the fallback string list  $\mS_t$  (which we will refer to as a ``token'' below for the extra number of strings to be charged in $\mS_t$).  

The high level idea is that this charging ensures that even if we encounter an arbitrarily long interval in $K$ that is disjoint from $O$, we can guarantee that before such an interval occurs, there are ``sufficiently" many outputs prior to this interval in the ordering of $K$, and thus the ordered density of $O$ is high in $K$.

\vspace{0.1in}
\vspace{0.1in}
\noindent\textbf{Dynamic trees $\md_t$. }
Recall the function $\ell$ defines a total ordering of the set of languages, i.e., a language with smaller level value is always below the language with a larger level value. 
We will again construct a dynamic forest \( \md_t \) and the  dynamic tree \( \md_t^* \) at each step $t$. 

At each time \( t \), after the adversary generates an string $w_t$, we will have a sequence of strictly critical sets, and the language \( \ga{t} \) identified by $\Acc$ will be among them. We will also eliminate the languages that are no longer consistent. We construct a forest \( \mathcal{D}_t \) on the set of remaining strings as follows, similar to before:

For each language $\lan$ (which may not necessarily be critical), we draw a directed edge from $\lan$ to the minimum strictly critical language $\tilde{\lan}$ that strictly contains $\lan$ with $\ell(\tilde \lan) > \ell(\lan)$, if such a language exists. More specifically:

\begin{enumerate}
    \item\label{e1inf} If there is no strictly critical language $\lan'$ that strictly contains $\lan$ and $\ell(\lan') > \ell(\lan)$, do not add a directed edge from $\lan$ (so $\lan$ is a root).
    \item\label{e2inf} If there are strictly critical languages $\lan'$ that strictly contain $\lan$ and  $\ell(\lan') > \ell(\lan)$, and a minimum (under set inclusion) exists among them, add a directed edge from $\lan$ to this minimum.
    \item\label{e3inf} If no minimum (under set inclusion) exists, it means that there is an infinite descending chain (under set inclusion) of strictly critical languages  with level higher than that of $\lan$. In this case, we pick the \( t \)-th language (from largest set to smallest set under set inclusion) and add a directed edge from $\lan$ to this language.
\end{enumerate}
The above defines a directed graph $\md_t$.
Since each vertex has outdegree at most one, and if it is exactly one then the directed edge is pointing upward, $\md_t$ is a forest. Let \( \md_t^* \) denote the connected component of the forest containing \( \ga{t} \). In this way, we have constructed a forest where there is a directed edge pointing from each child to its parent.

\vspace{0.1in}
\noindent\textbf{Fallback, fallback string list, and token.}
 Again let $\ga{t}$ be the language identified by the algorithm $\Acc$ at time $t$. 
The algorithm works similarly to before, except that we modify the fallback rule. 

When we fall back to $\fbga{t}$ at time~$t$, we assign a \emph{token} indicating the extent to which unused strings in $\fbga{t}$ will be added to the fallback string list~$\mS_{t-1}$. Recall that, similar to the finite-rank case, we maintain a \emph{fallback string list} $\mS_t$ at each time~$t$, which contains the strings that have the priority to be output by the algorithm.

By saying that the algorithm falls back to some language $\fbga{t}$ at time $t$ with {\it token} $N_t$, we mean the following:

Suppose that at time $t$, the string $w$ is the largest string (with respect to the universal ordering of strings) that has been produced by the input and the adversary up to time $t$. Let $w' \in \fbga{t}$ be the smallest string in $\fbga{t}$ that is larger than $w$.  
We add to the set $\mS_{t-1}$ all unused strings in $\fbga{t}$ that are at most $w'$. Furthermore, we add the next $N_t$ smallest  strings in $\fbga{t}$ that are larger than $w'$ to $\mS_{t-1}$. (Note these $N_t$ strings have not been used by either the algorithm nor the adversary yet).  
Next, we order the strings in the revised set $\mS_{t-1}$ from smallest to largest. The output is the first unused string in $\mS_{t-1}$. Finally, we update $\mS_t$ by removing the most recently output string from $\mS_{t-1}$.

\vspace{0.1in}
\noindent\textbf{Algorithm.}
\begin{enumerate}
\item Initialize the fallback list set $\mS_0 = \{\emptyset \}$. Set token $N_0 = 0$.
\item If $\ga{t+1} = \ga{t}$, do not fall back to any language. If $\mS_t \neq \emptyset$ and the smallest unused string in $\mS_t$ is smaller than the next available string in $\ga{t}$, output the smallest unused string in $\mS_t$, and update $\mS_t$ to be $\mS_{t+1}$ by removing this string and all the strings already used by the adversary. Otherwise, output the smallest unused string in $\ga{t}$ and set $\mS_{t+1}$ to be $ \mS_t$ after removal the strings already generated by the adversary. 
    \item If $\ga{t+1}$ is a strict subset of $\ga{t}$. We fall back to $\ga{t}$. Since $\ga{t+1}$ and $\ga{t}$ are both strictly critical at the moment, they both belong to the descending chain of strictly critical languages at time $t$, which is $\lanj{1} \supsetneq \lanj{2} \supsetneq \lanj{3} \supsetneq \dots$. Suppose $\ga{t}$ is the set $\lanj{i}$ and $\ga{t+1}$ is set $\lanj{j}$. Set the token $N_{t+1} = 2(j-i)$.

\item If $\ga{t+1}$ is a strict superset of $\ga{t}$. Then we fall back to $\ga{t+1}$. Set the token $N_{t+1} = 2$. 

\item \label{enum:alginf5} Else, we look for the minimum language $\mz$ such that $\mz$ is an ancestor of $\ga{t}$ in $\md_t^*$ and $\mz$ is also an ancestor of $\ga{t+1}$ in $\md_{t+1}^*$.  

    \begin{enumerate}
        \item If there is no $\mz$ such that $\mz$ is an ancestor of $\ga{t}$ in $\md_t^*$ and $\mz$ is also an ancestor of $\ga{t+1}$ in $\md_{t+1}^*$, then do not fall back. If $\mS_t \neq \emptyset$ and the smallest unused string in $\mS_t$ is smaller than the next available string in $\ga{t}$, output the smallest unused string in $\mS_t$, and update $\mS_t$ after removing this  string and the strings used up by the adversary, denote the resulting set as $\mS_{t+1}$. Otherwise, output the smallest unused string in $\ga{t}$ and set $\mS_{t+1} = \mS_t$ after removing strings that have been used by the adversary. 
        \item
If there exists a language \( \mz \) that is both an ancestor of \( \ga{t} \) in \( \md_t^* \) and an ancestor of \( \ga{t+1} \) in \( \md_{t+1}^* \), then find the minimum such \( \mz \) in terms of set inclusion. 

\begin{enumerate}
    \item 
If this minimum exists, then call this $\mz = \fbga{t+1}$ to be the language we fall back to. 
\item If the minimum does not exist. Since the strictly critical languages  at time $t+1$ form a descending chain under set inclusion, we pick the $(t+1)$-th largest one in the descending chain (under set inclusion) of strictly critical set. Set this language to be $\fbga{t+1}$.
\end{enumerate}
In both cases above, notice that the fallback language $\fbga{t+1}$ is also strictly critical at time $t+1$. Let the descending chain of strictly critical languages at time $t+1$ be $\lanj{1} \supsetneq \lanj{2} \supsetneq \dots$. 
Suppose $\fbga{t+1} = \lanj{i}$ and $\ga{t+1} = \lanj{j}$. 
Set the token $N_{t+1} = 2(j-i)$. 
    \end{enumerate}
\end{enumerate}

Note that our linear ordering guarantee that there is no strict subset of $K$ with level higher than $K$ itself with respect to $\ell$. 
\begin{claim}\label{claim:finitetimesubsetK}
     After some finite time $T$, which could depend on the adversary's enumeration of strings in $K$, any language identified by $\Acc$ will have $K$ as an ancestor in $\md_t^*$.
\end{claim}
\begin{proof}
We first assume that we are past the finite time after which $K$ is always strictly critical. Additionally, we assume that $T$ is greater than the position number of $K$ in the original language listing. This ensures that for any descending chain of strictly critical languages, $K$ is always among the first $T$ of them. Consequently, whenever we encounter Item \ref{e3inf}, the fallback language will be a subset of $K$.

After $T$, we can assume that the language \( \ga{t} \) identified by the algorithm $\Acc$ at time \( t \) is a subset of \( K \). By the definition of total linear order, \( \ell(\ga{t}) < \ell(K).\) 

We show that any strict subset $\lan$ of $K$, if consistent at $t$, should has a direct edge towards a strictly critical language that is a subset of $K$. By the definition of the linear ordering $\ell$ and our construction, we know $K$ has higher level than any of its strict subset. 
  So $\lan$ in $\md_t$ must has outdegree at least one, since connecting to \( K \) is one option (by Item \ref{e1inf}). By the minimality of the directed edges (Items \ref{e2inf} and \ref{e3inf}), and our choice of $T$, it must has a directed edge towards some strict subset of \( K \). 

  We are left to show that for any strictly critical language that is a subset of $K$, $K$ is its ancestor in $\mathcal{D}_t^*$. Suppose the sequence of strictly critical sets following $K$ in the ordering is $\lanj{1} \supsetneq \lanj{2} \supsetneq J_3 \dots$. We will show that for all $i \geq 1$, $\lanj{i+1}$ has a directed edge pointing to $\lanj{i}$. Since each directed edge must point to a strictly critical set at time $t$, and by the minimality condition (Item \ref{e2inf}), the minimum exists and it is $\lanj{i}$. Therefore, there is a directed edge from $\lanj{i+1}$ to $\lanj{i}$. Combining this with the previous result, we reach the conclusion of the claim.
\end{proof}

This claim implies the validity of the algorithm, as stated in the following corollary.
\begin{cor}\label{cor:validityinfinite}
There exists a finite time $T^*$, which may depend on the adversary's enumeration of strings in $K$, such that after $T^*$, the algorithm always outputs an string in $K$.
\end{cor}
\begin{proof}
We assume that we are past the finite time guaranteed to exist by the conclusion of Claim \ref{claim:finitetimesubsetK}. Then, we only need to verify that we always fall back to a language that is subset of $K$. This is guaranteed by Item \ref{enum:alginf5} in the algorithm description and by Claim \ref{claim:finitetimesubsetK}.
\end{proof}

Again, at time $t$, let $w_t$ be the string generated by the adversary. 
Let $o_t$ be the string output by the algorithm at time $t$. Define $O = \bigcup_{t=1}^\infty \{o_t\}$ as the set of strings output by the algorithm in the end. Clearly, $O \subsetneq K$ since the algorithm cannot repeat a string already generated by the adversary.

Again without loss of generality, we could assume the ground set is the set of natural numbers $\mathbb{N}$. Then $K$ is a subset of $\mathbb{N}$. $K$ consists of the set of ordered integers $\{\psi(1), \psi(2), \dots\}$ where the $i$-th string in $K$ is $\psi(i)$. For each string $x \in K$, recall $\Succ_K(x)$ is defined to be the next string in $K$ following immediately after $x$. i.e., $\Succ_K(x) = \psi( \psi^{-1}(x)+1)$.

Theorem \ref{thm:infRankHighLD} is equivalent to the following proposition. 
\begin{prop}\label{lem:infLD}
The lower density of $O$ in $K$ is at least $1/8$. 
\end{prop}
The rest of the section is dedicated to prove Proposition \ref{lem:infLD}.
\begin{proof}[Proof of Proposition \ref{lem:infLD}]
Let $\mB_g$ be the set of ``good" strings in $K \setminus O$, defined as \[\mB_g = \{ w_t \in K \setminus O: o_t \leq \Succ_K(w_t) \text{ or } o_{t-1} \leq \Succ_K(w_t) \}.\]
This set consists of adversary inputs for which the algorithm immediately outputs an string that is not far behind or the previous output is not far behind. 

Let $\mB_b$ be the set of ``bad" strings, defined as \[\mB_b = K \setminus O \setminus \mB_g.\] Clearly, $\mB_b$ consists of a disjoint union of intervals (i.e., consecutive strings in $K$) based on the string ordering in $K$. Let $\mB_2$ be the collection of these intervals in $\mB_b$ with a length of at least two. Let $\mB_1$ be the collection of singletons in $\mB_b$ relative to $K$. For example, if $K = \{1,2,\dots, 11\}$, and $\mB_b= \{2,3,4,5\} \cup \{7\} \cup \{9,10\}$, which forms a disjoint union of three intervals. Then $\mB_2 = \{2, 3, 4, 5\} \cup \{ 9, 10\}$, comprising the union of intervals with lengths of at least two, and $\mB_1 = \{7\}$. In another example where $K = \{1,3,5\dots, 11\}$ and $\mB_b= \{3,5\}  \cup \{9\}$, which forms a disjoint union of two intervals, then $\mB_2 = \{3,5\}$, comprising of one interval with lengths of at least two, and $\mB_1 = \{7\}$. 

The next lemma is the main lemma in the proof. 

\begin{lem}\label{claim:infrho}
Given any enumeration of $K$ by the adversary, there exists a way to remove at most one string from each maximal interval in $\mB_2$, resulting in $\mB_2'$, such that the following property holds for $\mB_2'$:
There exists a finite time $T$ and a map $\rho: \mB_2' \to O$ such that for all $t > T$, and for any $t$ where $w_t \in \mB_2'$, we have $\rho(w_t) = o_{t'}$ for some $t' < t$, with $o_{t'} < w_t$ in the string ordering of $K$. In addition, $\rho$ is injective. 
\end{lem}  

We first show that Lemma \ref{claim:infrho}  implies Proposition \ref{lem:infLD}. 

Let $N$ be a positive integer. We will use $[N]_K$ to denote the set of the first $N$ strings of $K$, i.e., $[N]_K = \{\psi(1), \psi(2), \dots, \psi(N)\}$. Suppose the density of $\mB_2$ in $[N]_K$ is $\alpha$ (note that $\alpha$ can depend on $N$). We will use  two methods to derive two lower bounds on the density of $O$ in $[N]_K$, expressed in terms of $\alpha$. Together, these two bounds will imply the conclusion of Proposition \ref{lem:infLD}.

Observe that any string in the set difference $K \setminus \mathcal{B}_2$ must belong to one of three mutually exclusive sets: $O$, $\mathcal{B}_g$, or $\mathcal{B}_1$. Consequently, the density of $O \cup \mathcal{B}_g \cup \mathcal{B}_1$ in $[N]_K$ is exactly $1 - \alpha$.

The set $O \cup \mathcal{B}_g \cup \mathcal{B}_1$ forms disjoint intervals in $[N]_K$. The complement of this set in $[N]_K$ consists of intervals formed by strings in $\mathcal{B}_2$. By the definition of $\mathcal{B}_1$, the following hold. 
No string in $\mathcal{B}_1$ can be at an endpoint of these intervals, except possibly the last one, since no string in $\mathcal{B}_1$ can be adjacent to a string in $\mathcal{B}_2$. Furthermore no two strings in $\mathcal{B}_1$ can be adjacent in $[N]_K$. 
Therefore, the density of $\mathcal{B}_1$ in $[N]_K$ is at most $({1 - \alpha})/2 + o_N(1)$. It follows that the density of strings in $O \cup \mathcal{B}_g$ in $[N]_K$ is at least 
$(1 - \alpha)/{2} + o_N(1)$.

We first use one method to lower bound the density of $O$ in $[N]_K$. By the definition of $\mB_g$, there exists a mapping $g: \mB_g \to O$ such that $g(b) \leq \Succ_K(b)$, and $|g^{-1}(o)|\leq 2$. Therefore
\[
|O \cap [N+1]_K| \geq |\{ g(b): b \in \mB_g \cap [N]_K\}| \geq |\mB_g \cap [N]_K|/2. 
\]
Therefore $2|O \cap [N]_K| + 2  \geq |\mB_g \cap [N]_K|$. Note that the set $O$ is disjoint from the set $\mB_g$.  Together with the fact that $|O \cap [N]_K| + |\mB_g \cap [N]_K| \geq ((1 - \alpha)/{2} + o_N(1)) N$, we have the first lower bound
\begin{equation}
    |O \cap [N]_K| \geq ((1 - \alpha)/{6} + o_N(1)) N. \label{eq:b1}
\end{equation}

We now use an alternative method to bound the density of $O$ in $[N]_K$ in terms of $\alpha$. By the definition of $\mB_2$, each maximal interval $I$ in $\mathcal{B}_2 \cap [N]_K$ with respect to $[N]_K$ has a length of at least two, except possibly for the last interval. This implies that by removing one string from each interval $I$, the density of $\mathcal{B}_2'$ in $[N]_K$ is at least
\[
\alpha \max_{I: |I| \geq 2} ({|I|-1})/{|I|} + o_N(1) \geq \alpha/{2} + o_N(1).
\]
Let $M$ be the position of the largest string output by the adversary by some finite time $T_0$, where $T_0$ is given by  Corollary \ref{cor:validityinfinite}. 
By the definition of $\rho$ and the fact that $|\rho^{-1}(o)| \leq 1$, there are at least 
\[ (|\mB_2'| - M) \geq  (\alpha/2 + o_N(1))N - M\] strings from $O$ among $[N]_K$. Since $M$ only depends on $T_0$
 and is independent from $N$, we conclude that the density of $O$  in $[N]_K$ satisfies
 \begin{equation}
     |O \cap [N]_K| \geq (\alpha/{2} + o_N(1)) N. \label{eq:b2}
 \end{equation}

Combining the two bounds in (\ref{eq:b1}) and (\ref{eq:b2}), we conclude that the density of $O$ in $[N]_K$ is at least
\[ \max\left( ( 1-\alpha)/6 + o_N(1), \alpha/2 + o_N(1) \right) \geq 1/8.\]
The equality is achieved when $\alpha$ satisfies the equation $(1-\alpha)/6 = \alpha/2$, which yields $\alpha = 1/4$.
Therefore, we conclude that the lower bound on the density of $O$ in $K$ is at least $1/8$. This completes the proof of Proposition \ref{lem:infLD}.

To complete the proof of Proposition \ref{lem:infLD}, it now suffices to prove Lemma \ref{claim:infrho}.

 \begin{proof}[Proof of Lemma \ref{claim:infrho}]
We can assume that we are after some finite time $T_0$, after which $K$ remains strictly critical, and all strings in the fallback string list $\mathcal{S}$ which are not in $K$ have been exhausted. Such a $T_0$ exists because, by Claim \ref{claim:finitetimesubsetK}, after some finite time, we only revisit languages that are subsets of $K$. Since $K$ will be visited infinitely often, we can select a time $T$ that is both after $T_0$ and after the second return to $K$ after $T_0$.

Let $N$ be the index of the largest string in $K$ that has been output by either the adversary or the algorithm up to and including time $T+1$. In other words, $\{o_1, \dots, o_{T+1}\} \cup \{w_1, \dots, w_{T+1}\} \subset \{\psi(1), \dots, \psi(N)\} \subset K$.
Let $I$ be any  maximal interval in $\mB_2$ with respect to $K$ consisting solely of strings $w_t$ for $t > T$, and not containing any of the first $N$ strings in $K$.

Let $I = [a_1, a_1+1, a_1+2, \dots]$. By the definition of $\mathcal{B}_2$, each string in $I$ is used exclusively by the adversary. Suppose the sequence of strings from $I$ output by the adversary, in chronological order, is $b_1, b_2, \dots$, which may not correspond to consecutive time stamps. Suppose $b_i$ is generated by the adversary at time $t_i$. Thus $t_1 < t_2 < \dots$.

Define $\mB_2'$ restricted to $I$ be $\{b_2, b_3, \dots\}$. In other words, $\mB_2' \cap I$ is the set of strings in $I$ that are  generated  by the adversary excluding the first one that is generated. 
 For each $i \geq 2$ where $b_i$ is defined, let $\theta(t_i)$ be the most recent time $t < t_i$ at which the following conditions hold simultaneously.
\begin{enumerate}
    \item[(a)] $\fbga{t} \supset \ga{t_i}$ or $\fbga{t} = \ga{t_i}$
    \item[(b)] \label{item:b} $\ga{t}$ is not a superset of $\ga{t_i}$ and $\ga{t} \neq \ga{t_i}$. 
\end{enumerate}
We will show that such $\theta(t_i)$ exists in Claim \ref{claim:theta}. 

Assume that $\theta(t_i)$ exists. 
Notice that at time $\theta(t_i)$, the language $\ga{\theta(t_i)}$ is strictly critical,  and at time $t_i$, $\ga{t_i}$ is strictly critical. 
Let $F(t_i)$ be the smallest language, under set inclusion, that is strictly critical at both times $\theta(t_i)$ and $t_i$, and such that $\ga{\theta(t_i)}\subset F(t_i)$  and $\ga{t_i} \subset F(t_i)$. Therefore 
\begin{equation}
    \ga{\theta(t_i)} \subset F(t_i) \subset \fbga{\theta(t_i)} \label{eq:Ft1}
\end{equation}
and 
\begin{equation}
    \ga{t_i} \subset F(t_i) \subset \fbga{\theta(t_i)}. \label{eq:Ft2}
\end{equation}
The fact that $F(t_i) \subset \fbga{\theta(t_i)}$ is by the minimality in the definition of $F(t_i)$, since  $\fbga{\theta(t_i)}$ is an option that satisfies both (\ref{eq:Ft1}) and (\ref{eq:Ft2}), and there are only finitely many candidate languages. In particular, this means $F(t_i)$ exists. By Item (b) above, (\ref{eq:Ft1}) can be written as 
\begin{equation}
    \ga{\theta(t_1)} \subsetneq F(t_i) \subset \fbga{\theta(t_1)}. \label{eq:Ft3}
\end{equation}

After defining $F(t_i)$, we define $F'(t_i)$ as the strictly critical set at time $\theta(t_i)$ that immediately follows $F(t_i)$ in the descending chain of strictly critical languages at time $\theta(t_i)$. In other words, suppose that at time $\tti$, the set of strictly critical languages forms a descending chain $\lanj{1} \supsetneq \lanj{2} \supsetneq \dots$. If $F(t_i) = \lanj{s}$ for some positive integer $s$, then $F'(t_i) = \lanj{s+1}$.

  By (\ref{eq:Ft3}), at time $\theta({t_i})$, the token value $N_{\theta(t_i)} \geq 2$. In particular, $N_{\theta(t_i)}$ is at least twice the distance in the strictly critical set descending chain at time $\theta(t_i)$ between $\fbga{\theta(t_i)}$ and $\ga{\theta(t_i)}$. In other words, if $\fbga{\theta(t_i)} = \lanj{s}$ and $\ga{\theta(t_i)} = \lanj{r}$, then $N_{\theta(t_i)} = 2(r-s)$. 

  One could interpret the count $N_{\theta(t_i)}$ as charging $\mS_{\theta(t_i)}$ with two elements for each step down the descending chain of strictly critical languages, from $\fbga{\theta(t_i)}$ to $\ga{\theta(t_i)}$. Using this interpretation, let $\Phi(b_i)$ denote the two strings charged as we move down one step along the descending chain from $F(t_i)$ to $F'(t_i)$ at time $\theta(t_i)$. Thus $\Phi(b_i)$ is determined by $\theta(t_i)$ and $F(t_i)$. 
  
Note that both sets $F(t_i)$ and $F'(t_i)$ are strictly critical languages at time $\theta(t_i)$, and both are positioned between (inclusive) $\ga{\theta(t_1)}$ and $\fbga{\theta(t_1)}$ in the descending chain of strictly critical languages at time $\tti$, as per (\ref{eq:Ft3}). Therefore, $\Phi(b_i)$ is well-defined if $\theta(t_i)$ exists.

  \begin{claim}\label{claim:theta}
      $\theta(t_i)$ exists. 
  \end{claim}
\begin{proof}
    Here we use the fact that $i \geq 2$, i.e., $b_i$ is not the first string generated by the adversary among the strings in $I$. 
    By a similar reasoning as in the proof of Claim \ref{claim:climbup} in Theorem \ref{thm:finiteRankHighLD}, using the fact that $I$ is disjoint from $\mB_g$, there must exists $t_1 \leq t_{i-1}' < t_i$ such that $\fbga{t_{i-1}'}$ does not contain $b_{i}$. 
    Together with the fact that $b_i \in \ga{\ti}$, Item (b) is satisfied with $t$ replaced by $t_{i-1}'$.  

    On the other hand, by our assumption on the finite cutoff time $T$, we know there is a time $ t_K < t_1$, such that $\ga{t_K} = \fbga{t_K} = K$. Consider the first time $t' > t_K$ such that $\ga{t'} \subsetneq K$. Then for this $t'$, we have $\fbga{t'} = K$. Since $\ga{t_{i-1}'} \subsetneq K$ as it does not contain $b_i$, by the minimality of $t'$
 , it must be that $t' \leq t_{i-1}'$. Furthermore, since $\fbga{t_{i-1}'} \subsetneq K$, we must have $t' \neq t_{i-1}'$. Therefore $t' < t_{i-1}'$.

    We claim that there is a time $t$ where $t' \leq t \leq t'_{i-1}$  such that Conditions (a) (b) are both satisfied, and therefore $\theta(\ti)$ must exist. 
    
   We have just shown that $t' < t_{i-1}'$. Notice $t'$ satisfies (a). 
At the critical moment when we have the smallest $t \geq t'$ such that $\ga{t}$ is neither a superset of nor equal to $\ga{t_i}$, Item (b) will be satisfied.   Therefore, if (b) is never satisfied with any $t' \leq t \leq t_{i-1}'$, it means for all $t$ such that $t' \leq t \leq t_{i-1}'$, it must be that $\ga{t} \supset \ga{\ti}$. However, this contradicts with the fact that (b) is satisfied at time $t_{i-1}'$.   Therefore the critical moment $t$ exists. And at this time $t$, Item (b) is satisfied. 

By our definition of the fallback language, we know that at time $t-1$ and time $t$, $\fbga{t} \supset \ga{t}$ and $\fbga{t} \supset \ga{t-1}$. Since $t$ is the critical moment, we know $\ga{t-1} \supset \ga{\ti}$ if $t > t'$.  Thus $\fbga{t} \supset \ga{\ti}$. Therefore, the time $t'$ satisfies both (a) and (b).  If $t=t'$, then by the definition of $t'$, $\fbga{t'
} = K$. Therefore, $t'$ also satisfies (a). Since $t' = t$ also satisfies (b), we have that both Items (a) and (b) are satisfied. 
% This leads to a contradiction.
   \end{proof}  

   \begin{claim}\label{claim:smallPhi}
       The function $\Phi: \mB_2' \to K^2$  (here $K^2$ denote the set of  unordered pairs of strings in $K$) defined above is injective. 
   \end{claim}
   \begin{proof}
       Notice that given any $b  \in \mB_2'$, where $b$ is generated by the adversary at time $t_b$, the values $\theta(t_b)$ and $F(t_b)$ uniquely determine the image of $\Phi$. It suffices to show that the induced function $\bar \Phi: \mB_2' \to \mathbb{N} \times \mX $ where $b$ is sent to  $(\theta(t_b), F(t_b))$ is injective. 

Fix an image where $\theta(\tb) = \ths$ and $F(\tb) = \Fs$. We seek all possible times $\tb$ such that $w_{\tb} \in \mB_2'$, $\theta(\tb) = \ths$, and $F(\tb) = \Fs$.
Notice that $\theta(\tb) = \ths$ and $F(t_b) = \Fs$ uniquely determined $F'(t_b)$. Write $F'(\tb) = \Fs'$. 

 At time $t_b$, one of the following two cases will happen. Either (1) the language $\Fs'$ is no longer strictly critical; Or (2) $\Fs'$ is still strictly critical at time $\tb$.  In this case, by the minimality of $F(t_b) = \Fs$, it will imply $\Fs' \subsetneq \ga{\ths}$ or $\Fs' \subsetneq \ga{\tb}$.

 Now suppose we are in Case (1). By Claim \ref{claim:consist}, it implies that $\Fs'$ is no longer consistent at time $\tb$. 
However, $\Fs'$ is consistent at time $\ths$ since it is a strictly crictial set at time $\ths$. 
       Let $t^*$ be the first time after $\ths$ such that $\Fs'$ is not consistent. Thus $\ths < \ts \leq \tb$.

       Let $\mC(\ts)_{\leq \Fs}$ be the set of strictly critical languages at time $\ts$ which are subsets of $\Fs$. 
       \begin{claim} \label{claim:stilltb}
       All the languages in $\mC(\ts)_{\leq \Fs}$ are still consistent at time $\tb$. 
       \end{claim}
       \begin{proof} We prove by contradiction. Suppose the conclusion does not hold. 
       It means that some language  in $\mC(\ts)_{\leq \Fs}$ is not consistent at some point between $\ts$ and $\tb$.
       
       Let $t_d$ be the first time after $\ts$ such that some language in $\mC(\ts)_{\leq \Fs}$ is not consistent. Thus $\ts < t_d  \leq \tb$. Therefore the string generated by the adversary at time $t_d$, denoted by $w_{t_d}$, is not in $\ga{t_d-1}$. However, $w_{t_d} \in \ga{\tb}$. Then the language $\ga{t_d-1}$ is not a superset of $\ga{\tb}$ and thus Item (b) is satisfied at time $t_d-1$. At time $\ts$, since $\fbga{\ts} = \Fs \supset \ga{\tb}$, we have that Item (a) is satisfied. Then by the same argument as in Claim \ref{claim:theta} with $t'$ replaced with $\ts$ and $\tim$ replaced with $t_d-1$ and $K$ replaced with $\Fs$, we have that $\theta(\tb)$ should be at least $\ts$. This contradicts with the fact that $\theta(\tb) = \ths$ as $\ths < \ts \leq \tb$. 
       \end{proof}

       Claim \ref{claim:stilltb} and Claim \ref{claim:consist} imply that $\mC(\ts)_{\leq \Fs}$ is the same as $\mC(\tb)_{\leq \Fs}$. In particular, it implies that 
       $\ga{\tb}$ must be in $\mC(\ts)_{\leq \Fs}$.  

Observe that if there exist $s < s'$ such that both $\ga{s}$ and $\ga{s'}$ are strictly critical at time $s$, 
% in $\mC(\ts)_{\leq \Fs}$, 
but $\ga{s} \subsetneq \ga{s'}$, then at time $s'$, the language $\ga{s}$ is no longer consistent. This fact together with Claim \ref{claim:stilltb} implies that for all $t' < t''$ with $\ts \leq t'  < t'' \leq  \tb$, it must be $\ga{t'} \supset \ga{t''} \supset \ga{\tb}$.      
       
       We are now ready to prove that in Case (1), $\bar\Phi$ is injective. Suppose there are more than one choices of $\tb \in \mB_2'$ such that $\theta(\tb) = \ths$ and $F(\tb) = \Fs$. Denote these choices of $\tb$ as $s_1 < s_2< \dots$. By the arguments so far, they should satisfy that $\ga{s_{i+1}} \subset \ga{s_i}$ for all $i \geq 1$ and $\ga{s_i} \in \mC(\ts)_{\leq \Fs}$. Furthermore, for all $t$ such that $s_i \leq t < s_{i+1}$, it must be $\ga{s_{i+1}} \subset \ga{t} \subset \ga{s_i}$. 
       % Here, the claim that $\ga{s_{i+1}} \subsetneq \ga{t}$ is justified because if $\ga{s_{i+1}} = \ga{t}$, then $\theta(s_{i+1})$ would be at least $t$, which contradicts the fact that $\theta(s_{i+1}) = \ths < t$.

      We will now show that all but at most one of the timestamps $s_1, s_2, \dots$ satisfy $w_{s_i} \in \mB_2$, which will lead to a contradiction. Fix any $i \geq 2$. Consider the time unit immediately before $s_i$, i.e., $s_i - 1$. At this time, the algorithm generates $o_{s_i-1}$. We claim that $o_{s_i-1} \leq w_{s_i}$. This is because $w_{s_i}$ is only used by the adversary at time $s_i$, and since $w_{s_i} \notin O$, the string $w_{s_i}$ is unused by either the adversary nor the algorithm prior to time $s_i-1$. Since $\ga{s_i-1} \supset \ga{s_i}$ and $w_{s_i}$ is a string in $\ga{s_i}$, we have $w_{s_i} \in \ga{s_i-1}$. By our algorithm, the output string by the algorithm $o_{s_i-1}$ should be the smallest available one in $\mS_{s_i-2} \cup \ga{s_i-1}$, which includes $w_{s_i}$. Therefore, $o_{s_i-1} \leq w_{s_i}$. Since $w_{s_i} \notin O$, it follows that $o_{s_i-1} < w_{s_i}$. Consequently, $w_{s_i} \in \mB_g$, which contradicts with the assumption that $s_i \in \mB_2'$.

       Suppose we are in Case (2).
        Since $\ga{\ths} \subsetneq \Fs$, the situation $\Fs' \subsetneq \ga{\ths}$ cannot happen. If $\Fs' \subsetneq \ga{\tb}$, since in Case (2) we assume $\Fs'$ is strictly critical a time $\tb$, it must be that $\ga{\tb} = \Fs$. 
In this case, $F^*$ is consistently strictly critical from time $\ths$ to $\tb$. Consider all the time $t$ such that $\ths+1 \leq t < \tb$. By a similar argument as in Claim  \ref{claim:stilltb}, the language $\ga{t}$ is again a subset of $F^*$ by the fact that $F^*$ is consistently strictly critical from time $\ths$ to $\tb$. If $\ga{t}$ is a strict subset of $F^*$, then since $\ga{\tb} = F^*$, it means all the strictly critical sets at time $\tb-1$ which are strict subsets of $F^*$ are no longer consistent at time $\tb$. In particular it implies ${F^*}'$  is not consistent at time $\tb$, which contradicts with our assumption of Case (2).  Therefore, we must have that for all $t$ with $\ths+1 \leq t \leq \tb$, the language $\ga{t} = \Fs$.
       However, since $\ga{\ths}\subsetneq \Fs$, it means that again $\Fs'$ is not consistent anymore at time $\tb$, a contradiction. 
   \end{proof} 
We now define $\rho$. Given $\Phi$, any adversary input string $w_{t_b} \in \mB_2'$ at time $t_b$ is mapped to two strings $x_1, x_2$ in $K^2$. Notice that $x_1, x_2 \in \fbga{\theta(\tb)}$ and $\ga{\tb} \subset \fbga{\theta(\tb)}$. 

 Notice that $x_1, x_2 \in \fbga{\theta(\tb)}$ and $w_{\tb} \in \ga{\tb} \subset F({\tb}) \subset \fbga{\theta(\tb)}$. When we are charging $\mS_{\theta(\tb)-1}$ to obtain $\mS_{\theta(\tb)}$, we are charging it with the smallest few unused strings in $\fbga{\theta(\tb)}$. Since $x_1, x_2$ are two strings charged consecutively to $\mS_{\theta(\tb)-1}$ at time $\theta(\tb)$, and $w_{\tb} \in \fbga{\theta(\tb)}$ has not been used at time $\theta(\tb)$, we must have 
\[ x_1, x_2 \leq w_{\tb}, \text{ or } w_{\tb} \leq x_1, x_2. \]

Case 1: $x_1, x_2 \leq w_{\tb}$.    

If at time $\tb$, after $w_{\tb}$ is generated by the adversary, at least one of these two strings $x_1, x_2$ are still available in $\mS_{\tb-1}$, then the output $o_{\tb}$ should satisfy $o_{\tb} \leq \max(x_1, x_2) \leq w_{\tb}$. This is because by our algorithm, we always prioritize outputting the smallest unused elements in $\mS_{\tb-1} \cup \ga{\tb}$. This also contradicts with the fact that $w_{\tb} \notin \mB_g$. 
Therefore at time $\tb$, none of the two strings $x_1, x_2$ are available. So these two strings have been used before time $\tb$ or used as $w_{\tb}$.

Let $r_1$ be the first time when $x_1$ is used by the adversary or the algorithm. It means either the adversary generates $x_1$ at time $r_1$ or the algorithm generates $x_1$ at time $r_1$. Define $r_2$ analogously. Since we have assumed none of the two strings $x_1, x_2$ are available after time $\tb$, we have  $r_1, r_2 \leq \tb$. Without loss of generality, assume $r_1 \leq r_2$. We can assume this because we never specify the order between $x_1, x_2$. 

If $r_1 = r_2$, it means at time $r_1 = r_2$, the strings $x_1, x_2$ are picked by the adversary and the algorithm up to exchange. Thus $r_1=r_2<\tb$. Then let $\rho(w_{t_b}) = o_{r_1}$. We thus have $\rho(w_{t_b}) \leq \max(x_1, x_2) \leq w_{t_b}$. Since $w_{t_b} \notin O$, we have $\rho(w_{t_b}) < w_{t_b}$. 

Now we assume $r_1 < r_2$. If $x_1$ is generated by the algorithm, then let $\rho(b) = x_1 = o_{r_1}$. If $x_1$ is generated by the adversary at time $r_1$, as $r_1 < r_2$, then the output string at time $r_1$, denoted as $o_{r_1}$, should satisfy that $o_{r_1} \leq x_2$. If this is not the case, at time $r_1$, the algorithm should have output a string that is at most $x_2$, since $x_2$ is in the set $\mS_{r_1-1}$ and also not generated by the adversary at time $r_1$. However, $o_{r_1} \neq x_2$ as this would imply $r_1 = r_2$. Therefore, since $x_2 \leq w_{\tb}$, we have that $o_{r_1} \leq w_{\tb}$. Together with the fact that $w_{t_b} \notin O$, we have $o_{r-1} < w_{t_b}$.   Define $\rho(w_{t_b})  = o_{r_1}$. 
In all the cases, we have 
$\rho(w_{t_b})  = o_{r_1} < w_{t_b}$. 

Case 2: $w_{\tb} \leq x_1, x_2$. 

% Recall that $w_{\tb} \notin \mS_{\theta(\tb)-1}$. 
When we are charging $\mS_{\theta(\tb)-1}$ to obtain $\mS_{\theta(\tb)}$, we are charging it with the smallest few unused strings in $\fbga{\theta(\tb)}$, which potentially include $w_{\tb}, x_1, x_2$.  In the case when $w_{\tb} \leq x_1, x_2$, as $x_1, x_2$ are added to $\mS_{\theta(\tb)-1}$ at time $\theta(\tb)$, $w_{\tb}$ is also added to $\mS_{\theta(\tb)-1}$ at time $\theta(\tb)$ if not earlier. At time $\tb-1$, the string $w_{\tb}$ is still available in $\mS_{\tb-2}$, since $w_{\tb}$ is only used by the adversary at time $\tb$ and never used by the algorithm.  Thus the algorithm output at time $\tb-1$, denoted $o_{\tb-1}$, should satisfy that $o_{\tb-1} < w_{\tb}$. This contradicts with the fact that $w_{\tb} \notin \mB_g$. 

Therefore we have the desired function $\rho$ and  $\rho$ is injective. 
% $|\rho^{-1}(o)| \leq 2$ for each $o \in O$. 
\end{proof}
\end{proof}

\section{Concluding Remarks and Open Questions}

In this paper, we have used measures of density to 
quantify the {\em breadth} of an algorithm for
language generation in the limit:
roughly speaking, an algorithm for this problem has high
breadth if its outputs have high upper or lower density in the
true language $\trueL$.
We study this question for two notions of generation in the limit:
{\em element-based}, in which the goal is to output strings $\out_t$
that all belong to the true language $\trueL$ after some finite time
$\fint$, and 
{\em index-based}, in which the goal is to produce internal hypotheses
$\lang{i_t}$ that should all be subsets of $\trueL$ after some finite time
$\fint$.
Prior algorithms for this problem achieved both guarantees, but
with a set of outputs that could have zero upper density in $\trueL$.
It was therefore natural to ask whether this was a necessary trade-off:
whether any algorithm achieving generation in the limit must have some
instances where its outputs have zero density.

We show that this trade-off is not necessary for element-based
generation, providing a constant $c > 0$ and
an algorithm whose set of outputs has 
density at least $c$ in $\trueL$ for all instances.
For index-based generation, the situation is more complicated:
we show that there are instances where any algorithm solving this problem must have a subsequence
of its languages $\lang{i_t}$ whose upper densities converge to 0,
but we also give an algorithm for which $\lang{i_t} = \trueL$ infinitely
often. 
That is, this algorithm exhibits an intrinsically bimodal type of
behavior, and unavoidably so:
it always has a subsequence
of its languages $\lang{i_t}$ whose lower densities are 1, even though
 it must have another subsequence of its languages $\lang{i_t}$ 
whose upper densities converge to 0.

There are a number of interesting questions left open by this work.
First, we have not optimized the constant $c > 0$ in our algorithm
that achieves an output set of lower density at least $c$ in every instance.
Our analysis shows that $c = 1/8$ is achievable, but we find it likely
that the same algorithm guarantees a larger lower density;
the best possible result would be a lower density of at least $c$ for
every $c < 1/2$, as we have for our guarantee on upper density.

It would also be interesting to see if there are any additional 
guarantees that could be established if we relax the notion of
generation in the limit to allow for failures of validity that
can occur infinitely often but increasingly rarely.
In particular, suppose that $Z$ is the set of time steps $t$ for 
which the output string $\out_t$ does not belong to $\trueL$.
Language generation in the limit requires that $Z$ be finite.
If we relax this restriction to
require only that $Z$ have zero upper density in the set
of all time steps $\{1, 2, 3, 4, \ldots\}$, can we achieve
any stronger guarantees from a language generation algorithm?

Finally, we could consider connecting the notion of density
to some of the other recent models for breadth explored in
Charikar and Pabbaraju \cite{charikar-pabbarju} and
Kalavasis, Mehrotra, and Velagkas \cite{kalavasis-stoc25}.
As noted earlier in the paper, these other models show negative
results for stronger breadth requirements in models different from ours,
but our definitions of upper and lower density can be adapted to their
settings, and it would be interesting what types of density
guarantees can be achieved in their models.
In a different direction, the work of 
Li, Raman, and Tewari \cite{li-generation-pac} cast some of
the core definitions for language generation in the limit
in a learning-theoretic framework, deriving interesting new
results in this formalism, and it would be interesting to see
whether our notions of density can be adapted to their setting as well.

\end{document}